\documentclass[a4paper, 10pt, reqno]{amsart}

\usepackage{amssymb, amsmath, amsthm}

\usepackage{graphicx, psfrag, setspace, subfigure, gensymb}
\usepackage{amsfonts}
\usepackage{bbm}
\usepackage{fix-cm}
\usepackage{comment}
\usepackage{tikz,tikz-cd}
\usetikzlibrary{matrix,arrows,decorations.pathmorphing,decorations.pathreplacing}
\tikzset{commutative diagrams/diagrams={baseline=-2.5pt},commutative diagrams/arrow style=tikz}
\usepackage[colorlinks]{hyperref}
\usepackage{float}
\usepackage{setspace} 
\usepackage[utf8]{inputenc}

\newcommand\A{\mathbb A}
\newcommand\Z{\mathbb Z}
\newcommand\C{\mathbb C}

\newcommand\N{\mathbb N}

\newcommand{\bbS}{\mathbb{S}}

\newcommand{\cA}{\mathcal{A}}
\newcommand{\cB}{\mathcal{B}}
\newcommand{\cC}{\mathcal{C}}

\newcommand{\cE}{\mathcal{E}}
\newcommand{\cF}{\mathcal F}
\newcommand{\cG}{\mathcal{G}}
\newcommand{\cH}{\mathcal{H}}

\newcommand{\cK}{\mathcal{K}}
\newcommand{\cL}{\mathcal{L}}

\newcommand{\cO}{\mathcal{O}}
\newcommand{\cP}{\mathcal{P}}

\newcommand{\cS}{\mathcal{S}}
\newcommand{\cT}{\mathcal T}
\newcommand{\cU}{\mathcal{U}}

\newcommand{\cW}{\mathcal{W}}
\newcommand{\cX}{\mathcal{X}}
\newcommand{\cY}{\mathcal{Y}}
\newcommand{\cZ}{\mathcal{Z}}

\newcommand{\set}[1]{\left\{{#1}\right\}}

\newcommand\onehalf{\tfrac12}

\newcommand\isoto{\stackrel{\sim}{\To}}
\newcommand\id{\mathrm 1}

\newcommand\into{\hookrightarrow}
\newcommand\onto{\twoheadrightarrow}
\newcommand\To{\longrightarrow}

\newcommand\Coh{\operatorname{Coh}}
\newcommand\Hom{\operatorname{Hom}}
\newcommand\REnd{\operatorname{REnd}}
\newcommand\End{\operatorname{End}}

\renewcommand\hom{\mathcal{H}om }
\newcommand\RHom{\operatorname{RHom}}
\newcommand\Ext{\operatorname{Ext}}

\renewcommand{\P}{\mathbb{P}}
\newcommand\Gr{\operatorname{Gr}}
\newcommand{\Pf}{\mathrm{Pf}}
\newcommand\pt{\operatorname{pt}}

\newcommand{\Sym}{\operatorname{Sym}}
\newcommand{\Spec}{\operatorname{Spec}}

\newcommand{\Crit}{\operatorname{Crit}}

\newcommand{\Wedge}{\mbox{\scalebox{1.2}{$\wedge$}}}
\newcommand{\GL}{\operatorname{GL}}
\newcommand{\cok}{\operatorname{cok}}
\newcommand{\sm}{\mathrm{sm}}

\newcommand{\al}[1]{\begin{align*}#1\end{align*}}

\newcommand{\beq}[1]{\begin{equation}\label{#1} }
  \newcommand{\eeq}{\end{equation}}
\newcommand{\pgap}{\vspace{5pt}}

\newtheorem{prop}[equation]{Proposition}
\newtheorem{thm}[equation]{Theorem}
\newtheorem{lem}[equation]{Lemma}
\newtheorem{defn}[equation]{Definition}

\newtheorem{cor}[equation]{Corollary}

\theoremstyle{remark}
\newtheorem{rem}[equation]{Remark}

\makeatletter \@addtoreset{equation}{section} \makeatother

\let\oldtocsection=\tocsection
\let\oldtocsubsection=\tocsubsection
\let\oldtocsubsubsection=\tocsubsubsection
\renewcommand{\tocsection}[3]{\hspace{0em}\oldtocsection{#1}{#2}{#3}}
\renewcommand{\tocsubsection}[3]{ \hspace{1em} \oldtocsubsection{#1}{\small{#2}}{\small{#3}} }
\renewcommand{\tocsubsubsection}[3]{\hspace{2em}\oldtocsubsubsection{#1}{\small{#2}}{\small{#3}}}

\newcommand{\marginparstretch}{0.6}
\let\oldmarginpar\marginpar
\renewcommand\marginpar[1]{\-\oldmarginpar[\framebox{\setstretch{\marginparstretch}\begin{minipage}{\marginparwidth}{\raggedleft\scriptsize #1}\end{minipage}}]{\framebox{\setstretch{\marginparstretch}\begin{minipage}{\marginparwidth}{\raggedright\scriptsize #1}\end{minipage}}}}

\makeatletter
\def\subsubsection{\@startsection{subsubsection}{3}%
  \z@{.5\linespacing\@plus.7\linespacing}{-.5em}%
  {\normalfont\bfseries}}
\makeatother

\newcommand{\qquotes}[1]{``{#1}''}

\newcommand{\aand}{\quad\quad\mbox{and}\quad\quad}

\newcommand{\GSp}{\mathrm{GSp}}
\newcommand{\Sp}{\mathrm{Sp}}
\newcommand{\supp}{\operatorname{supp}}
\newcommand{\codim}{\operatorname{codim}}
\newcommand{\Br}{\mathsf{DB}}

\newcommand{\Schur}[1]{\bbS^{#1}}
\newcommand{\Schpur}[1]{\bbS^{\langle#1\rangle}}
\newcommand{\pr}{\prime}

\renewcommand{\ss}{\mathrm{ss}}
\DeclareMathOperator{\depth}{depth}
\newcommand{\Stack}[2]{\big[#1\,/\,#2\big]}
\renewcommand{\mod}{\mbox{-mod}}
\newcommand{\rmod}{\mbox{mod-}}
\newcommand{\tms}{\!\times\!}
\newcommand{\dualW}{W'}
\newcommand{\pd}{\mathrm{pd}}

\newcommand{\proj}{\Omega}

\newcommand{\Vect}{\mathrm{Vect}}

\newcommand\OS{O}

\newtheorem*{thm:SPHoriDuality}{Theorem \ref{thm:SPHoriDuality}}
\newtheorem*{thm:HPDualityAlgebrasConcrete}
{Theorem \ref{thm:HPDualityAlgebrasConcrete}}
\newtheorem*{thm:ProjectiveWindows}{Proposition \ref{thm:ProjectiveWindows}}

\title{Hori-mological projective duality}
\author{J\o rgen Vold Rennemo and Ed Segal}
\begin{document}

\begin{abstract}
  Kuznetsov has conjectured that Pfaffian varieties should admit non-commutative crepant resolutions which satisfy his Homological Projective Duality. We prove half the cases of this conjecture, by interpreting and proving a duality of non-abelian gauged linear sigma models proposed by Hori.
\end{abstract}

\maketitle

\tableofcontents

\section{Introduction}

Let $V$ be a vector space of odd dimension $v$. For any even number $0\leq 2q<v$, we have a Pfaffian variety
$$\Pf_q \; \subset \P(\Wedge^2 V^\vee) $$
consisting of all 2-forms on $V$ whose rank is at most $2q$. This variety is not a complete intersection, and is usually highly singular -- the singularities occur where the rank drops below $2q$. We only get smooth varieties in the cases $q=1$, which gives the Grassmannian $\Gr(V, 2)$, and $q=\onehalf(v-1)$, which gives the whole of $\P(\Wedge^2 V^\vee)$.  

The projective dual of $\Pf_s$ is another Pfaffian variety; it's the locus
$$\Pf_s\; \subset \P(\Wedge^2 V) $$
consisting of bivectors of rank at most $2s$, where $2s = v-1-2q$.
\pgap

This paper achieves two closely-connected goals. The first is to
establish that \emph{Homological Projective Duality} (HPD) holds for this pair of varieties. This is a conceptual framework due to Kuznetsov \cite{kuznetsov_homological_2007}, for understanding how we should compare the derived categories $D^b(X)$ and $D^b(Y)$ for a pair of projectively-dual varieties $X$ and $Y$.  The idea is that we should pick a generic linear subspace $L$ and then look at the derived category of the slice $X\cap \P L$ and of the dual slice $Y\cap \P L^\perp$; then the \qquotes{interesting part} of these categories will be equivalent. Often there is a critical value of $\dim L$ such that both slices are Calabi--Yau, and they are derived equivalent. 

A complicating factor here is that Pfaffian varieties are singular, and it seems that it is not sensible to try to apply HPD to singular varieties.  Instead Kuznetsov suggests we replace both of them with non-commutative crepant resolutions. A non-commutative resolution of a variety $X$ is a sheaf of non-commutative algebras $A$ on $X$ which has an appropriate smoothness property, and is Morita-equivalent to $\cO_X$ over some open subset. Then instead of working with $\cO_X$-modules we work with $A$-modules, and obtain a category that behaves a lot like the derived category of a geometric resolution.  \v{S}penko and Van den Bergh \cite{spenko_non-commutative_2015} have constructed non-commutative crepant resolutions for Pfaffian varieties, we denote the resolution of $\Pf_s$ by $A$ and the resolution of $\Pf_q$ by $B$.  

We prove that the non-commutative varieties $(\Pf_s, A)$ and $(\Pf_q, B)$ are HP dual to each other.  For example, in the Calabi--Yau case we have the following result (a special case of Theorem \ref{thm:HPDualityAlgebrasConcrete} and Proposition \ref{thm:SerreOnTheProjectiveNCR}):



\begin{thm}\label{thm:CYcase} Let $L\subset \Wedge^2 V^\vee$ be a generic subspace of dimension $sv$. Then the sheaves $A$ and $B$ restrict to give  non-commutative crepant resolutions of the varieties $\Pf_s\cap \P L^\perp$ and  $\Pf_q\cap \P L$, the categories
  $$ D^b(\Pf_s \cap \P L^\perp,\; A|_{\P L^\perp}) \aand D^b(\Pf_q\cap \P L, \;B|_{\P L}) $$
  are both Calabi--Yau of dimension $2qs - 1$, and they are equivalent.
\end{thm}

Our results build on Kuznetsov's own pioneering work \cite{kuznetzov_lines_2006}, where he proved that for $v\leq 7$ the Grassmannian $\Gr(V, 2)$ is HP dual to a non-commutative resolution of $\Pf_{v-3}$. In the same work, he conjectures that this could be made to work for all Pfaffians, if one could find the correct non-commutative resolutions. Hence we have confirmed Kuznetsov's conjecture (in the case where $\dim V$ is odd; the even case is discussed more below).

\pgap

Our second goal is to interpret and prove a duality in quantum field theory proposed by Hori \cite{hori_duality_2013}. This duality relates certain gauged linear sigma models (GLSMs). These are gauge theories in two dimensions with $N=(2,2)$ supersymmetry, and the models in question have symplectic gauge groups. Using various physical arguments, Hori proposes that each such model has a dual model producing an equivalent theory.  These dual GLSMs are closely related to projectively-dual Pfaffian varieties and it is clear that there is a connection to HPD; in fact Hori states that this connection was one of the motivations that led to his proposal.

In this paper we give a rigorous formulation of this duality at the level of B-branes. Each GLSM should have an associated category of B-branes, with a purely algebro-geometric construction, and then the duality predicts that dual models will produce equivalent categories. We propose a definition of these categories, and prove that with our definition the predicted equivalence does indeed hold. We then use this result to deduce our HPD statement, using some ideas from the physics together with other mathematical ingredients.

Our proposal for the category of B-branes appears to be new -- it is not present in Hori's paper -- so we hope that this aspect is a useful contribution to the physics literature.
\pgap

As mentioned above, another important input for this paper is the work of \v{S}penko and Van den Bergh \cite{spenko_non-commutative_2015}. They describe a very general procedure for constructing non-commutative (crepant) resolutions for quotient singularities, and our proposed category of B-branes turns out to be an example of their procedure. Hence their results imply that we obtain non-commutative resolutions of the Pfaffian varieties. 

As well as these external sources, this paper is  a continuation of previous work of the authors and their coauthors \cite{ADS_pfaffian_2015, ST_2014, rennemo_homological_2015}. 
Indeed, once we have defined the correct non-commutative resolution and the functor relating the two sides, (both of which were only available for $s,q \le 4$ before), the proof that HPD holds over the smooth locus is a simple extension of these previous works. However, extending this to a full HPD statement valid over the singular loci requires completely new arguments.
\pgap

Up to this point we have assumed that $v=\dim V$ is odd, but one can also consider Pfaffians for even-dimensional vector spaces. These too come in projectively-dual pairs $\Pf_s$ and $\Pf_q$, where now $2s=v-2q$, and Kuznetsov conjectures that HPD holds in this case too. However it seems that there is no physical duality in this case, because the field theories defined by the GLSMs are not regular if $v$ is even. Despite this, we can prove some partial results. For ease of exposition we keep $v$ odd for most of the paper, and in Section \ref{sec:EvenCase} we explain what is different about the even case. 
\pgap

For a more significant variation one can swap two-forms $\Wedge^2 V^\vee$ for quadratic forms $\Sym^2 V^\vee$. In the GLSMs this corresponds to replacing the symplectic gauge groups with orthogonal groups, and Hori predicts a similar duality. 

We believe strongly that HPD can be proved in this situation using essentially the methods of this paper, and we hope to present such a proof in future work. This would  generalise work of Kuznetsov \cite{kuznetsov_derived_2008} corresponding to quadratic forms of rank $1$, and of Hosono--Takagi \cite{hosono_duality_2013} and the first author \cite{rennemo_homological_2015} for forms of rank $\le 2$.

\pgap

The remainder of the introduction discusses the constructions used in this paper and sketches the main ideas of the proofs. 

\subsubsection*{Conventions}
For an algebra $A$, we write $D^b(A)$ for the derived category whose objects are complexes of left  $A$-modules with finitely generated cohomology.
If $A$ is graded, then $D^b(A)$ means the derived category of graded modules.

If $(X,A)$ is a variety with a coherent sheaf of algebras on it, we write $D^b(X,A)$ for the derived category whose objects are complexes of left  $A$-modules with coherent cohomology.

If $\cE$ is a chain-complex in some abelian category, we write $h_{\bullet}(\cE)$ for its homology object.

\subsubsection*{Acknowledgements}
It's a pleasure to thank Michel Van den Bergh for his patient explanations of \cite{spenko_non-commutative_2015}, and Kentaro Hori, \v{S}pela \v{S}penko and Michael Wemyss for illuminating discussions. E.S. also thanks Nick Addington, Will Donovan, and Richard Thomas -- this paper draws heavily on his previous joint work with them -- and  Dan Halpern-Leistner for some helpful conversations. Finally we would like to thank the anonymous referees for their careful readings of the paper.
\pgap

This project has received funding from the European Research Council (ERC) under the European Union Horizon 2020 research and innovation programme (grant agreement No.725010).

\subsection{The non-commutative resolutions}\label{sec:introNCRs}

In this section we briefly describe the non-commutative resolutions, see Sections \ref{sec:statement} and \ref{sec:Br(X)} for more details. Let $V$ be an odd-dimensional vector space as before, and let $Q$ be a symplectic vector space of dimension $2q$. Let $\widetilde{\cY}$ be the stack:
$$\widetilde{\cY} = \Stack{\Hom(V, Q)}{\Sp(Q)} $$
This stack maps to $\Wedge^2 V^\vee$ by pulling-back the symplectic form, and its image is the cone 
$$\widetilde{\Pf_q} \subset \Wedge^2 V^\vee$$
over the Pfaffian variety $\Pf_q$. 

Since we're ultimately interested in projective varieties we need to also quotient by rescaling. Let $\GSp(Q)$ denote the `symplectic similitude' group of $Q$, the subgroup of $\GL(Q)$ that preserves the symplectic form up to scale. It's a semi-direct product:
$$\GSp(Q) = \Sp(Q)\rtimes \C^* $$
Then we let:
$$\cY = \Stack{\Hom(V, Q)}{\GSp(Q)} $$
This stack is a quotient of $\widetilde{\cY}$ by an additional $\C^*$, and it maps to the stack $[\Wedge^2 V^\vee \,/\, \C^*]$ with image $[\widetilde{\Pf_q} \,/ \,\C^*]$. The pre-image of the origin is the locus where $V$ maps to an isotropic subspace in $Q$, and if we delete this locus then we get an open substack
$$\cY^{ss} \subset \cY$$
whose underlying scheme is the projective variety $\Pf_q$. Our notation here reflects the fact that $\cY^{ss}$ is the semi-stable locus for the obvious GIT stability condition on $\cY$, but note that if $q>1$ then it is still an Artin stack.

In some sense the stack $\cY^{ss}$ is a resolution of $\Pf_q$, but the category $D^b(\cY^{ss})$ is very large; for example its Hochschild homology is infinite-dimensional. To get a more finite category, more akin to the derived category of an honest geometric resolution, we pick out a subcategory $\Br(\cY^{ss}) \subset D^b(\cY^{ss})$. The notation here refers to B-branes, as we'll explain in the next section. To define this subcategory we start by defining a subcategory of $D^b(\widetilde{\cY})$, as follows.

Recall that irreducible representations of $\Sp(Q)$ are indexed by Young diagrams of height at most $q=\onehalf \dim(Q)$. For any such Young diagram $\delta$ there is a corresponding vector bundle on $\widetilde{\cY}$, associated to that irrep, and we'll denote this vector bundle by:
$$\Schpur{\delta} Q $$
Here the operation $\Schpur{\delta}$ is a `symplectic Schur functor' \cite[Section 17.3]{fulton_representation_1991}, we'll use the notation
$\Schur{\delta}$ for ordinary ($\GL$) Schur functors.

For any $a,b\in \N$, let's write $Y_{a,b}$ for the set of Young diagrams of height at most $a$ and width at most $b$.  To define the category $\Br(\widetilde{\cY})$ we consider the set $Y_{q,s}$,
where 
$$ s = \onehalf(v - 1) - q $$
as in the classical projective duality discussed above.  There is a corresponding set of vector bundles on $\widetilde{\cY}$, and we define $\Br(\widetilde{\cY})$ to be the subcategory generated by these vector bundles,  that is:
$$\Br(\widetilde{\cY}) = \langle\, \Schpur{\delta}Q,\; \gamma\in Y_{q,s}\,\rangle\quad\subset D^b(\widetilde{\cY})$$
Our original motivations for this definition were the analogy with Kapranov's exceptional collections on Grassmannians \cite{kapranov_derived_1984}, and Hori's calculation of the Witten index of the associated GLSM (see the next section). However, a much more compelling reason to consider it is found in the recent work of \v{S}penko and Van den Bergh \cite{spenko_non-commutative_2015}. The category $\Br(\widetilde{\cY})$ has a tilting bundle
$$\widetilde{T}  = \bigoplus_{\delta\in Y_{q,s}}   \Schpur{\delta}Q $$
essentially by definition, and we can consider the non-commutative algebra:
$$\widetilde{B} = \End_{\widetilde{\cY}}(\widetilde{T}) $$
There is an evident map from the commutative ring $\cO_{\widetilde{\Pf_q}}$ to the centre of $\widetilde{B}$, since $\cO_{\widetilde{\Pf_q}}$ is the ring of invariant functions on $\widetilde{\cY}$. 

\v{S}penko and Van den Bergh prove that $\widetilde{B}$ is a non-commutative crepant resolution of the singularity $\widetilde{\Pf_q}$. In particular, we have an equivalence
$$ \Hom(\widetilde{T}, - ) : \Br(\widetilde{\cY}) \isoto  D^b(\widetilde{B} \mod) $$
and $\Br(\widetilde{\cY})$ has all the properties of the derived category of a geometric crepant resolution. For example, it is a `non-compact Calabi--Yau' in the following sense: given objects $E, F \in \Br(\widetilde{\cY})$ with $F$ supported over $0 \in \widetilde{\Pf}_q$, we have
$$\Hom^\bullet(E,F) \cong \Hom^{\dim  \widetilde{\Pf}_q -\bullet}(F,E)^\vee,$$
by \cite[Lemma 6.4.1]{van_den_bergh_non-commutative_2004}. In fact \v{S}penko--Van den Bergh give a very general construction that produces a non-commutative resolution for any quotient of a smooth affine variety by a reductive group, but they cover $\widetilde{\Pf_q}$ as an explicit example, and show that in this case the resolution can be chosen to be crepant \cite[Section 6]{spenko_non-commutative_2015}.

Now we add in our additional $\C^*$-action, and define a subcategory
$$\Br(\cY) \subset D^b(\cY) $$
to be the pre-image of $\Br(\widetilde{\cY})$ under pull-back along the map $\widetilde{\cY}\to\cY$. To understand this subcategory observe that a representation of $\GSp(Q)$ defines a vector bundle on $\cY$, and pulling-back to $\widetilde{\cY}$ corresponds to restricting the representation to $\Sp(Q)$. For each $\delta\in Y_{q,s}$ there are $\Z$-many irreps of $\GSp(Q)$ that restrict to give $\Schpur{\delta} Q$, so the set $Y_{q,s}$ specifies an infinite set of vector bundles of $\cY$ and they generate $\Br(\cY)$. The subcategory can also be defined by a `grade-restriction-rule' at the origin (see Section \ref{sec:statement}).

We can choose a vector bundle  $T$ on $\cY$ that pulls-back to $\widetilde{T}$, then by taking endomorphisms of $T$ we can form a graded algebra $B$ whose underlying ungraded algebra is $\widetilde{B}$. Then $\Br(\cY)$ is equivalent to $D^b(B\mod)$. 

Now delete the origin, \emph{i.e.} restrict to the open substack $\cY^{ss}$ and the projective variety $\Pf_q$.  The subcategory $\Br(\cY)$ restricts to a subcategory
$$\Br(\cY^{ss}) \subset D^b(\cY^{ss})$$
(generated by the restrictions of the vector bundles that generate $\Br(\cY)$), the graded algebra $B$ restricts to give a sheaf  of non-commutative algebras on $\Pf_q$, and we have an equivalence:
$$D^b(\Pf_q, B) \; \cong \; \Br(\cY^{ss})$$
This is the non-commutative crepant resolution of $\Pf_q$. 
\pgap

As mentioned above, for HPD we are interested in linear slices of $\Pf_q$, \emph{i.e.} considering  $\Pf_q\cap \P L$  for a subspace $L\subset \Wedge^2 V^\vee$. We can restrict $B$ to $\Pf_q\cap \P L$ and get a sheaf of algebras $B|_{\P L}$; for generic $L$ this will yield a non-commutative crepant resolution of $\Pf_q\cap \P L$ by the `non-commutative  Bertini theorem' \cite{RSVdB}.

\pgap

We can of course do everything in a precisely analogous way for the  projectively-dual Pfaffian $\Pf_s \subset \P(\Wedge^2 V) $. We fix a symplectic vector space $S$ of dimension $2s$, and define the stacks:
$$\widetilde{\cX} = \Stack{\Hom(S, V)}{\Sp(S)}\aand \cX = \Stack{\Hom(S, V)}{\GSp(S)}$$
To define a subcategory $\Br(\widetilde{\cX})$ we just `rotate our rectangle', and consider the set of Young diagrams $Y_{s,q}$. These all correspond to irreps of $\Sp(S)$, and we define:
$$\Br(\widetilde{\cX}) = \langle\, \Schpur{\gamma}Q,\; \gamma\in Y_{s,q}\,\rangle\quad\subset D^b(\widetilde{\cX})$$
Now we proceed through the same steps: we have an algebra $\widetilde{A}$ which is a non-commutative crepant resolution of $\widetilde{\Pf_s}$, and a graded algebra $A$ which restricts to give a sheaf of algebras over $\Pf_s$.
\pgap

It should already be evident that at the most crude level our duality comes down to a bijection between the sets $Y_{q,s}$ and $Y_{s,q}$. However even at this level one must take care to choose the right bijection!  For any $\gamma\in Y_{s,q}$, let us denote by $\gamma^c$ the Young diagram in $Y_{s,q}$ obtained by taking the complement of $\gamma$ in a rectangle of height $s$ and width $q$ (and then rotating $180^\circ$).  The relevant bijection for us is the function:
\begin{align}\label{eq:combinatorialduality}\nonumber  Y_{s,q} &\isoto Y_{q,s} \\ \gamma &\mapsto (\gamma^c)^\top \end{align}
Unsurprisingly, it will take quite a lot of work to lift this to an actual comparison of any categories.

\subsection{A sketch of the physics}\label{sec:introPhysics}

For a string-theorist, a stack like $\cY$ or $\widetilde{\cY}$ can be thought of as the input data for a non-abelian GLSM. This kind of GLSM was analysed in some detail in \cite{hori_duality_2013} (see also \cite{hori_linear_2013}) and a duality was proposed, as discussed above. In this section we'll give a very rough summary of this proposed duality and its connection to HPD for Pfaffians.

Let's begin with the stack $\cY$, which corresponds to a theory with gauge group $\GSp(Q)$. This theory has a Fayet--Iliopoulos parameter $r$, which roughly corresponds to the value of the moment map or GIT stability condition, and for $r\gg 0$ the classical space of vacua is the GIT quotient $Y^{ss}/\GSp(Q) = \Pf_q$. One might expect that the theory reduces, in some limit, to the $\sigma$-model with target $\Pf_q$. This can only be approximately correct because $\Pf_q$ is singular, but one might hope that quantum corrections somehow resolve the singularities. 

To get a $\sigma$-model on a slice $\Pf_q\cap \P L$ we perform a standard trick that goes back to Witten. Write $L^\perp\subset \Wedge^2 V$ for the annihilator of $L$, and define an action of  $\GSp(Q)$  on $L^\perp$ by setting the subgroup $\Sp(Q)$ to act trivially, and the residual $\C^*$ to act diagonally with weight $-1$. Then we add this into our stack/GLSM data, forming:
$$\cY\tms_{\C^*} L^\perp = \Stack{\Hom(V, Q)\tms L^\perp}{\GSp(Q)} $$
This stack has a canonical invariant function on it, namely
\beq{eq:W}W(y, a) = \omega_Q(\wedge^2 y (a)) \eeq
where $y\in \Hom(V,Q)$ and $a\in L^\perp$ and $\omega_Q\in \Wedge^2 Q^\vee$ is the symplectic form. We add $W$ to our GLSM as a superpotential.

If we didn't add $W$, then in the $r\gg 0$ phase we might expect to get the $\sigma$-model on the corresponding GIT quotient, which is the total space of the vector bundle $L^\perp(-1)$ over $\Pf_q$. The presence of $W$ localizes the theory onto the critical locus $\Crit(W)$. After quotienting by $\GSp(Q)$ the superpotential becomes quadratic, and $\Crit(W)$ is the subvariety:
$$\set{ a = 0, \; y\circ a = 0 } = \Pf_q \cap \P L $$
In fact this is only true over the smooth locus in $\Pf_q$, really $\Crit(W)$ has some non-compact branches over the singular locus. But again we might dream that quantum corrections will somehow solve this problem.

The addition of $L^\perp$ has another important consequence. In the previous model if we set $r\ll 0$ then the (classical) vacuum space is empty, but in the new model this is not true and we have a second interesting phase; however,  this second phase is more difficult to analyze. In GIT terms, only the locus $\set{a=0}$ is unstable, and we can think of this phase as family of models living over $\P L^\perp$. Each fibre is the stack/GLSM $\widetilde{\cY}$, equipped with a superpotential $W_a$ which varies with the point $[a]\in \P L^\perp$.  To connect to  (homological) projective duality, we need to show that these fibre-wise models are very simple: they give no contribution at all unless $[a]$ lies in the dual Pfaffian $\Pf_s$, and when $[a]$ does lie in $\Pf_s$ they look like a $\sigma$-model on a point.

In earlier work \cite{hori_aspects_2007} Hori and Tong directly analyse these fibre-wise models in the case $q=1$ and give some arguments that this desired conclusion holds;  the paper \cite{ADS_pfaffian_2015} was a mathematical treatment of this work. However in \cite{hori_duality_2013} Hori gives a much cleaner approach. Based on various physical arguments (that we do not understand well enough to summarize), he proposes that $\widetilde{\cY}$ has a dual description as the model:
$$\widetilde{\cX}\tms \Wedge^2 V^\vee  = \Stack{\Hom(S, V)}{\Sp(S)} \tms \Wedge^2 V^\vee $$
Here, as in the previous section, $S$ is a symplectic vector space of dimension $2s=v-2q-1$. This dual model comes equipped with a tautological superpotential
\beq{eq:W'}\dualW(x, b) = b(\wedge^2 x (\beta_S)) \eeq
where $x\in \Hom(S, V)$ and $b\in \Wedge^2 V$ and $\beta_S\in \Wedge^2 S$ is the Poisson bivector. We will explain shortly how we interpret this duality as predicting an equivalence of categories, but let us first fill in the final steps connecting it to projective duality.

As just stated this duality looks very asymmetric, since the dual side has `extra directions'  $\Wedge^2 V^\vee$ and a superpotential. To correct this asymmetry we choose a subspace $L\subset \Wedge^2 V^\vee$, and cross both sides with $L^\perp$. On the original side this gives the model $\widetilde{\cY}\tms L^\perp$, and to this we can add a superpotential $W$ as in \eqref{eq:W}. On the dual side, we get:
$$  \widetilde{\cX}\tms \Wedge^2 V^\vee  \tms L^\perp $$
Under the duality, the variable $\omega_Q(\wedge^2 y)\in \Wedge^2 V^\vee$ corresponds to $b$. Hence adding $W$ on the original side corresponds, on the dual side, to adding the term $b(a)$ to the existing superpotential $\dualW$. This term is quadratic, so we may integrate out its non-degenerate part, which means deleting the directions $(L^\perp)^\vee\tms L^\perp$.  What remains is the model $\cX\tms L$. So a more symmetric way to state the duality is that it exchanges
$$\widetilde{\cY}\tms L^\perp \quad \leftrightarrow \quad \widetilde{\cX} \tms L $$
with their tautological superpotentials. Now we simply add the additional $\C^*$ action, promoting our gauge groups to $\GSp(Q)$ and $\GSp(S)$.  This introduces an FI parameter, and the duality exchanges the `easy' ($r\gg 0$) and `difficult' ($r\ll 0$) phases of the two sides. We have argued that the original model reduces in the easy phase to a $\sigma$-model on the target $\Pf_q\cap \P L$, or some kind of resolution thereof, so when we pass to the difficult phase and apply the duality we  must reduce to a $\sigma$-model with the target $\Pf_s\cap \P L^\perp$. 
\pgap

Now we discuss the implications of this story for B-branes. A $\sigma$-model on a smooth variety $X$ should have an associated category of B-branes, and everyone knows that this is the derived category $D^b(X)$. A GLSM should also have an associated category of B-branes, but this is much less well understood, particularly if the model has FI parameters. The GLSM $\widetilde{\cY}$ has no FI parameters because the symplectic group is simple, so one can predict with some confidence that there should be a single associated category. Our proposal is that the category of B-branes in this GLSM is the subcategory
$$\Br(\widetilde{\cY}) \subset D^b(\widetilde{\cY}) $$
defined in the previous section. Our evidence for this is:
\begin{enumerate}
\item $\Br(\widetilde{\cY})$ is a non-commutative crepant resolution, so it behaves like a non-compact Calabi--Yau variety. This would seem to be a desirable property for the B-brane category. Furthermore, Van den Bergh conjectures that all (non-commutative or commutative) crepant resolutions are derived equivalent \cite[Conj.\ 4.6]{van_den_bergh_non-commutative_2004}, if this is correct then this property uniquely determines the B-brane category.

\item Hori calculates \cite[Section 5.3]{hori_duality_2013} that the Witten index of this model is $ {q+s \choose q} $, and this is the size of the set $Y_{q,s}$ indexing the generators of $\Br(\widetilde{\cY})$. One might conjecture that it is also the Euler characteristic of the cyclic homology of the category.
\end{enumerate}

Of course we also propose that the category of B-branes for the GLSM $\widetilde{\cX}$ should be the subcategory $\Br(\widetilde{\cX})\subset D^b(\widetilde{\cX})$. To interpret Hori's duality we need a little bit more, we need to identify the category of B-branes in the GLSMs $\widetilde{\cY}\tms L^\perp$ and $\widetilde{\cX} \tms L $ with their tautological superpotentials $W$ and $\dualW$. Fortunately there is an obvious way to generalize our proposal --  we consider analogous subcategories inside the categories of matrix factorizations:
$$D^b(\widetilde{\cY}\tms L^\perp,\, W) \aand D^b(\widetilde{\cX}\tms L,\, \dualW)$$
A matrix factorization can be represented as a vector bundle, equipped with a `twisted differential', and we define full subcategories
$$\Br(\widetilde{\cY}\tms L^\perp,\, W) \aand \Br(\widetilde{\cX}\tms L,\, \dualW)$$
where we insist that this vector bundle is a direct sum of the bundles coming from the sets $Y_{q,s}$ or $Y_{s,q}$ (a slightly more elegant definition of these categories is given in Section \ref{sec:statement}).

So for us, Hori's duality becomes the predicted equivalence of categories
$$\Br(\widetilde{\cY}\tms L^\perp,\, W) \;\cong\; \Br(\widetilde{\cX}\tms L,\, \dualW) $$
for any $L$. In fact we really want this with the additional $\C^*$ in place, so it's an equivalence between categories defined on $\cY\tms_{\C^*} L^\perp$ and $\cX\tms_{\C^*} L$.

\subsection{A sketch of our proof}\label{sec:introSketchProof}

There are two things we need to prove: first that our interpretation of Hori's duality holds, and second that this implies HPD for Pfaffians. The first point will be proved in Section \ref{sec:affine}, and the second in Section \ref{sec:Projective}. Here we give a sketch of both proofs.
\pgap

We begin with Hori's duality in the extreme case $L = 0$, so we want to prove the equivalence:
$$\Br(\cX) \;\cong \;\Br(\cY \tms_{\C^*} \Wedge^2 V,\, W) $$
Recall that the stack $\cX$ maps to $[\Wedge^2 V /\C^*]$, hitting the locus $[\widetilde{\Pf_s}/\C^*]$.  The other stack $\cY \tms_{\C^*} \Wedge^2 V$ also maps to $[\Wedge^2 V /\C^*]$, just by projection onto the second factor. An important aspect of our argument is that our equivalence will be constructed relative to this common base, \emph{i.e.} it comes from a Fourier-Mukai kernel defined on their fibre product over this base. The definition of this kernel is fairly straightforward (see Section \ref{sec:Kernel}).

Having defined our functor, we can then base-change to open substacks of  $[\Wedge^2 V /\C^*]$ and examine it there. In particular we can delete the rank $< 2s$ locus: this removes all the singularities of $\widetilde{\Pf_s}$, and $\cX$ becomes equivalent to the smooth locus in $\Pf_s$. Here the methods of our previous papers apply, and we use them to prove that over this locus our functor gives an equivalence between the $\Br$ subcategories on each side (see Section \ref{sec:generic}).

Next we need to extend this over the singularities (Section \ref{sec:Extending}). As we discussed in Section \ref{sec:introNCRs}, $\Br(\cX)$ is generated by a finite set of objects, and is equivalent to the derived category of their endomorphism algebra $A$. On the other side, we identify a `dual' set of generating objects in the category $\Br(\cY \tms_{\C^*} \Wedge^2 V,\, W)$, and prove that the endomorphisms of these dual objects also form a Cohen--Macaulay algebra. Because our functor is generically an equivalence these two algebras are generically isomorphic, and then we can use the Cohen--Macaulay property to deduce that they are isomorphic everywhere. It follows immediately that our functor is an equivalence.

This proves the case $L=0$, and then it's easy to prove it for general $L$ using the physical sketch of the previous section - just replace `integrating out the quadratic term' with Kn\"orrer periodicity. Then we have an equivalence
$$\Br(\cX\tms_{\C^*} L, \, W') \; \cong \; \Br(\cY\tms_{\C^*} L ^\perp,\, W) $$
which is our version of Hori's duality (Theorem \ref{thm:HoriDualityWithL}). 

\pgap

Now we move on to deducing HPD. Our equivalence is relative to $\Wedge^2 V$, so restricting  to the complement of the origin gives an equivalence:
$$\Br(\cX^{ss}\tms_{\C^*} L, \, W') \; \cong \; \Br(\cY\tms_{\C^*} (L^\perp\setminus 0),\, W)$$
In the terminology of the previous section, this relates the `easy' phase on the left-hand side with the `difficult' phase on the right-hand side. Let's discuss the left-hand side first. The non-stacky locus in $\cX^{ss}$ is the smooth locus in $\Pf_s$, and here a standard application of Kn\"orrer periodicity implies that (assuming $L$ is generic)
$$D^b(\Pf_s^{sm} \tms_{\C^*} L, \, W')  \cong D^b(\Pf_s^{sm}\cap \P L^\perp ) $$
as in the physical sketch. It's fairly straightforward to extend this over the singular locus and we prove in Section \ref{sec:NCRslices} that we have an equivalence
$$ \Br(\cX^{ss} \tms_{\C^*} L, \, W') \cong D^b( \Pf_s \cap \P L^\perp,\, A|_{\P L^\perp}) $$
where $A$ is the sheaf of non-commutative algebras discussed in Section \ref{sec:introNCRs}. This proves that, for generic $L$, our non-commutative resolution of  $\Pf_s\cap \P L^\perp$ is equivalent to the `difficult' phase of the dual model.

A more challenging step is to compare the two categories
$$\Br(\cY\tms_{\C^*} (L^\perp\setminus 0),\, W) \aand \Br(\cY^{ss}\tms_{\C^*} L^\perp,\, W)  $$
\emph{i.e.} the categories for the `difficult' and `easy' phases of the dual model. This is a kind of variation-of-GIT process, and we use the idea of `windows' \cite{segal_equivalence_2011,halpern-leistner_derived_2015,ballard_variation_2012} (which was also used in \cite{ADS_pfaffian_2015,ST_2014,rennemo_homological_2015}). What we do is to lift both categories to the ambient stack $\cY \tms_{\C^*} L^\perp$, by finding  subcategories of $\Br(\cY\tms_{\C^*} L^\perp,\, W)$ to which they are equivalent. The window for the `difficult' phase is essentially standard, but the window that we need for the `easy' phase is not the one provided by general theory, and although it's easy to describe it takes us quite a lot of new calculations to prove that it works (see Section \ref{sec:Windows}).

One of our windows is obviously contained in the other, with the direction of containment depending on the dimension of $L$, and equality in the case $\dim L = sv$. This means that we have an embedding
\beq{eq:WindowCompare}\Br(\cY\tms_{\C^*} (L^\perp\setminus 0),\, W)\;\into\; \Br(\cY^{ss}\tms_{\C^*} L^\perp,\, W)  \eeq
or vice-versa. Putting this together with our previous results, we get an embedding
$$D^b( \Pf_s \cap \P L^\perp,\, A|_{\P L^\perp})\;  \into \; D^b( \Pf_q \cap \P L,\, B|_{\P L})  $$
or vice-versa. In the critical case $\dim L = sv$ the two categories are equivalent, and in this case both are Calabi--Yau (as stated in Theorem \ref{thm:CYcase}).

The claims of the preceding paragraph are the most important consequences of HPD but are some way from being the full statement as Kuznetsov wrote it. No doubt it is possible to prove the full statement directly in our situation, but instead we take a slight digression (Section \ref{sec:HPD}) so that we can bootstrap off some other work of the first author. If we look at \eqref{eq:WindowCompare} in the case $\dim L^\perp=1$ we get an embedding:
$$\Br(\widetilde{\cY})\;\into\; \Br(\cY^{ss}\tms_{\C^*} \C,\, W)  $$ 
Notice that $\cY^{ss}\tms_{\C^*} \C$ is the total space of a line bundle on $\cY^{ss}$, and that $W$ is essentially a choice of section of the dual line bundle.  We can do this in families where we let  $L^\perp$ vary in $\Wedge^2 V$ -- which means letting $W$ vary -- and the universal such family gives an embedding:
$$ \Br(\cY \tms_{\C^*} (\Wedge^2 V\setminus 0),\, W)  \; \into \; D^b\big( (\cY^{ss}\tms_{\C^*} \C) \tms_{\C^*}  (\Wedge^2 V\setminus 0), \, W \big) $$
The target space here is the total space of a  line bundle over $\cY^{ss}\tms \P(\Wedge^2 V)$; it is the setting for the `tautological' HP dual of $\Br(\cY^{ss})$ as described by the first author \cite{rennemo_homological}  (see Section \ref{sec:HPDualityBackground}). We show that the image of our embedding is exactly the HP dual to $\Br(\cY^{ss})$, and our previous results show that the category being embedded is equivalent to $\Br(\cX^{ss})$. This proves that $\Br(\cY^{ss})$ is HP dual to  $\Br(\cX^{ss})$, or in other words, $D^b(\Pf_q, B)$ is HP dual to $D^b(\Pf_s, A)$.


\section{Technical background}

\subsection{Matrix factorization categories}
As is clear from the introduction, our proofs and results involve (graded) matrix factorization categories.
We will here restrict ourselves to stating precisely what we mean by these categories -- for further background on definitions and tools, see \cite{ADS_pfaffian_2015, rennemo_homological_2015, ballard_homological_2013}.

By a \emph{Landau--Ginzburg B-model} we mean the data of
\begin{itemize}
\item a stack $\cX$ of the form $[X/(G \times \C^{*})]$, where $X$ is a smooth, quasi-projective variety, $G$ is a reductive group.
\item a function $W : X \to \C$, called the \emph{superpotential}. 
\end{itemize}
We denote the distinguished $\C^{*}$-factor of the group by $\C^{*}_{R}$, and refer to it as the \emph{$R$-charge}.

This set of data is subject to some restrictions:
\begin{itemize}
\item The function $W$ is $G$-invariant and has degree 2 with respect to the $\C^{*}_{R}$-action.
\item The element $-1 \in \C^{*}_{R}$ acts trivially on $[X/G]$, i.e.\ there exists a $g \in G$ such that $(g,-1)$ acts trivially on $X$.
\end{itemize}

Given this data, by work of Positselski and Orlov \cite{positselski_two_2011, efimov_coherent_2015, orlov_matrix_2012} one can define a category of matrix factorizations $D^{b}(\cX,W)$.
Let's say briefly what this category looks like.

There is a natural map $\cX \to [\pt/\C^{*}_{R}]$, we denote the pullback of the standard line bundle on $[\pt/\C^{*}_{R}]$ via this map by $\cO_{\cX}[1]$.
More generally, for any sheaf $\cE$ on $\cX$, we write $\cE[1]$ for $\cE \otimes \cO_{\cX}[1]$.
The objects of $D^{b}(\cX,W)$ can then be described as curved dg sheaves $(\cE,d_{\cE})$ on $\cX$, meaning the data of
\begin{itemize}
\item A coherent sheaf $\cE$ on $\cX$
\item A `twisted differential' $d_{\cE} : \cE \to \cE[1]$ such that $d_{\cE}^{2} = W \otimes \id_{\cE} : \cE \to \cE[2]$.
\end{itemize}
Given two curved dg sheaves $(\cE,d_{\cE})$ and $(\cF,d_{\cF})$, the sheaf $\hom(\cE,\cF)$ on $\cX$ inherits a differential, so becomes a dg sheaf, and one can take its cohomology to turn the set of curved dg sheaves into a category.
Just as for the ordinary derived category, this definition is too naive, and defining the morphism spaces and the triangulated category structure on $D^{b}(\cX,W)$ requires that we take the Verdier quotient by some subcategory of `acyclic' curved dg sheaves.

Without going into details of this, let's just mention that given two objects $(\cE,d_{\cE})$ and $(\cF,d_{\cF})$, the morphism spaces $\Hom(\cE,\cF)$ can be computed as follows.
We can find a curved dg sheaf $(\cE^{\prime},d_{\cE^{\prime}})$ such that $\cE^{\prime}$ is locally free, together with a quasi-isomorphism $\cE^{\prime} \to \cE$.
As mentioned above, $\hom(\cE^{\prime},\cF)$ is a complex, and we can compute
\[
  \RHom(\cE,\cF) \cong R\Gamma(\hom(\cE^{\prime},\cF)).
\]
Given two  LG models $(\cX, W_{\cX})$ and $(\cY,W_{\cY})$, a morphism between them is a map $f:\cX \to \cY$ such that $W_{\cX} = W_{\cY} \circ f$ and such that $f^{*}(\cO_{\cY}[1]) \cong \cO_{\cX}[1]$. 
The usual ensemble of (derived) functors exists, e.g. we have a pull-back functor $f^{*}: D^{b}(\cY,W_{\cY}) \to D^{b}(\cX,W_{\cX})$, and if $f$ is proper we have a push-forward functor $f_{*} : D^{b}(\cX,W_{\cX}) \to D^{b}(\cY,W_{\cY})$. Given an object $\cE \in D^{b}(\cX, W)$, we get a functor $- \otimes \cE : D^{b}(\cX,W^{\prime}) \to D^{b}(\cX, W^{\prime} + W)$.

Composing these functors, the formalism of Fourier--Mukai kernels can be used: if we have a kernel object $\cE \in D^{b}(\cX \times \cY,W_{\cY} - W_{\cX})$ (whose support is proper over $\cY$) then we get a functor $(\pi_{\cY})_{*}(\cE \otimes \pi_{\cX}^{*}(-))$.

\begin{rem}
  From this point on, we will abuse notation and write $D^{b}([X/G],W)$ for what is here denoted $D^{b}([X/G \times \C^{*}_{R}],W)$ -- i.e.\ we will leave the $\C^{*}_{R}$-action implicit.
  This notational choice reflects the idea that it is best to think of $[X/G]$ as the geometric object underlying these categories, as illustrated by the fact \cite[Prop. 2.1.6]{ballard_homological_2013} that if $W = 0$ and $\C^{*}_{R}$ acts trivially, then $D^{b}([X/G\times \C^{*}_{R}],0)$ is equivalent to the usual derived category $D^{b}([X/G])$. 
\end{rem}

\begin{rem}
  Since everything is derived, from now on we'll denote morphism spaces in $D^b(\cX, W)$ just by 
  $$\Hom(\cE, \cF) $$
  rather than $\RHom(\cE, \cF)$. Similarly when we write
  $$\hom(\cE, \cF)$$
  we'll always mean the derived sheaf homomorphisms, not the naive version appearing in the discussion above.
\end{rem}

\begin{rem}
  We should highlight one piece of our terminology which may differ from other authors: we reserve the term \emph{matrix factorization} for a curved dg-sheaf $(\cE, d_\cE)$ where $\cE$ is actually a finite-rank vector bundle. Every object in $D^b(\cX, W)$ is equivalent to a matrix factorization; for us this statement is part of the definition of the category (though really one should define $D^b(\cX, W)$ as some category of compact objects and then prove the statement).
\end{rem}

\subsection{Homological projective duality via LG models}
\label{sec:HPDualityBackground}
We recall the basic definitions and theorems of HP duality, phrased in terms of LG models as in \cite{rennemo_homological} or \cite{ballard_homological_2013}.
The original source is \cite{kuznetsov_homological_2007}; see also \cite{ballard_homological_2013, thomas_notes_2015, rennemo_homological_2015} for further background.

Let $\cS$ be an algebraic stack with a globally generated line bundle $\cL$, and set $U = H^0(\cS,\cL)$. We assume that $\cS$ is actually a quotient of a smooth quasi-projective variety by a reductive group,  which is the hypothesis used in \cite{ballard_variation_2012}.

Let $\cW \subseteq D^{b}(\cS)$ be a full admissible subcategory which is also saturated,\emph{i.e.}~every functor $\cW \to D^{b}(\C)^{\mathrm{op}}$ and $\cW^{\mathrm{op}} \to D^{b}(\C)^{\mathrm{op}}$ is representable.
We assume $\cW$ is closed under tensoring with $\cL$. Suppose further that $\cW$ is equipped with a Lefschetz decomposition with respect to $\cL$,  meaning that we have a semiorthogonal decomposition
\[
  \cW = \big\langle \cA_{0}, \;\cA_{1}(1), \;\ldots, \;\cA_{N}(N)\big\rangle
\]
where $\cA_{0} \supseteq \cA_{1} \supseteq \cdots \supseteq \cA_{N}$ are full admissible subcategories, and we're writing $\cA_i(i)$ as shorthand for $\cA_i\otimes \cL^i$.

The choice of $\cW$ and its Lefschetz decomposition are the input data for HPD. In its original formulation $\cS$ is a smooth variety and $\cW$ is the whole of $D^b(\cS)$, but we need this extra generality to handle our non-commutative resolutions.

The stack $\cS$ admits a map
$$\cS \to \P U^\vee$$
By assumption $\cW$ is also defined relative to this base, since it is preserved by $\cL$.  If we pick a linear subspace $L \subset U^\vee$ then the pre-image of $\P L$ is a linear section $\cS_{L}\subset \cS$, and roughly-speaking it's possible to base-change the category $\cW$ to give a subcategory $\cW_L \subset D^b(\cS_L)$ (we explain this more accurately below). 

Let $l'$ denote the codimension of $l'$.  It is not hard to show that each of the categories $\cA_{l'}(l'), \ldots, \cA_{N}(N)$ embeds fully faithfully into $\cW_{L}$ under the restriction functor $\cW \to \cW_L$, and also that these subcategories remain semi-orthogonal in $\cW_L$. If we denote their common orthogonal by $\cC_L$ then we have a semi-orthogonal decomposition:
$$ 
\cW_{L} = \big\langle \cC_{L},\; \cA_{l'}(l'),\; \ldots,\; \cA_{N}(N) \big\rangle
$$
We regard $\cC_L$ as the \qquotes{interesting part} of the category $\cW_L$; if the codimension of $L$ is high enough then it is the whole of $\cW_L$. 

The goal of HPD is to construct a `dual' category $\cW^\vee$, defined over the dual projective space $\P U$, and also equipped with Lefschetz decomposition. A choice of subspace $L$ determines an orthogonal subspace $L^\perp \subset U$, then we have a base-changed category $\cW^\vee_{L^\perp}$; as before this category has a semi-orthogonal decomposition into some pieces from $\cW^\vee$ and some remaining piece. The key required property of $\cW^\vee$ is that, for any choice of $L$,  this remaining piece must be equivalent to the category $\cC_L$. As a slogan: \emph{the interesting parts of $\cW_L$ and $\cW^\vee_{L^\perp}$ are equivalent}. 

In \cite{rennemo_homological} (following  \cite{ballard_homological_2013}) the first author gives a `tautological' construction of the HP dual in this level of generality.  We now describe the construction. 
\pgap

View $\cS \times U$ as a trivial vector bundle over $\cS$ and form the direct sum $ \cL^\vee \oplus U $. Now let $\C^*$ act fibre-wise on this vector bundle, with weight 1 on the first factor and weight ${-1}$ on the second factor, and take the quotient stack:
$$\cT = \Stack{ \cL^\vee \oplus U}{\C^*}$$
Any element $u \in U$ defines a section of $\cL$ and hence a function $W_{u} : \cL^\vee \to \C$.
Thus there is a canonical superpotential $W : \cT \to \C$ given by $W(x,u) = W_u(x)$. We add R-charge by declaring that $\C^*_R$ acts trivially on $\cL^\vee$ and with weight 2 on $U$, then $(\cT,W)$ is an LG model.

Let $\pi : \cT \to \Stack{U}{\C^{*}}$ be the projection, and let:
$$\cT^{*} = \pi^{-1}(\P U ) = \Stack{\cL^\vee\tms (U\setminus 0)}{\C^*}$$
We can view $(\cT^*, W)$ as a family of LG models over $\P U$; each fibre is a copy of the line bundle $\cL^\vee$, but equipped with a varying superpotential $W_u$. On each fibre Kn\"orrer periodity gives an equivalence
\begin{equation}\label{eq:FibreKP}D^b(\cL^\vee, W_u) \cong D^b(\{u=0\}) \end{equation}
to the derived category of the hypersurface cut out by $u$. 

Given a curved dg-sheaf $\cE$ on $\cT^*$ and a point $[u]\in \P U$ we can restrict $\cE$ to the fibre $\cL^\vee\tms[u]$ to get an object in $D^b(\cL^\vee, W_u)$, and we can further restrict this object to the zero section to get a complex in $D^b(\cS)$.

\begin{defn} \label{defn:tautologicalHPD} Let
  $$\cW^{\vee} \subset D^{b}(\cT^*,W)$$
  be the full subcategory of  objects $\cE$ such that for all $[u] \in \P U$, the restriction $\cE|_{\cS \tms [u]}$ is contained in the subcategory $\cA_{0} \subseteq D^b(\cS)$. We call the category $\cW^{\vee}$  the \emph{HP dual} of $\cW$.
\end{defn}

In Kuznetsov's original formulation \cite{kuznetsov_homological_2007} the term `HP dual' is reserved for varieties derived equivalent to a certain subcategory $\cC \subseteq D^{b}(\cH)$, where $\cH \subseteq \cS \times \P U$ is the universal hyperplane section. However, using the equivalence  \eqref{eq:FibreKP} one can show that if  if we restrict ourselves to Kuznetsov's situation then the two definitions agree \cite{rennemo_homological, ballard_homological_2013}.  Furthermore $\cW^\vee$ satisfies Proposition \ref{thm:HPDualHasLefschetzDecomposition} and Theorem \ref{thm:generalHPDTheorem} below, which are key properties of an HP dual. 
\pgap

Observe that $\cW^\vee$ is closed under tensoring by the line bundle $\cO(-1)$ pulled-up from $\P U$. Let $N'$ be the minimal integer such that $\cA_{N'} \not= \cA_{0}$, and let $M = \dim V - N' - 1$.
\begin{prop}[{\cite[Thm.\ 6.3]{kuznetsov_homological_2007}, \cite{rennemo_homological}}]
  \label{thm:HPDualHasLefschetzDecomposition}
  The category $\cW^{\vee}$ admits a Lefschetz decomposition
  \[
    \cW^{\vee} = \big\langle \cB_{-M}(-M), \; \ldots,\; \cB_{-1}(-1),\; \cB_{0}\big\rangle.
  \]
\end{prop}

Now fix a linear subspace $L \subset U^\vee$, with orthogonal $L^\perp\subset U$, and hence a linear section $\cS_L \subset \cS$.  We need to describe how to base-change the categories $\cW$ and $\cW^\vee$ to $L$ and $L^\perp$. 

The stack $\cT$ contains a substack:
$$\cT_{L^\perp} = \pi^{-1}(L^\perp) = \Stack{\cL^\vee \oplus L^\perp}{\C^*}$$
We can intersect this with $\cT^*$ to get 
$$\cT_{L^\perp}^* = \Stack{\cL^\vee \tms (  L^\perp \setminus 0)}{\C^*}$$
which is simply the restriction of $\cT^*$ to the linear subspace $\P L^\perp \subset \P U$. Obviously we can restrict $W$ to $\cT^*_{L^{\perp}}$, and it's $\C^*_R$-invariant, so $(\cT^*_{L^\perp}, W)$ is an LG model. The definition of $\cW^\vee_{L^\perp}$ is a trivial adaptation of the definition of $\cW^\vee$: we let   
$$\cW_{L^\perp}^{\vee} \;\subset\; D^{b}(\cT^*_{L^\perp}, W)$$
be the full subcategory of objects $\cE$ such that for each $[u] \in \P(L^\perp)$ the restriction of $\cE$ to $D^{b}(\cS \tms [u])$ lies in the subcategory  in $\cA_{0} \subseteq \cW$.

The definition of $\cW_L$ is less obvious. If take the stack $\cT_{L^\perp}$ and cut out the zero section $\cS\subset \cL^\vee$ we obtain
$$ \Stack{(\cL^\vee \setminus \cS)\tms L^\perp}{\C^*} \; \cong\; \cL^\vee\otimes L^\perp $$
which is just the total space of a vector bundle over $\cS$.  Kn\"orrer periodicity says that 
$$D^b( \cL^\vee\otimes L^\perp, W) \; \cong \; D^b(\cS_L) $$
provided that $\cS_L$ has the correct dimension. If $\cS_L$ doesn't have the correct dimension then we regard the left-hand-side as the correct replacement for $D^b(\cS_L)$, in particular we define $\cW_L$ as a subcategory of $D^b( \cL^\vee\otimes L^\perp, W)$ rather than as  a subcategory of $D^b(\cS_L)$.  Specifically, we let
$$\cW_L \; \subset\; D^b( \cL^\vee\otimes L^\perp, W) $$
be the full subcategory of objects $\cF$ such that the restriction of $\cF$ to the zero section lies in the subcategory $\cW\subset D^b(\cS)$. If $\cS_L$ has the correct dimension then $\cW_L$ is equivalent to the subcategory of $D^b(\cS_L)$ consisting of objects whose pushforward into $D^b(\cS)$ lie in $\cW$. 



Let  $l = \dim L$, and recall that $l' = \dim L^{\perp} = \dim U - l$.

\begin{thm}[\cite{kuznetsov_homological_2007, rennemo_homological}]
  \label{thm:generalHPDTheorem}
  There exist semiorthogonal decompositions 
  \[
    \cW_{L} = \big\langle \cC_{L},\; \cA_{l'}(l'),\; \ldots,\; \cA_{N}(N) \big\rangle
  \]
  and
  \[
    \cW^{\vee}_{L^\perp} = \big\langle \cB_{-M}(-M), \;\ldots,\; \cB_{-l}(-l), \;\cC_{L} \big\rangle
  \]
  for some category $\cC_L$. 
\end{thm}

Of course the fact that $\cL^\vee\otimes L^\perp$ and $\cT^*_{L^\perp}$ are both open substacks of $\cT_{L^\perp}$ is a key part of the proof. 

\subsection{Minimal resolutions}\label{sec:windowsthesame}

Let $G$ be a reductive group and $V$ a $G$-representation. Add a linear $R$-charge $\C^*_R$ action on $V$ and a $G$-invariant superpotential $W$, so $([V/ G], W)$ is an LG model. 

Now let $Y$ be a set of irreps of $G$, or the corresponding set of vector bundles on $[V/G]$.  Given an object $\cE\in D^b([V/G], W)$, we can ask whether $\cE$ can be represented by a matrix factorization whose underlying vector bundle is a direct sum of bundles from $Y$. 

This question will occur frequently in this paper, because our brane subcategories are specified by exactly this kind of condition.  To get a weaker condition consider the restriction $\cE|_0$ to the origin; this is a complex of $G$-representations and we can ask if its homology $h_\bullet(\cE|_0)$ contains only irreps from the set $Y$. Obviously the first condition implies the second condition, and the second condition is much easier to check. 

We will make frequent use of the following lemma, which says that under a particular hypothesis the two conditions are equivalent.

\begin{lem} \label{lem:minimalmodels} Assume that the combined group $G\tms \C^*_R$ acting on $V$ includes a central 1-parameter-subgroup $\sigma$ which acts as some non-zero power of the usual dilation action on $V$. Then any object $\cE\in D^b_G(V, W)$ is equivalent to a matrix factorization $(E, d_E)$ such that $d_E|_0=0$, and $E$ is the vector bundle associated to the $G$-representation $h_\bullet(\cE|_0)$. 
\end{lem}

If $W=0$ this is the standard theory of minimal resolutions, e.g. \cite[Prop.\ 5.4.2]{weyman_cohomology_2003}. Without the $G$-action the statement is essentially \cite[Prop.\ 7]{khovanov_matrix_2008}. The following proof works whether $W=0$ or not.

\begin{proof}
  Since $\cO_V$ is graded by $\sigma$ (either non-positively or non-negatively) with $(\cO_V)_0=\C$, it's elementary that any $(G\tms \C^*_R)$-equivariant vector bundle on $V$ must be the bundle associated to some representation; see e.g. \cite[p.\ 150]{weyman_cohomology_2003}. By definition, any object $\cE\in D^b_G(V, W)$ is equivalent to a matrix factorization $(E, d_E)$, and then $E$ is the vector bundle associated to the $(G\tms \C^*_R)$-representation $E|_0$. 
  If $d_E|_0 = 0$, then it is immediate that $E$ is the vector bundle associated to the $G$-representation $h_\bullet(\cE|_0)$.

  Suppose that $d_E|_0\neq 0$. We claim that then $E$ contains a contractible sub-object whose underlying sheaf is a subbundle of $E$. Quotienting by this subbundle produces an equivalent matrix factorization of smaller rank, so applying this recursively results in a model of the required form.

  To prove the claim we consider $(E|_0, d_E|_0)$. This is a bounded complex of $G$-representations, which we can decompose into graded irreps. If $d_E|_0\neq 0$ there must be some component of $d_E|_0$ of the form
  $$U \stackrel{\id_U}{\To} U[-1] $$
  where $U$ is a (shift of an) irrep.  Then $E$ contains two associated subbundles, and there is a component of $d_E$
  $$\delta: U \To U[-1]$$
  mapping between them. This $\delta$ is a $(G\tms \C^*_R)$-invariant element of $\Hom_V(U, U)$, reducing to $\id_U$ at the origin. In fact $\delta$ must be constant -- the 1-PS $\sigma$ is central so it acts on $U$ as scalar multiple of the identity, and then the only $\sigma$-invariant elements of $\Hom_V(U, U)$ are constant. So $\delta= \id_U$. 

  Now let $\iota: U\into E$ denote the inclusion of the first subbundle. The map $d_E\circ \iota: U[-1] \to E$ is $\C^*_R$-invariant, and one of its components is $\delta\circ\iota$ which is the inclusion of the second subbundle. Hence the map $(\iota, d_E\circ\iota): U\oplus U[-1] \to E$ has full rank at all points so it's the inclusion of a subbundle. This map embeds the contractible matrix factorization
  $$\begin{tikzcd} U \arrow[yshift=.4ex]{r}{\id_U}& U[-1] \ar[yshift=-.4ex]{l}{W\id_U} \end{tikzcd} $$
  as a subobject of $(E, d_E)$.
\end{proof}

\begin{rem} The proof still works if $V$ is just a cone instead of a vector space, \emph{i.e.} $\cO_V$ is a non-negatively graded (by $\sigma$) ring with $(\cO_V)_0=\C$. In fact it still works if $(\cO_V)_0$ is a local ring, and we replace `restrict to the origin' with `restrict to the unique closed point of $\Spec \,(\cO_V)_0$'. The only difference in the proof is that now the map $\delta$ need not be exactly $\id_U$, but it must be of the form $t\id_U$ for some unit $t\in (\cO_V)_0$ (the rest of the argument is identical). We will need this more general case in the proof of Lemma \ref{lem:formalnhdresult}.
\end{rem}

\begin{cor}
  \label{thm:ObjectVanishingAtOriginImpliesVanishing}
  In the situation of Lemma \ref{lem:minimalmodels}, suppose that $\cE \in D^{b}_{G}(V,W)$ is such that $\cE|_{0} \cong 0$.
  Then $\cE = 0$.
\end{cor}

\subsection{Symplectic similitude groups}\label{sec:GSp}

Let $S$ be a symplectic vector space of dimension $2s$, with symplectic form $\omega_S\in \Wedge^2 S^\vee$ and Poisson bivector $\beta_S\in \Wedge^2 S$. We'll make constant use of the group
$$\GSp(S) \subset \GL(S) $$
which preserves $\omega_S$ up to scale. $\GSp(S)$ has a canonical character $\GSp(S) \to \C^*$ defined by its action on the line $\langle \beta_{S} \rangle\subset \Wedge^2 S$ spanned by the Poisson bivector. This character is an $s$th root of the determinant, it generates the character lattice, and fixes a canonical isomorphism:
$$ \GSp(S)/\Sp(S)\;\cong \;\C^* $$
Let $\Delta : \C^{*} \to \GSp(S)$ be the diagonal subgroup. A Young diagram $\gamma$ of height $\leq s$ determines an irrep $\Schpur{\gamma} S$ of $\Sp(S)$.  An irrep of $\GSp(S)$ can be described by a pair $(\gamma, k)$, where $k\in \Z$ is a weight of $\Delta$, subject to the requirement that  $\sum_i \gamma_i + k \equiv 0$ mod 2.  We'll donote the corresponding irrep by:
$$\Schpur{\gamma, k} S$$ 
For example the character $\langle \beta_S \rangle$ is $\Schpur{\phi,2} S$ where $\phi$ is the empty Young diagram. Note that if two  irreps $\Schpur{\gamma, k} S$ and $\Schpur{\gamma, k'} S$ restrict to give the same irrep $\Schpur{\gamma} S$ of $\Sp(S)$ then they differ by some power of the character  $\langle \beta_S \rangle$.

\section{The affine duality}\label{sec:affine}

\subsection{The duality statement}\label{sec:statement}

Let's recall our notation from the introduction. We fix a vector space $V$ of dimension $v$, and two symplectic vector spaces $S$ and $Q$ of dimensions $2s$ and $2q$ respectively, where $v=2s+2q+1$. We let $\cX$ and $\cY$ denote the following stacks:
$$\cX =  \Stack{\Hom(S,V)}{\GSp(S)} \aand  \cY = \Stack{\Hom(V,Q)}{\GSp(Q)}  $$

We let $\GSp(S)$ act on the vector space $\Wedge^2 V^\vee$ as the representation $\Wedge^2 V^\vee\otimes \langle \beta_S \rangle$, \emph{i.e.}~the subgroup $\Sp(S)$ acts trivially and the residual $\C^*$ acts with weight 1. We also let $\GSp(Q)$ act on the dual $\Wedge^2 V$ as the representation $\Wedge^2 V \otimes \langle \beta_Q \rangle^{\vee}$, so $\Sp(Q)$ acts trivially but here the residual $\C^*$ acts with weight $-1$. Since 
$$\GSp(S)/\Sp(S)\cong \GSp(Q)/\Sp(Q)\cong \C^*$$
canonically it's safe to think of this as a single $\C^*$ acting with weight $1$ on $\Wedge^2 V^\vee$, and its dual action on $\Wedge^2 V$. 

Now we pick a subspace $L\subset \Wedge^2 V^\vee$ and its annihilator $L^\perp\subset \Wedge^2 V$. We can form the products
$$\cX\tms_{\C^*} L := \Stack{\Hom(S, V)\tms L}{\GSp(S)} $$
and:
$$\cY\tms_{\C^*}L^\perp := \Stack{\Hom(V, Q)\tms L^\perp}{\GSp(Q)} $$
We let $W'$ and $W$ denote the tautological invariant functions (`superpotentials') on these two stacks from \eqref{eq:W'} and \eqref{eq:W}:
$$ W' = b ( \wedge^2 x (\beta_S))$$
$$ W = \omega_Q(\wedge^2 y (a)) $$
Under the $\C^*$ action the element $\wedge^2 x(\beta_S)\in \Wedge^2 V$ has weight $-1$, and the element $b\in \Wedge^2 V^\vee$ has weight 1, so $W'$ really is invariant under $\GSp(S)$ and not just under $\Sp(S)$. Similarly $a\in \Wedge^2 V$ has weight $-1$, but $\omega_Q\circ \wedge^2 y\in \Wedge^2 V^\vee$ has weight 1.

We must specify an R-charge on both our stacks (this is in addition to the $\C^*$ actions already discussed). We do this by declaring that $\C^*_R$ acts with weight 0 on $\Hom(S,V)$ and $\Wedge^2 V$, with weight 1 on $\Hom(V,Q)$, and with weight 2 on $\Wedge^2 V^\vee$. Then both $W'$ and $W$ have R-charge 2, and both pairs
$$(\cX\tms_{\C^*} L, \, W') \aand (\cY\tms_{\C^*}L^\perp,\,  W) $$
are Landau-Ginzburg B-models.

Both stacks can be mapped ($\C^*_R$-equivariantly) to the stack $[\Wedge^2 V / \C^*] $. For the first stack we do this by projecting to $\cX$ and then applying the  map from $\Hom(S,V)$ to $\Wedge^2 V$, and for the second stack we simply project onto $L^\perp$ and then include. Throughout Section \ref{sec:affine} it will be helpful to think of $[\Wedge^2 V / \C^*]$ as the base over which all our constructions live, though in Section \ref{sec:Projective} this will no longer be true.

\begin{rem}\label{rem:Rcharge}
  It would be more symmetric to define the R-charge on $\cY\tms_{\C^*} L^\perp$ to have weight zero on $\Hom(V, Q)$ and weight 2 on $L^\perp$. These two actions differ by the diagonal 1-parameter subgroup $\Delta: \C^* \into \GSp(Q)$, so switching from one to the other doesn't change the category $D^b(\cY\tms_{\C^*}L^\perp,\,  W)$ and it doesn't really matter which action we choose -  see Remark \ref{rem:Rcharge2} below for more details. However for moment  it is simplest if we keep the map to $[\Wedge^2 V/\C^*]$ explicitly $\C^*_R$-equivarant. 
\end{rem}

Now we define/recall our `B-brane' subcategories. Recall that $Y_{s,q}$ denotes the set of Young diagrams of height at most $s$ and width at most $q$, and we'll use the same notation for the corresponding set of irreps of $\Sp(S)$. Similarly $Y_{q,s}$ denotes the transposed set of Young diagrams, or the corresponding set of irreps of $\Sp(Q)$. Given any object
$$ \cE\in D^b(\cX\tms_{\C^*} L, \, W') $$
we can consider its restriction $\cE|_0$ to the origin, which is a chain-complex of representations of $\GSp(S)$. The homology $h_\bullet(\cE|_0)$  of this is a graded $\GSp(S)$-representation, which we may view just as a representation of $\Sp(S)$. We define
$$ \Br(\cX\tms_{\C^*} L, \, W') \; \subset \; D^b(\cX\tms_{\C^*} L, \, W') $$
to be the full subcategory of all objects $\cE$ satisying the following `grade-restriction-rule':
$$\mbox{all irreps of }\Sp(S)\mbox{ occuring in }h_\bullet(\cE|_0)\mbox{ lie in the set }Y_{s,q} $$
Similarly 
$$ \Br(\cY\tms_{\C^*} L^\perp, \, W) \; \subset \; D^b(\cY\tms_{\C^*} L^\perp, \, W) $$
is the full subcategory of all objects $\cF$ satisying the grade-restriction-rule:
$$\mbox{all irreps of }\Sp(Q)\mbox{ occuring in }h_\bullet(\cF|_0)\mbox{ lie in the set }Y_{q,s} $$
Note that both subcategories are obviously triangulated because the grade-restriction-rule is preserved under taking mapping cones.

Let's connect these definitions to the ones used in the introduction. Any irrep of $\Sp(S)$ corresponds to $\Z$-many irreps of $\GSp(S)$, differing by powers of the character $\langle\beta_S\rangle$,  so the set $Y_{s,q}$ determines an infinite set of irreps of $\GSp(S)$. To any such irrep there is an associated vector bundle on $\cX\tms_{\C^*} L$. By Lemma \ref{lem:minimalmodels}, an object $\cE\in  D^b(\cX\tms_{\C^*} L, \, W')$ lies in the subcategory $\Br(\cX\tms_{\C^*} L, \, W')$ if and only if it can be represented as a matrix factorization whose underlying vector bundle is a direct sum of the bundles coming from $Y_{s,q}$. In particular in the case $L=0$, the statement is that $\Br(\cX)\subset D^b(\cX)$ is the subcategory generated by this infinite set of vector bundles.  Both versions of the definition will be useful in the course of our proofs. 

Now we can state our interpretation of Hori's duality for B-branes.

\begin{thm}\label{thm:HoriDualityWithL}
  For any $L\subset \Wedge^2 V^\vee$, we have an equivalence
  $$ \Br(\cX\tms_{\C^*} L, \, W') \;\isoto \;  \Br(\cY\tms_{\C^*} L^\perp, \, W) $$
  of categories over $[\Wedge^2 V/\C^*]$. 
\end{thm}
When we say that our equivalence is `over $[\Wedge^2 V/\C^*]$' we mean simply that it is given by a Fourier--Mukai kernel on the fibre product over this base. It follows that we may restrict our categories and our kernel to open substacks in the base, and still get an equivalence.

Of course we could have chosen to work over $[\Wedge^2 V^\vee/\C^*]$ instead, in which case we'd get an equivalence relative to that base. This might suggest that we should really be able to define our kernel over $[\Wedge^2 V\tms \Wedge^2 V^\vee /\C^*]$, but this appears not to be the case.\footnote{This same phenomenon can be seen in the simplest examples of Kn\"orrer periodicity: the kernel for equivalence between $D^b(\A^2, xy)$ and $D^b(pt)$ can be constructed either relative to the base $\A^1_x$, or relative to $\A^1_y$, but not relative to $\A^2$.}
\pgap

If we forget the $\C^*$-action (\emph{i.e.}~pull-back along the map $\Wedge^2 V\to [\Wedge^2 V/\C^*]$)  we obtain an equivalence
$$ \Br(\widetilde{\cX}\tms L, \, W') \;\cong \;  \Br(\widetilde{\cY}\tms L^\perp, \, W) $$
which expresses the duality between symplectic GLSMs.  Here $\Br(\widetilde{\cX}\tms L, \, W')$ and $\Br(\widetilde{\cY}\tms L^\perp, \, W)$ are defined to be the images of $\Br(\cX\tms_{\C^*} L, \, W')$ and $\Br(\cY\tms_{\C^*} L^\perp, \, W)$ under the pull-back functor; they can also be defined to be the subcategories consisting of matrix factorizations built from the vector bundles  associated to the sets $Y_{s,q}$ and $Y_{q,s}$, as in Section \ref{sec:introNCRs} (they cannot be defined by a grade-restriction-rule at the origin, because now the group action has closed orbits other than the origin and Lemma \ref{lem:minimalmodels} doesn't apply). 
\pgap

The remainder of Section \ref{sec:affine} will be devoted to the proof of Theorem \ref{thm:HoriDualityWithL}.  As sketched in Section \ref{sec:introSketchProof} most of the work will be in proving the extreme case $L=0$, from there we will deduce the general case quite easily.

\subsection{Relating the brane subcategories to (curved) algebras}\label{sec:Br(X)}

We have a map 
$$\cX \to \Stack{\Wedge^2 V}{\C^*}$$
and for much of this paper we want to work relative to this base; to facilitate this we'll now introduce some notation for relative morphisms.  Given objects $\cE, \cF\in D^b(\cX)$ there is a dg-sheaf of (derived) local homomorphisms $\hom(\cE, \cF) \in D^b(\cX)$, this is a chain-complex of $\GSp(S)$-equivariant modules over the ring of functions on $\Hom(S, V)$. The global sections
$$\Hom(\cE, \cF) = \hom(\cE, \cF)^{\GSp(S)}$$
are the morphisms in $D^b(\cX)$. Instead we can push $\hom(\cE, \cF)$ down to the stack $[\Wedge^2 V / \C^*]$: this just takes $\Sp(S)$ invariants, and the result is a module over the ring of functions on the Pfaffian cone $\widetilde{\Pf}_s \subset \Wedge^2 V$.  Let us denote this push-down by:
$$\Hom(\cE, \cF)_{gr} =  \hom(\cE, \cF)^{\Sp(S)}\; \in D^b([\Pf_s / \C^*]) $$
The notation reflects the fact that $\Hom(\cE, \cF)_{gr}$ is a graded module (because of the $\C^*$ action), and note that $\Hom(\cE, \cF)$ is the degree zero part of $\Hom(\cE, \cF)_{gr}$. 
\pgap

On the stack $\widetilde{\cX}$ we have the subcategory $\Br(\widetilde{\cX})$, generated by the vector bundles associated to irreps in $Y_{s,q}$. We discussed in Section \ref{sec:introNCRs} how we may also view this as the derived category of an algebra: we take the tilting bundle 
$$\widetilde{T}= \bigoplus_{\gamma\in Y_{s,q}} \Schpur{\gamma} S$$
and consider its endomorphism algebra $\widetilde{A} = \End_{\widetilde{\cX}}(\widetilde{T})$, then $\Br(\widetilde{\cX})$ is equivalent to the derived category of $\widetilde{A}$. We'll now do a similar procedure for the category $\Br(\cX)$.


We need to choose a vector bundle $T$ on $\cX$ which pulls-back to give the bundle $\widetilde{T}$ on $\widetilde{\cX}$. The summands of $\widetilde{T}$ correspond to irreps $\Schpur{\gamma} S$ of the group $\Sp(S)$, and for each summand we must make a choice of integer $k_\gamma$ to extend the irrep to an irrep $\Schpur{\gamma, k_\gamma} S$ of  $\GSp(S)$ (see Section \ref{sec:GSp}). Then 
$$T = \bigoplus_{\gamma\in Y_{s,q}} \Schpur{\gamma, k_\gamma} S $$
is a vector bundle on $\cX$, and we can define: 
$$A= \Hom(T, T)_{gr}$$
This is a graded algebra,  defined over the  ring of functions on $\widetilde{\Pf_s}$, \emph{i.e.} there is a map from the commutative graded ring $\cO_{\widetilde{\Pf_s}}$ to the centre of $A$. We can also think of it as a quiver algebra (with relations), where the underlying quiver has vertices indexed by the set $Y_{s,q}$.

\begin{rem}\label{rem:AisAmbiguous}
  This definition of $A$ is slightly ambiguous because it depends on the choice of integers $k_\gamma$ in the bundle $T$, but this ambiguity is harmless. Making a different choice just means tensoring our $\GSp(S)$-irreps by powers of the character $\langle \beta_S \rangle$, \emph{i.e.} changing the weight of the $\C^*$ action for each summand, this results in the same algebra but with a different grading. For another way to understand this: observe that $A$ splits as an $A$-module into `vertex projectives' $A= \oplus_{\gamma\in Y_{s,q}} P_{\gamma}$ and tautologically $A$ is the endomorphism algebra of these projective modules. If we shift the grading on each $P_\gamma$ and then form their endomorphism algebra we get the algebra $A$ but with a different grading; this exactly corresponds to changing our choice of $T$. In particular all choices lead to Morita equivalent algebras. 

  This ambiguity will be convenient in Section \ref{sec:Extending} as it saves us from calculating an irrelevant set of degree shifts, see Remarks \ref{rem:AmbiguityAgain} and \ref{rem:A'isAmbiguous}.
\end{rem}

\begin{lem}
  \label{lem:BranesToModules}
  We have an equivalence:
  \[
    \Hom(T, - )_{gr} : \Br(\cX) \isoto D^b(A)
  \]
\end{lem}
This equivalence takes the vector bundles which are the summands of $T$ to the projective $A$-modules $P_\gamma$. This lemma is basically a tautology: $\Br(\cX)$ is generated by $T$, and $D^b(A)$ is generated by $A$, and the endomorphism algebras of the generators are both $A$. The proof below is the just the preceding sentence written out carefully.
\begin{proof}
  The functors $F := \Hom(T,-)_{gr}$ and $F^{*} = T \otimes_{A}- $ are an adjoint pair of functors between $D(\cX)$ and $D(A)$ (the unbounded derived categories), and $F^{*}$ is fully faithful since $F\circ F^*$ is the identity. 

  Since $A$ has finite global dimension as an ungraded algebra by \cite[Thm.\ 1.4.2]{spenko_non-commutative_2015}, it has finite global dimension as a graded algebra by \cite[Cor.\ I.2.7]{nastasescu_graded_1982}. Hence $D^{b}(A)$ equals the smallest thick subcategory of $D(A)$ containing $\{A(i)\}_{i \in \Z}$. Now $F^{*}(A(i)) = T(i) \in \Br(\cX)$, and so since $\Br(\cX)$ is a thick subcategory of $D^{b}(\cX)$ and $F^{*}$ preserves direct sums, it follows that $F^{*}$ takes $D^{b}(A)$ into $\Br(\cX)$.

  Let $0 \not= \cE \in \Br(\cX)$.
  We may represent $\cE$ by a minimal complex
  \[
    0 \to \cE_{i} \to \cdots \to \cE_{j} \to 0
  \]
  as in Lemma \ref{lem:minimalmodels}.
  Let $U$ be an irreducible summand of $\cE_{j}$. The inclusion map $U \to \cE_{j}$ does not factor through $\cE_{j-1} \to \cE_{j}$ by minimality of the resolution, hence we get a non-zero map $U \to \cE[j]$.
  It follows that $\Hom(T,\cE[j])_{gr} \not= 0$, and so $F(\cE) \not= 0$.
  Thus $\ker F \cap \Br(\cX) = \{0\}$, and by \cite[Thm.\ 3.3]{kuznetsov_homological_2007}, $F^{*} : D^{b}(A) \to \Br(\cX)$ is then essentially surjective.
\end{proof}

The above equivalence is linear over our base $[\Wedge^2 V/\C^*]$, so it makes sense to base-change it to open subsets. For example, suppose we delete the origin in $\Wedge^2 V$. The restriction of $\cX$ to this locus is the open substack $\cX^{ss}\subset \cX$, and we define $\Br(\cX^{ss})$ to be the full subcategory of $D^b(\cX^{ss})$ generated by the image of $\Br(\cX)$ under restriction.  Equivalently, this is the subcategory generated by the (infinite) set of vector bundles corresponding to the set $Y_{s,q}$.

If we delete the origin from the stack $[\widetilde{\Pf_s}/\C^*]$ we get the projective variety $\Pf_s$. The algebra $A$ restricts to give a sheaf of algebras on $\Pf_s$, and have an associated abelian category of coherent modules over it, together with the various flavours of derived category.

Given objects $\cE, \cF\in D^b(\cX^{ss})$ we write 
$$\Hom(\cE, \cF)_{gr} \in D^b(\Pf_s)$$
to mean the push-down of $\hom(\cE, \cF)$ to the base $\Pf_s \subset \P(\Wedge^2 V)$. This is consistent with our previous notation: if $\cE$ and $\cF$ are the restrictions of objects $\widetilde{\cE}, \widetilde{\cF}\in D^b(\cX)$ then $\Hom(\cE, \cF)_{gr}$ is the complex of sheaves that we get from the complex of graded modules $\Hom(\widetilde{\cE}, \widetilde{\cF})_{gr}$. 

\begin{cor}
  \label{thm:BranesToModulesBaseChanged}
  We have an equivalence:
  \[
    \Hom(T|_{\cX^{ss}}, - )_{gr} : \Br(\cX^{ss}) \isoto D^b(\Pf_s  ,A)
  \]
\end{cor}
\begin{proof}
  Let $F_{ss} = \Hom (T |_{\cX^{ss}},  -  )_{gr}: D(\cX^{ss}) \to D(\Pf_{s},A)$, and let $F^{*}_{\ss}$ be the left adjoint.
  The functor $F^{*}_{\ss}$ is automatically fully faithful.

  If $\cE \in \Br(\cX^{\ss})$ then there exists an $\widetilde{\cE}\in \Br(\cX)$ restricting to $\cE$.
  Since $F$ is linear over $[\Wedge^{2}V/\C^*]$ we have that $\cE = F^{*}_{\ss}(F(\widetilde{\cE})|_{\Pf_{s}})$ and so $F^{*}_{\ss} : D^{b}(\Pf_{s},A) \to \Br(\cX^{\ss})$ is essentially surjective.
\end{proof}

\begin{rem}\label{rem:BasechangingBranesToModules}
  This same proof works if we restrict to other ($\C^*$-invariant) open subsets in $\Wedge^2 V$. 
\end{rem}

Now fix $L\subset \Wedge^2 V^\vee$ and consider our Landau--Ginzburg B-model $(\cX\tms_{\C^*} L, \dualW)$. We can adapt the equivalence of $\Br(\cX)$ with $D^b(A)$ to give another description of the category $\Br(\cX\tms_{\C^*} L, \dualW)$. 

Let's start by forgetting $\dualW$, and just tensoring the previous lemma by $\cO_L$. The vector bundle $T$ on $\cX$ can be pulled up to give a vector bundle on $\cX\tms_{\C^*} L$ which we'll continue to denote by $T$. and we have a functor
$$\Hom(T, -)_{gr}: D^b(\cX\tms_{\C^*} L) \To D^b(A\otimes \cO_L) $$
defined as before by taking $\Sp(S)$-invariants in $\hom(T,-)$. 

The algebra $A\otimes \cO_L$ has an obvious set of projective modules, obtained by taking the `vertex projective' $A$-modules $P_\gamma$ and  tensoring them with $\cO_L$. The adjoint functor
$$ T\otimes - : D^b(A\otimes \cO_L)  \To  D^b(\cX\tms_{\C^*} L) $$
sends each of these projective modules to the  corresponding vector bundle. It is an embedding with image $\Br(\cX\tms_{\C^*} L)$.

Now we introduce the superpotential $\dualW$. We can view $\dualW$ as a central element of the algebra $A\otimes\cO_L$, since this an algebra over the ring of functions on $\Wedge^2 V\times L$. So the pair $(A\otimes\cO_L, \dualW)$ is a `curved algebra', and there is an  associated category 
$$D^b(A\otimes \cO_L, \dualW)$$
of curved dg-modules. By definition every object in this category is equivalent to a curved dg-module whose underlying module is projective.

\begin{lem}
  \label{lem:BranesToModules2}
  We have a functor
  $$\Hom(T, -)_{gr} : D^b(\cX\tms_{\C^*} L,\, \dualW) \;\To\; D^b(A\otimes \cO_L, \dualW) $$
  with a left adjoint $T\otimes - $. The adjoint is an embedding and gives an equivalence:
  $$T\otimes - : D^b(A\otimes \cO_L, \dualW)\; \isoto \; \Br(\cX\tms_{\C^*} L,\, \dualW)$$
  In particular the subcategory $\Br(\cX\tms_{\C^*} L,\, \dualW)$ is right-admissible.
\end{lem}

\begin{proof}
  If we have an object $(\cE, d_\cE)\in D^b(\cX\tms_{\C^*} L,\, \dualW)$ then applying the functor $\Hom(T,-)_{gr}$ to $\cE$ gives an (R-charge equivariant) $A\otimes \cO_L$-module.  Under this functor the endomorphism $d_\cE$ maps to a endomorphism of the module $\Hom(T, \cE)_{gr}$, which squares to $\dualW$ -- note that the functor $\Hom(T,-)_{gr}$ is exact so there are no `up-to-homotopy' complications here.\footnote{The functor is exact because $T$ is a vector bundle and taking $\Sp(S)$ invariants is exact.}  We claim that this defines a functor
  $$\Hom(T, -)_{gr}: D^b(\cX\tms_{\C^*} L,\, \dualW)\To D^b(A\otimes \cO_L, \dualW)$$
  The technical issue in this claim is to check that the functor  does indeed land in  $D^b(A\otimes \cO_L, \dualW)$ and not some larger category of curved dg-modules. To see this observe that $A\otimes \cO_L$ has finite global dimension (since $A$ does) so $\Hom(T, \cE)_{gr}$  has a finite projective resolution. Then the perturbation technique of \cite[Lemma 3.6]{segal_equivalence_2011} implies that the curved dg-module $\Hom(T, \cE)_{gr}$ does lie in $D^b(A\otimes \cO_L, \dualW)$.

  Now we claim that the only projective $A\otimes \cO_L$-modules are the  obvious ones corresponding to the vertex projective $A$-modules. Observe that $A\otimes \cO_L$ is bi-graded; one grading is by the diagonal $\Delta\subset \GSp(S)$  and the other is the R-charge. If we collapse to a single grading appropriately then it becomes non-negatively graded, with its degree-zero part the semi-simple algebra $\C^{Y_{s,q}}$. Then the claim follows by the graded Nakayama lemma.

  Therefore any object in $ D^b(A\otimes \cO_L, \dualW)$ is equivalent to a curved dg-module built from these projective modules. The adjoint functor $T\otimes -$ identifies this category of curved dg-modules  with the category of matrix factorizations on $\cX\tms_{\C^*} L$ built from the summands of $T$ (this is an equivalence even at the chain level). So $T\otimes -$ gives an equivalence between $ D^b(A\otimes \cO_L, \dualW)$ and $\Br(\cX\tms_{\C^*} L,\, \dualW)$. 

  The final statement of the lemma follows formally.
\end{proof}

We shall see later in Section \ref{sec:NCRslices} that if $L$ is generic and its dimension is not too big, then there is a third way to describe the category $\Br(\cX\tms_{\C^*} L,\, \dualW)$, by using Kn\"orrer periodicity to remove the $L$ directions entirely.
\pgap

The fact that $A$ (or $\Br(\cX)$) is a non-commutative \emph{crepant} resolution of $\widetilde{\Pf_{s}}$ is reflected in the following partial Serre duality statement.
\begin{prop}
  \label{thm:SerreDuality1}
  Let $\cE, \cF \in \Br(\cX)$ with $\cF|_{\cX^{\ss}} = 0$.
  Then
  \[
    \Hom(\cE,\cF) = \Hom(\cF, \cE \otimes (\det S)^v [\dim \widetilde{\Pf_{s}}])^{\vee}.
  \]
\end{prop}

\begin{proof}
  The singular variety $\widetilde{\Pf_s}$ is Gorenstein, and the pull-back of $\omega_{\widetilde{\Pf}_{s}}$ to $\cX$ is $(\det S)^v$. Applying the equivalence $\Br(\cX) \cong D^{b}(A)$, the statement is a graded version of \cite[Lemma 6.4.1]{van_den_bergh_non-commutative_2004}.
\end{proof}

We have precisely analogous results on the $\cY$ side.  We have a graded algebra $B$, defined as the endomorphisms of a vector bundle on $\cY$, and an equivalence:
$$\Br(\cY) \isoto D^b(B)$$
Then for any $L^\perp$ we have a curved algebra $(B\otimes \cO_{L^\perp}, W)$, and an equivalence:
\[ \Br(\cY\tms_{\C^*} L^\perp,\, W) \;\isoto\; D^b(B\otimes \cO_{L^\perp}, W)
\]
In particular the subcategory $\Br(\cY\tms_{\C^*} L^\perp,\, W)$ is right-admissible.

\begin{rem}\label{rem:Rcharge2}
  There is a subtlety here which we need to highlight. In Section \ref{sec:statement} we defined the R-charge to act non-trivially on $\cY$, though we claimed in Remark \ref{rem:Rcharge} that this didn't really affect the category $D^b(\cY)$. Let us now explain why. Suppose we write $\cY_1$ for the quotient stack
  $$\cY_1= \Stack{\Hom(V, Q) }{ \GSp(Q)\times \C^*_R }$$
  where the $\C^*_R$ acts by scaling with weight 1, as in Section \ref{sec:statement}, and write $\cY_2$ for the same quotient stack but with the $\C^*_R$ acting trivially.  Recall that $\Delta\subset \GSp(Q)$ denotes the diagonal 1-parameter subgroup. There then is an isomorphism $f:\cY_{2} \to \cY_{1}$ induced by the map of groups
  \[
    \begin{pmatrix}
      \id & \Delta \\
      0 & \id 
    \end{pmatrix} : \GSp(Q)\times \C^{*}_{R} \to \GSp(Q) \times \C^{*}_{R},
  \]
  and $f_{*}$ induces an equivalence $D^{b}(\cY_{2}) \isoto D^{b}(\cY_{1})$. The only difference between $\cY_1$ and $\cY_2$ is how the line bundles are labelled, because $f_*\langle \beta_Q \rangle = \langle \beta_Q\rangle[2]$. 
  
  To construct the algebra $B$ we need to use the map $\cY \to [\Wedge^2 V^\vee / \C^*]$. If we want the construction to be precisely analogous to the construction of $A$ then we need to make the $\C^*_R$-action on both $\cY$ and $\Wedge^2 V$ trivial, \emph{i.e.}~the result is really that $\Br(\cY_2)$ is equivalent to $D^b(B)$.  Alternatively we could put a non-trivial R-charge action on both $\cY$ and $\Wedge^2 V^\vee$, as we do in Section \ref{sec:statement}. Then we get a similar equivalence $\Br(\cY_1)\cong D^b(B)$, but here the algebra $B$ carries a non-trivial $\C^*_R$ action, \emph{i.e.} it is not concentrated in homological degree zero.

  There is only one point in our proof where this distinction will be important: Lemma \ref{lem:DegreesOfA'}. In particular we will need to use the Serre functor on $\Br(\cY_1)$, and the explicit expression for the Serre functor depends on how we label our line bundles.  
\end{rem}

\begin{prop}\label{thm:SerreDuality2}
  Equip $\cY$ with the non-trivial R-charge action discussed above. Let $\cE, \cF \in \Br(\cY)$ with $\cF|_{\cY^{\ss}} = 0$. Then:
  \[
    \Hom(\cE,\cF) = \Hom(\cF, \cE \otimes (\det Q)^{-v}[\dim \widetilde{\Pf_q}-2qv])^{\vee}
  \]
\end{prop}
Note that since $\widetilde{\Pf_q} = \Hom(V, Q) /\Sp(Q)$ the shift here is actually negative; it's:
$$\dim \widetilde{\Pf_q}-2qv= -\dim \Sp(Q) = -{2q+1\choose 2}$$

\begin{proof}
  Under the equivalence $f_*$ from Remark \ref{rem:Rcharge2} the line bundle $(\det Q)^{-v}[\dim \widetilde{\Pf_q}]$ on $\cY_{2}$ becomes the line bundle $(\det Q)^{-v}[\dim \widetilde{\Pf_q} - 2qv]$. on $\cY_1$. Now the result follows by the analogue of Proposition \ref{thm:SerreDuality1}.
\end{proof}

\subsection{The kernel}\label{sec:Kernel}

In this section we construct the Fourier--Mukai kernel for a functor:
$$ D^b(\cX) \To D^b(\cY\tms_{\C^*}\Wedge^2 V,\,W)$$
This functor will induce the equivalence of Theorem \ref{thm:HoriDualityWithL} in the case $L=0$. 

\subsubsection{The definition of the kernel}\label{sec:definingK}

As stated above, our kernel will be a matrix factorization living on the relative product of $\cX$ and $\cY\tms_{\C^*}\Wedge^2 V$ over the base $[\Wedge^2 V/ \C^*]$. 

Recall that $\langle\beta_S\rangle$ is the character of $\GSp(S)$ contained in $\Wedge^2 S$. 
We let $\GSp(S, Q) \subset \GSp(S) \tms \GSp(Q)$ denote the kernel of the character $\langle \beta_S \rangle^{-1} \langle \beta_Q \rangle$.
The relative product can then be described as 
$$ \cX \tms_{\C^*} \cY = \Stack {\Hom(S, V)\tms \Hom(V, Q)}{\GSp(S, Q)} $$

Notice that $\cX \tms_{\C^*} \cY$ admits a map to the stack
$$\cZ = \Stack{\Hom(S, Q)}{\GSp(S,Q)} $$
by composing the two factors. We denote this map by:
$$\psi: \cX \tms_{\C^*} \cY \To \cZ $$
There's an obvious superpotential $W\in \Gamma(\cO_{\cZ})$ which sends $z\in \Hom(S, Q)$ to 
$$W(z) = \omega_Q\big(\!\wedge^2 \! z (\beta_S) \big) $$
where as before $\omega_Q$ is the symplectic form on $Q$ and $\beta_S$ is the Poisson bivector on $S$. If we pull this up via $\psi$ we obtain the superpotential we already have, \emph{i.e.} it agrees with the pull-up of the  function $W$ on $\cY\tms_{\C^*}\Wedge^2 V$ \eqref{eq:W}, so it seems reasonable to denote all of them by $W$.

We also give $\cZ$ an R-charge by letting $\C^*_R$ act with weight 1 on the underlying vector space; this makes $(\cZ, W)$ a Landau--Ginzburg B-model, and the map $\psi$ R-charge equivariant.

Our kernel comes from the stack $\cZ$, that is it is the pull-up of an object $\cK \in D^b(\cZ,  W) $. The Landau--Ginzburg model $(\cZ, W)$ is very simple, it is just a vector space with a non-degenerate quadratic superpotential, plus a group action. The following form of Kn\"orrer periodicity applies:

\begin{lem}\label{lem:eqKP}
  There is an object $\cK \in D^b(\cZ,  W) $ such that the functor
  $$\hom(\cK, - ) : D^b(\cZ,  W) \To D^b\left(\Stack{\pt}{\GSp(S,Q)} \right) $$
  is an equivalence, sending $\cK$ to $\cO_{pt}$. 
\end{lem}
If we forget the group action then this statement is basic Kn\"orrer periodicity.  \emph{A priori} the equivalence might not hold equivariantly because there could be a Brauer class obstruction (or Brauer--Wall class to be precise).\footnote{For example consider the LG model $(\A^2, xy)$ with a $\Z_2$ action that swaps $x$ and $y$ (for this to make sense we must   forget about R-charge, so all the categories are only $\Z_2$-graded). The category $D^b([\A^2/\Z_2], xy)$ is \emph{not} equivalent to $D^b(B\Z_2)$ because there is no $\Z_2$-equivariant object that can implement the equivalence.}  For comparison Kn\"orrer periodity does not hold for a general quadratic vector bundle over a scheme; here we are considering a quadratic vector bundle over the stack $B\GSp(S,Q)$. 

It may be possible to prove that the relevant cohomology group of $\GSp(S, Q)$ vanishes and hence there can be no obstruction; instead we shall appeal to an explicit construction from earlier work of the second author. 
\begin{proof}
  In \cite[Prop.\ 4.6]{ST_2014} it was shown that there is an object $\cK\in D^b(\cZ,  W)$ such that when we forget the group action $\cK$ becomes equivalent to the sky-scraper sheaf along a maximal-isotropic subspace in $\Hom(S, Q)$. Since this sky-scraper sheaf induces the non-equivariant equivalence, it follows easily that $\cK$ induces the equivariant equivalence.
\end{proof}

This lemma says that all objects of the category $D^b(\cZ, W)$ can be obtained by tensoring $\cK$ with (a shifted sum of) $\GSp(S,Q)$-representations.
So picking 
$$\psi^*\cK \;\; \in D^b(\cX \tms_{\C^*} \cY,\, W) $$
as our Fourier--Mukai kernel seems like a natural choice. It defines a functor:
$$\widehat{\Phi}: D^b(\cX) \to D^b(\cY\tms_{\C^*}\Wedge^2 V,\,W)$$

We shall see shortly that it in fact defines a functor:
$$\Phi: \Br(\cX) \to \Br(\cY\tms_{\C^*}\Wedge^2 V,\,W) $$

\subsubsection{An important property of $\cK$}

We will need the following fact about $\cK$, which begins to make the duality manifest. Recall that at the most basic combinatorial level, our duality is the bijection $\gamma \mapsto \gamma^{c\top}$ \eqref{eq:combinatorialduality} between Young diagrams in $Y_{s,q}$ and $Y_{q,s}$. 

If we restrict $\cK$ to the origin in $\Hom(S, Q)$ and take its homology we obtain a representation of $\GSp(S, Q)$, which has an underlying $\Sp(S)\tms \Sp(Q)$-representation.

\begin{prop}\label{prop:IrrepsInKernel}
  The vector space $h_\bullet(\cK|_0)$ is concentrated in degree zero, and as an $\Sp(S)\tms \Sp(Q)$ representation we have:
  $$h_\bullet(\cK|_0) = \bigoplus_{\gamma\in Y_{s,q}} \Schpur{\gamma} S \otimes \Schpur{\gamma^{c\top}} Q $$
\end{prop}
Recall that it in our notation $\Schpur{\gamma}$ denotes a symplectic Schur functor, and $\Schur{\gamma}$ denotes an ordinary (GL) Schur functor. 
\begin{proof}
  The object $\cK$ was constructed to give an equivalence
  $$\hom(\cK, - ):  D^b(\cZ ,W)\isoto D^b(\Stack{\pt}{\GSp(S, Q)}) $$
  under which $\cK$ maps to the 1-dimensional trivial representation (in degree zero), see Lemma \ref{lem:eqKP}. Applying this functor to the sky-scraper sheaf $\cO_0$ of the origin gives some graded representation $R$, and then we must have $\cO_0\cong \cK\otimes R$. Hence $\hom(\cO_0, \cK) = R^\vee$  and  $\hom(\cO_0, \cO_0)=R\otimes R^\vee$. 

  However, $\cO_0$ is also equivalent to a matrix factorization given by taking the Koszul complex and perturbing it in the standard way. We can use this matrix factorization to compute $\hom(\cO_0, \cO_0)$ and we obtain:
  $$R\otimes R^\vee = \hom(\cO_0, \cO_0) = \bigoplus_{i} \big(\Wedge^i \Hom(S,Q)\big)[i]$$
  Since the vector space $\Hom(S, Q)$ already has R-charge 1 this vector space is actually concentrated in degree zero, so the same is true of $R$. We can also use the matrix factorization to compute $\hom(\cO_0, \cK)$, and we observe that
  $$R^\vee = \hom(\cO_0, \cK) = h_\bullet( \cK|_0) \otimes \det \Hom(S,Q)[4sq] $$
  because the perturbed Koszul complex is almost self dual. The determinant factor here is a non-trivial character of $\GSp(S,Q)$ but it is concentrated in degree zero. Therefore as a graded representation of $\Sp(S)\times \Sp(Q)$ we have that
  $$h_\bullet(\cK|_0) = R^\vee = R$$
  (all representations of this group are self-dual) for a representation $R$ satisying the equality:
  $$R^{\otimes 2} = \Wedge^{\bullet}(S \otimes Q)$$
  This equality determines $R$ uniquely, and the remainder of the argument is a computation in the representation ring of $\Sp(S)\times \Sp(Q)$.
  \pgap

  We start by computing the character of $R$. Let $x_{1}^{\pm 1}, \ldots, x_{s}^{\pm 1}$ and  $y_{1}^{\pm 1}, \ldots, y_{q}^{\pm 1}$ denote the standard characters of the maximal tori in $\Sp(S)$ and $\Sp(Q)$ respectively, so that the characters of the standard representations of $S$ and $Q$ are $\sum_{i=1}^{s} x_{i} + x_{i}^{-1}$ and  $\sum_{j = 1}^q y_j +y_j^{-1}$. The character of $S\otimes Q$ is then:
  $$\sum_{i,j} (x_i +x_i^{-1})(y_j + y_j^{-1}) $$
  We claim that the character of $\Wedge^\bullet(S\otimes Q)$ can be expressed as:
  $$\prod_{i,j} (x_{i}+x_{i}^{-1} + y_{j} + y_{j}^{-1})^2 $$
  This is an easy computation if $s = q = 1$, and then the general case follows from the fact that $\Wedge^{\bullet}$ converts sums to products. 

  Consequently, the character of $R$ is:
  $$
  \prod_{i,j} (x_{i}+x_{i}^{-1} + y_{j} + y_{j}^{-1})
  $$
  All monomials appearing in the the above expression have total degree $\le sq$. To get a monomial of degree exactly $sq$, we choose a subset $\Gamma$ of the rectangle $[1,s]\tms [1,q]$, then there is a corresponding monomial :
  $$\left(\prod_{(i,j)\in \Gamma} x_i \right)\left(\prod_{(i,j)\notin \Gamma} y_j\right)$$
  Now choose a partition $\gamma\in Y_{s,q}$, \emph{i.e.} a non-increasing sequence $(\gamma_1, \gamma_2, \ldots, \gamma_s)$ with $\gamma_1\leq q$. We can define an associated subset
  $$\Gamma = \set{ (i,j),\; q-\gamma_i <  j \leq q } \; \subset [1, s]\tms [1, q] $$
  and the corresponding monomial is
  \begin{equation*}
    \label{eqn:HighestWeightElement}
    m = \big( x_1^{\gamma_1}x_2^{\gamma_2}\cdots x_s^{\gamma_s}\big)\big(y_1^{\beta_1}y_2^{\beta_2}\cdots y_q^{\beta_q}\big) 
  \end{equation*}
  where $\beta=(\beta_1,\ldots, \beta_q)$ is the partition $\beta = \gamma^{c\top}$. 
  This is a dominant weight of $\Sp(S)\tms \Sp(Q)$, which we claim is the highest weight of a subrepresentation of $R$.

  To see this, we show that $m$ is a maximal element among the weights of $R$, in the standard partial ordering of the weight lattice $X(\Sp(S) \times \Sp(Q))$.
  Recall that the partial ordering on $X(\Sp(S))$ (and similarly on $X(\Sp(Q))$) is such that $\prod x_i^{\gamma_i} \le \prod x_i^{\gamma_i^{\prime}}$ if and only if $\sum_{i=1}^{k} \gamma_{i} \le \sum_{i=1}^{k} \gamma_{i}^{\prime}$ for all $k \in [1,s]$.

  We assume for a contradiction that there exists a monomial $\prod x_i^{\gamma^{\prime}_i} \prod y_i^{\beta^{\prime}_i}$ among the weights of $R$ such that
  \begin{equation}
    \label{eqn:AdAbsurdum}
    \prod x_{i}^{\gamma_{i}} \prod y_{i}^{\beta_{i}} < \prod x_i^{\gamma^{\prime}_i} \prod y_i^{\beta^{\prime}_i}
  \end{equation}
  We must then have 
  \[
    sq = \sum \gamma_{i} + \sum \beta_{i} \le \sum \gamma^{\prime}_{i} + \sum \beta^{\prime}_{i} \le sq,
  \]
  and so $\sum \gamma^{\prime}_{i} + \sum \beta^{\prime}_{i} = sq$.
  Then $\prod x_i^{\gamma^{\prime}_i} \prod y_i^{\beta^{\prime}_i}$ must arise from a set $\Gamma^{\prime} \subset [1,s]\times [1,q]$ as above.

  Using this description, it is easy to see that for a fixed $\gamma^{\prime}_{i}$ the monomial $\prod x_{i}^{\gamma_{i}^{\prime}} \prod y_{i}^{\beta_{i}^{\prime}}$ is maximal (i.e.\ any other choice of $\beta^{\prime}$ is strictly smaller) when $\beta^{\prime} = (\beta_{1}^{\prime},\ldots, \beta_{q}^{\prime}) = (\gamma^{\prime})^{c\top}$, so we may assume that this is the case.

  Let now $k \in [1,s]$ be the smallest integer such that $\gamma_{k} < \gamma_{k}^{\prime}$.
  Then we must have $\beta_{i} = \beta_{i}^{\prime}$ for $i = 1, \ldots, q-\gamma_{k}^{\prime}$, and $\beta_{q-\gamma_{k}^{\prime}+1}^{\prime} < \beta_{q-\gamma_{k}^{\prime}+1}$, which contradicts \eqref{eqn:AdAbsurdum}.
  This proves that $R$ contains the irrep $\Schpur{\gamma} S\otimes \Schpur{\gamma^{c\top}} Q $. 

  It remains to show that no other irreps occur in $R$. Let $N$ be the number of irreps appearing in $R$ (with multiplicities). We know that $N \ge |Y_{s,q}|$, and we want to show that this is an equality. Since $\Hom(R, R) = R^{\otimes 2}  = \Wedge^{\bullet}(S \otimes Q)$, we have $N = \dim \Wedge^{\bullet}(S\otimes Q)^{\Sp(S)\tms\Sp(Q)}$. A standard computation of  Littlewood--Richardson coefficients \cite[p.\ 80]{fulton_representation_1991} shows that as a $\GL(S) \tms \GL(Q)$-representation
  we have:
  \[
    \Wedge^{\bullet}(S \otimes Q) = \bigoplus_{\alpha} \Schur{\alpha}S \otimes \Schur{\alpha^\top}\!Q
  \]
  By Lemma \ref{lem:SpInvariants} below, and its analogue for $\Sp(Q)$, the dimension $N$ of the space of $\Sp(S)\tms \Sp(Q)$-invariants equals the number of partitions $\alpha=(\alpha_1, \ldots, \alpha_s)$ such that:
  \begin{itemize}
  \item Each $\alpha_{i}$ is even (so $(\Schur{\alpha^{\top}}\!Q)^{\Sp(Q)} = \C$).
  \item Each number appears an even number of times in $\alpha$ (so $\left(\Schur{\alpha}S\right)^{\Sp(S)} = \C$).
  \item $\alpha_{1} \le 2q$.
  \end{itemize}
  Mapping each such partition $(\alpha_i)$ to the partition $(\gamma_i) = (\onehalf\alpha_{2i})$ gives a bijection onto the set $Y_{s,q}$, so $N =  |Y_{s,q}|$.
\end{proof}

\begin{lem}\label{lem:SpInvariants}\cite[Cor. 12.5]{sundarem_combinatorics_1986}
  Let $\alpha$ be a partition of length $\leq 2s$. The space of $\Sp(S)$ invariants in $\Schur{\alpha} S$ is 1-dimensional if each entry in $\alpha$ occurs an even number of times -- equivalently, if each entry in $\alpha^\top$ is even -- and zero-dimensional otherwise.
\end{lem}

Applying Lemma \ref{lem:minimalmodels} we immediately get:

\begin{cor}\label{cor:resolutionofK}  $\cK$ is equivalent to a matrix factorization whose underlying vector bundle is
  $$\bigoplus_{\gamma \in Y_{s,q}} \Schpur{\gamma,\, n_\gamma} S\otimes \Schpur{\gamma^{c\top}\!\!,\, m_\gamma} Q $$
  for some set of integers $\{n_\gamma\}$ and $\{m_\gamma\}$. 
\end{cor}
We don't need to know the value of the integers $n_\gamma$ and $m_\gamma$, the point is that we know the vector bundle in this corollary up to twisting each summand by a line bundle (\emph{i.e.} tensoring the representation by a character of $\GSp(S,Q)$). Note however that there is no ambiguity about the R-charge; the summands may be twisted by line bundles, but not shifted.

\begin{rem}\label{rem:ExplicitDifferentials}
  It would be nice to have an explicit construction of the differentials in this matrix factorization, but we have not managed to find this.
\end{rem}

Now we take the object $\phi^*\cK \in D^b(\cX\tms_{\C^*} \cY, W)$, and consider the associated functor
$$\widehat{\Phi} = (\pi_{2})_{*}(\pi_{1}^{*}(-) \otimes \psi^*\cK): \;D^{b}(\cX) \;\To\; D^{b}(\cY \tms_{\C^*} \Wedge^{2} V, \,W) $$
where $\pi_1$ and $\pi_2$ denote the projection maps from $\cX\tms_{\C^*} \cY$ to the two factors $\cX$ and $\cY\tms_{\C^*}\Wedge^2 V$.

\begin{cor}\label{cor:PhitoBr}
  The image of $\widehat{\Phi}$ lies inside $\Br(\cY\tms\Wedge^2 V,\,W)$.
\end{cor}
\begin{proof}
  From Corollary \ref{cor:resolutionofK}, the image under $\widehat{\Phi}$ of any object is equivalent to an infinite-rank matrix factorization which satisfies the necessary grade-restriction-rule.
\end{proof}

If we let $\Phi$ denote the restriction of $\widehat{\Phi}$ to the subcategory $\Br(\cX)$, it follows that $\Phi$ defines a functor:
$$\Phi: \Br(\cX) \to \Br(\cY\tms\Wedge^2 V,\,W)$$

\subsubsection{Adjoints}\label{sec:adjoints} We'll also need to understand the adjoint to this functor $\Phi$. To get a functor from $D^b(\cY\tms_{\C^*} \Wedge^2 V,\, W)$ to $D^b(\cX)$ we need a kernel which is an object in $D^b(\cX\tms_{\C^*} \cY, \, -W)$. The obvious guess for the adjoint to $\Phi$ is to use the kernel $\psi^*\cK^\vee$, up to some line bundle and shift. This is correct, but the proof is a little involved.

Recall that $\pi_2$ is the projection from $\cX\tms_{\C^*} \cY$ to $\cY\tms_{\C^*} \Wedge^2 V$.
Define a subcategory $\Br(\cX\tms_{\C^*} \cY,\, W) \subset D^b(\cX\tms_{\C^*} \cY,\, W) $ by enforcing both the $\Sp(S)$ and $\Sp(Q)$ grade-restriction rules. We have a functor $(\pi_2)_*: \Br(\cX\tms_{\C^*} \cY,\, W) \to \Br(\cY\tms_{\C^*} \Wedge^2 V,\, W)$ and its left adjoint is $\pi_2^*$. We define:
\[
  \pi_{2}^{!} = \pi_{2}^{*}\left(sv-\binom{v}{2}\right)\left[-\binom{v-2s}{2}\right]
\]
\begin{lem}\label{lem:pi2adjunction}
  The functor
  $$\pi_{2}^{!} : \Br(\cY\tms_{\C^*} \Wedge^2 V,\, W)\;\To\; \Br(\cX\tms_{\C^*} \cY,\, W) $$
  is the right adjoint to $(\pi_2)_*$.
\end{lem}
\begin{proof}
  To see this, we pass to the description of the brane categories in terms of algebras as in Section   \ref{sec:Br(X)}. The category $\Br(\cX\tms_{\C^*} \cY,\, W)$ is equivalent to the derived category of the curved algebra $(A\otimes B, W)$. The functor $(\pi_{2})_{*}$ becomes a functor:
  \begin{equation}\label{eq:pi2*}
    D^{b}(A \otimes B, \,W) \To D^{b}(\cO_{\wedge^{2} V} \otimes B,\, W),
  \end{equation}
  To understand this functor, observe that pushing down from $\widetilde{\cX}$ to $\Wedge^2 V$ induces the functor
  $$ \Hom(\widetilde{P}_\phi, - ): D^b(\widetilde{A}) \To D^b(\cO_{\wedge^2 V})$$
  where $\widetilde{P}_\phi$ is the projective $\widetilde{A}$-module corresponding to the trivial line-bundle on $\widetilde{\cX}$, \emph{i.e.} the vertex  projective module corresponding to the empty Young diagram $\phi\in Y_{s,q}$. Similarly pushing down from $\cX$ to $[\Wedge^2 V / \C^*]$ is the functor 
  $$\Hom(P_\phi, -): D^b(A) \To D^b(\cO_{\wedge^2 V})$$
  (here $\cO_{\wedge^2 V}$ is a graded ring) where again $P_\phi$ is a vertex projective. So if we write $P$ for the project $A\otimes B$-module $P:= P_\phi\otimes B$ then \eqref{eq:pi2*} is the functor $\Hom( P, - )$.

  The functor $\pi_{2}^{!}$ becomes $P \otimes_{\cO_{\wedge^{2} V} \otimes B} -$, up to twists.

  In order to apply results from the literature, let us instead consider the functors
  \[
    \phi^* := (A \otimes B) \otimes_{\cO_{\wedge^{2} V} \otimes B} - : \;D^{b}(\cO_{\wedge^{2} V} \otimes B,\, W) \to D^{b}(A \otimes B,\, W)
  \]
  and:
  \[
    \phi_{*} := \Hom(A \otimes B, -) : \;D^{b}(A \otimes B,\, W) \;\To \;D^{b}(\cO_{\wedge^{2} V} \otimes B,\, W)
  \]
  We claim that some twist of $\phi^*$ is the right adjoint to $\phi_*$. Since $P$ is a summand of $A \otimes B$, the claim that $\pi_{2}^{!}$ is the right adjoint to $(\pi_{2})_{*}$ follows from this.

  By \cite[Ex.\ 6.4]{yekutieli_dualizing_2006}, $A \otimes B$ admits a rigid dualising complex $\omega_{A \otimes B}$, and since $A \otimes B$ is a self-dual maximal Cohen--Macaulay module over a Gorenstein ring, it is easy to check that $\omega_{A \otimes B} \cong A \otimes B$ up to twist and shift.
  We define the dualising functor $\mathbb{D}_{A \otimes B} = \hom(-, \omega_{A \otimes B})$.

  Similarly, we get a dualising complex $\omega_{\cO_{\wedge^{2} V} \otimes B}$ for $\cO_{\wedge^{2} V} \otimes B$, and an induced duality functor $\mathbb{D}_{\cO_{\wedge^{2} V} \otimes B}$.
  We now define
  \[
    \phi^{!} = \mathbb{D_{A \otimes B}} \circ \phi^{*} \circ \mathbb{D}_{\cO_{\wedge^{2} V} \otimes B},
  \]
  and see that up to twist $\phi^{!} \cong \phi^{*}$.

  Let us first ignore the presence of the superpotential and prove that $\phi^{!}$ is right adjoint to $\phi_{*}$ when $W = 0$. In this case by \cite[Thm.\ 4.13]{yekutieli_rigid_2004} we have 
  \[
    \phi^{!} \cong \phi^{\flat} := \hom(A \otimes B, -) : D^{b}(\cO_{\wedge^{2} V} \otimes B) \to D^{b}(A \otimes B).
  \]
  It is shown in \cite{yekutieli_rigid_2004} that a nondegenerate trace map $\phi_{*}\phi^{\flat} \to \id$ exists, and arguing as in \cite[Sec.\ III.6]{hartshorne_residues_1966} the adjointness of $\phi^{!}$ and $\phi_{*}$ follows.

  We now turn to the case $W \not= 0$.
  Note first that the counit map $\phi_{*}\phi^{!} \to \id$ is induced by a map of Fourier--Mukai kernels supported along the diagonal in (the centre of) $(\cO_{\wedge^{2} V} \otimes B) \otimes (\cO_{\wedge^{2} V} \otimes B)^{op}$.
  Hence we get a counit map in the case with superpotential as well.

  Given $\cE \in D^{b}(A \otimes B,\,W)$ and $\cF \in D^{b}(\cO_{\wedge^{2} V} \otimes B,\,W)$, we have now shown that the induced map
  \[
    \Hom(\cE, \pi^{!}\cF) \to \Hom(\pi_{*}\cE, \pi_{*}\pi^{!} \cF) \to \Hom(\pi_{*}\cE, \cF)
  \]
  is an isomorphism if $W = 0$.
  For general $W$, we may degenerate both $\cE$ and $\cF$ to honest complexes, and then the upper semicontinuity of cohomology implies that the map is an isomorphism in the general case as well.
\end{proof}

\begin{prop}\label{prop:Phidagger}
  The composition
  $$(\pi_1)_*\hom\big(\psi^*\cK, \pi_2^!(-)\big) $$
  defines a functor
  $$\Phi^\dagger:  \Br(\cY\tms_{\C^*}\Wedge^2 V, \, W) \; \To\; \Br(\cX) $$
  and this is the right adjoint to $\Phi$.
\end{prop}
\begin{proof}
  Using Corollary \ref{cor:resolutionofK} it's clear that we have a functor
  \beq{eq:afunctor}
  \hom\big(\psi^*\cK, \pi_2^!(-)\big):  \Br(\cY\tms_{\C^*}\Wedge^2 V, \, W) \; \To\; \Br(\cX\tms_{\C^*}\cY) \eeq
  and Lemma \ref{lem:pi2adjunction} shows that this is the right adjoint to the functor:
  $$(\pi_2)_*(\cK\otimes -)  $$

  Now we want to compose this with push-down along the projection 
  $$\pi_1: \cX\times_{\C^*} \cY \to \cX$$
  but there is a subtlety here. Pushing down a coherent sheaf along $\pi_1$ means taking $\Sp(Q)$-invariants, and the result will be only be a a quasi-coherent sheaf on $\cX$ in general (the fibres are copies of the stack $\widetilde{\cY}$ and the ring of functions on $\widetilde{\cY}$ is infinite-dimensional). However, we claim that if we apply $(\pi_1)_*$ to something in the image of \eqref{eq:afunctor} then we always get a coherent sheaf. If we can prove this claim then the proof of the proposition is complete: since $(\pi_1)_*$ is right-adjoint to $\pi_1^*$ it follows immediately that $\Phi^\dagger$ is the right adjoint to $\Phi$.

  The functor $(\pi_1)_*$ respects the $\Sp(S)$ grade-restriction-rule, or in terms of algebras, it maps $A\otimes B$-modules to (perhaps infinitely-generated) $A$-modules. Now take $\cF \in \Br(\cY \tms_{\C^{*}} \wedge^{2} V,\, W)$, and consider the $A$-module $M = (\pi_1)_*\hom(\psi^*\cK, \pi_2^!\cF)$. We want to know that $M$ is in fact finitely-generated. By adjunction, we have
  \begin{align*}
    M = \Hom_A(A, M) &= \Hom\!\big(T, \,(\pi_1)_*\hom(\psi^*\cK, \pi_2^!\cF)\big)_{gr}\\
                     & = \Hom\!\big( (\pi_2)_*(\cK\otimes \pi_1^*T),\, \cF\big)_{gr}
  \end{align*}
  This last one is a finitely generated graded module over $\cO_{\wedge^{2} V}$. Since the functors are linear over $\Wedge^{2} V$, it follows that $M$ is finitely-generated module over $\cO_{\wedge^{2} V}$ and hence also over $A$. 
\end{proof}

\subsection{The equivalence over the smooth locus}\label{sec:generic}

Since we've constructed our kernel $\psi^*\cK$ relative to the base $[\Wedge^2 V / \C^*]$, we can restrict ourselves to open sets in this base and examine the functor there.

Specifically, we're going to look at the open subset

$$\OS \subset \Wedge^2 V $$
consisting of bivectors having rank at least $2s$. This gives a substack $[\OS / \C^*]\subset [\Wedge^2 V/\C^*]$, which is just a quasi-projective variety  $\P \OS\subset \P(\Wedge^2 V)$. The intersection of $\P \OS$  with the Pfaffian $\Pf_s$ is exactly the smooth locus $\Pf_s^{sm}$, so $\cX|_{\P \OS}$ is equivalent to  $\Pf_s^{sm}$. This brings us very close to the situation considered in \cite{ADS_pfaffian_2015,ST_2014}, so we'll now spend a little time making the connection to the point-of-view of those earlier papers.

On the dual side, restricting $\cY\tms_{\C^*}\Wedge^2 V$ to $\OS$ simply gives $\cY\tms_{\C^*} \OS$. We can think of this a bundle of stacks over the  variety $\P \OS$, with fibres $\widetilde{\cY} = \Stack{\Hom(V,Q)}{\Sp(Q)}$. Locally in $\P \OS$ it's just a vector bundle quotiented by a fibre-wise $\Sp(Q)$ action. This is not true globally, but it is a vector bundle over the stack $[\OS / \GSp(Q)]$, which is a bundle of stacks over  $\P \OS$ with fibres $B\Sp(Q)$. \

Our kernel lives on the  product of $\cX|_{\P \OS}$ and $\cY\tms_{\C^*} \OS$ relative to $\P \OS$, which is simply the restriction of the bundle of stacks $\cY\tms_{\C^*}\OS$ to the subvariety $\Pf_s^{sm}\subset \P \OS$. So we have a functor 
$$D^b(\Pf_s^{sm}) \to D^b(\cY\tms_{\C^*}\OS,\, W) $$
given by pulling up to $\cY\tms_{\C^*} \OS|_{\Pf_s^{sm}}$, tensoring with the object $\psi^*\cK|_{\P \OS}$, and then pushing-forward along the inclusion map.

Now, to any point $a\in \Pf_s^{sm}$ we can associate the image of $a$, this gives a rank $2s$ subbundle
$$ \Sigma \subset V(1) $$
which carries a symplectic form, determined up to scale. There's an associated bundle of stacks over $\Pf_s^{sm}$ whose fibres are $\Stack{\Hom(\Sigma, Q)}{\Sp(Q)}$. To be precise: take the smooth locus $(\widetilde{\Pf_s})^{sm}$ of the affine cone over $\Pf_s$; this is a $\C^*$-bundle over $\Pf_s^{sm}$. We have a vector bundle $\widetilde{\Sigma}$ over  $(\widetilde{\Pf_s})^{sm}$, which is a subbundle of the trivial bundle $V\tms (\widetilde{\Pf_s})^{sm}$. Then we form the stack:
$$\overline{\cZ} = \Stack{\Hom(\widetilde{\Sigma}, Q)}{\GSp(Q)} $$
This is a bundle of stacks over $\Pf_s^{sm}$, or a vector bundle over $[(\widetilde{\Pf_s})^{sm}/\GSp(Q)]$.  We have a quotient map
$$\overline{\psi}:  \cY\tms_{\C^*} \OS|_{\Pf_s^{sm}}   \To \overline{\cZ}$$
and  by construction our kernel is pulled-back from $\overline{\cZ}$. If this is not immediately obvious, observe that we can factor the map 
$$\psi:  \cX\tms_{\C^*}\cY   \To     \cZ =  \Stack{\Hom(S,Q)}{\GSp(S,Q)}$$
through an intermediate step:
$$ \cX\tms_{\C^*}\cY \To \cX\tms_{\GSp(S)}\cZ   =   \Stack{\Hom(S,V)\tms \Hom(S,Q)}{\GSp(S,Q)}$$
Restricting this to $\P \OS$, we obtain the stack $\overline{\cZ}$ and the map $\overline{\psi}$.  We can pull-back the superpotential $W$ to $\overline{\cZ}$, and we have an object
$$\overline{\cK} \in D^b(\overline{\cZ}, W) $$
given by pulling back $\cK$. So over $\P \OS$, our kernel is equal to $\overline{\psi}^*\overline{\cK}$.

The superpotential on $\overline{\cZ}$ gives a non-degenerate quadratic form on each fibre, so Kn\"orrer periodicity should apply unless there's a Brauer class obstruction. However the whole point of the object $\cK$ was that it gives a universal way to implement Kn\"orrer periodicity for this kind of bundle, so indeed there is no Brauer class obstruction, and the object $\overline{\cK}$ induces an equivalence between $D^b(\overline{\cZ}, W)$ and $D^b([(\widetilde{\Pf_s})^{sm}/\GSp(Q)])$. Another point to note is that  if we restrict to open neighbourhoods in $\Pf_s^{sm}$ then there are other ways to implement this equivalence. For example, we may pick  a Lagrangian subbundle $\Lambda\subset \Sigma$ (this is possible in a small-enough neighbourhood), which  induces an $\Sp(Q)$-invariant maximal isotropic subbundle $\Hom(S/\Lambda, Q)$ in $\overline{\cZ}$. The sky-scraper sheaf on this subbundle gives a second object in $D^b(\overline{\cZ}, W)$ which is equivalent to $\overline{\cK}$.
\pgap

Now let's take our subcategories $\Br(\cX)$ and $\Br(\cY\tms_{\C^*}\Wedge^2 V,\,W)$ and base-change them to the open set $\OS$, meaning (as in Section \ref{sec:Br(X)}) that we consider the full triangulated subcategories generated by their images under restriction.

When we do this to $\Br(\cX)$ we get the whole of $D^b(\Pf_s^{sm})$,  because $\Br(\cX) = D^b(A)$ is the derived category of a non-commutative resolution of $\Pf_s^{sm}$. Indeed by Corollary \ref{thm:BranesToModulesBaseChanged} (and Remark \ref{rem:BasechangingBranesToModules}) $\Br(\Pf_s^{sm})$ is the derived category of the sheaf of algebras $A=\End_\cX(T)$,  but the restriction of $T$ to $\cX|_\OS=\Pf_s^{sm}$ is still a vector bundle so this is a trivial Azumaya algebra. However, on the dual side we  get a proper subcategory:
$$\Br(\cY\tms_{\C^*}\OS,\, W) \;\subset \; D^b(\cY\tms_{\C^*} \OS,\, W)$$

Everything in this subcategory satisfies the grade-restriction-rule along the zero section $\P \OS$ --   this follows from Lemma \ref{lem:minimalmodels} and the fact that the grade-restriction rule is preserved under taking mapping cones. In fact this rule exactly characterizes $\Br(\cY\tms_{\C^*}\OS,\, W)$ (see Lemma \ref{lem:pointwise}) but we won't need to use this here. 

We've seen that over the whole of $[\Wedge^2 V/\C^*]$, our kernel induces adjoint functors $\Phi$ and $\Phi^\dagger$  between our $\Br$ subcategories (Corollary \ref{cor:PhitoBr}). If we restrict to $\P \OS$ it follows immediately that we get a  pair of adjoint functors between $D^b(\Pf_s^{sm})$ and the subcategory $\Br(\cY\tms_{\C^*}\OS,\,W)$. Let's denote these functors by $\Phi_\OS$ and 
$\Phi^\dagger_\OS$.

\begin{thm} \label{thm:genericequivalence} These functors give an equivalence
  $$\Phi_\OS: D^b(\Pf_s^{sm}) \isoto \Br(\cY\tms_{\C^*}\OS,\,W)$$
  and its inverse.
\end{thm}
Most of this theorem was proven in \cite{ADS_pfaffian_2015,ST_2014}. We'll give a complete proof here, partly for convenience, and also because those previous papers did not prove the essential-surjectivity of $\Phi_\OS$. We'll get the result as a corollary of the Lemma \ref{lem:formalnhdresult} below.

Fix a point $a\in \P \OS$, and let $\widehat{\P \OS}_a$ denote the formal neighbourhood of $a$; this is isomorphic to a formal neighbourhood of 0 in $\Wedge^2 V/\langle a \rangle$.  In the case that $a$  lies in $\Pf_s$, we write $\widehat{(\Pf_s)}_a$ for the formal neighbourhood of $\Pf_s$ at $a$. 

If we restrict $\cY\tms_{\C^*} \OS $ to the formal neighbourhood $\widehat{\P \OS}_a$ it becomes:
$$\widetilde{\cY}\tms \widehat{\P \OS}_a \;= \; \Stack{\Hom(V,Q)}{\Sp(Q)}\tms    \widehat{\P \OS}_a$$
We define a subcategory
$$\Br(\widetilde{\cY}\tms\widehat{\P \OS}_a,\, W) \;\subset \ D^b(\widetilde{\cY}\tms \widehat{\P \OS}_a, \,W) $$
as the full subcategory of objects $\cE$ such that all $\Sp(Q)$-irreps occuring in $h_\bullet(\cE|_{(0,a)})$ come from the set $Y_{q,s}$.  The restriction of an object in $\Br(\cY\tms_{\C^*}\OS,\,W)$ lands in this subcategory, and presumably these objects generate the category. If that is true then it follows formally that $\Phi$ and $\Phi^\dagger$ induce adjoint functors between $ D^b(\widehat{(\Pf_s)}_a)$ and $\Br(\widetilde{\cY}\tms \widehat{\P \OS}_a,\, W)$; however this can be proven directly using the arguments of Corollary \ref{cor:PhitoBr}.

\begin{lem}\label{lem:formalnhdresult}\noindent
  \begin{enumerate}
  \item If $a\in \Pf_s^{\sm}$ then $\Phi$ induces an equivalence:
    $$ D^b(\widehat{(\Pf_s)}_a) \isoto  \Br(\widetilde{\cY}\tms \widehat{\P \OS}_a,\, W)$$ 
  \item If $a\notin \Pf_s$ then $\Br(\widetilde{\cY}\tms \widehat{\P \OS}_a,\, W)$ is zero.
  \end{enumerate}
\end{lem}
\begin{proof}
  Let the rank of $a$ be $2t$, so $a$ lies in the smooth locus of the Pfaffian variety $\Pf_t$. Along $\Pf_t$ there is a rank $2t$ bundle $\Sigma\subset V(1)$ given by the images of the 2-forms at each point, and a rank $v-2t$ bundle $K\subset V^\vee$ given by the kernels. In the formal neighbourhood $\widehat{\P \OS}_a$ we can do a gauge transformation to trivialize the family of 2-forms on the bundle $V\tms\Pf_t$; then both $\Sigma$ and $K$ become trivial bundles with fibre the image and kernel of $a$. The fibre of normal bundle to $\Pf_t$ at $a$ is $\Wedge^2 K^\vee$, so after a formal change of co-ordinates:
  $$\widehat{\P \OS}_a \cong  \widehat{(\Wedge^2 K^\vee)}_0\tms \widehat{(\Pf_t)}_a$$
  If we choose a splitting $V^\vee = K\oplus \Sigma$,  we can write:
  $$\widetilde{\cY}\tms \widehat{\P \OS}_a = \Stack{\Hom(\Sigma, Q)}{\Sp(Q)}\tms \Stack{K\otimes Q}{\Sp(Q)}\tms  \widehat{(\Wedge^2 K^\vee)}_0\tms \widehat{(\Pf_t)}_a $$
  The superpotential is now a a sum of a quadratic $W_q$ and  a cubic term $W_c$. The quadratic term is the tautological superpotential on the first factor, and the cubic term  is the tautological superpotential on $\Stack{K\otimes Q}{\Sp(Q)}\tms  \widehat{(\Wedge^2 K^\vee)}_0 $. We can use  Kn\"orrer periodicity to remove the quadratic term, and get an equivalence:
  $$\Psi:  D^b(\Stack{K\otimes Q}{\Sp(Q)}\tms  \widehat{\P \OS}_a, \, W_c) \;\; \isoto\;\; D^b(\widetilde{\cY}\tms  \widehat{\P \OS}_a,\, W) $$
  We can do this using  Lemma \ref{lem:eqKP}, which provides (forgetting the $\GSp(S)$ action) an object $\cL\in D^b([\Hom(\Sigma, Q)/ \Sp(Q)],\, W_q)$; then the equivalence is given by pulling-up and tensoring with $\cL$. Alternatively we may pick a Lagrangian $\Lambda\subset \Sigma$ and use the sky-scraper sheaf along $\Hom(\Sigma/\Lambda, Q)$. 

  Let's examine this construction in the case that $a\in \Pf_s$. Within the formal neighbourhood the functor $\Phi$ is given by pulling-up to 
  $$[\Hom(V, Q) / \Sp(Q)]\tms\widehat{(\Pf_s)}_a$$
  followed by tensoring with the object $\overline{\psi}^*\overline{\cK}$, and then pushing-forward into:
  $$[\Hom(V, Q) / \Sp(Q)]\tms\widehat{\P \OS}_a$$
  We can form a commutative diagram:
  $$\begin{tikzcd} D^b\!\big(\Stack{\Hom(V, Q)}{ \Sp(Q)} \tms\widehat{(\Pf_s)}_a,\, W\big) \ar{r}  
    &  
    D^b\!\big(\Stack{\Hom(V, Q) }{\Sp(Q)} \tms\widehat{\P \OS}_a,\, W\big)  \\
    D^b\!\big(\Stack{ K\otimes Q }{ \Sp(Q)} \tms\widehat{(\Pf_s)}_a \big) \ar{r} \ar{u}{\Psi|_{\widehat{(\Pf_s)}_a}} & 
    D^b\!\big( \Stack{K\otimes Q }{ \Sp(Q)}\tms\widehat{\P \OS}_a, \, W_c\big) \ar{u}{\Psi} \\
    D^b\!\big(\widehat{(\Pf_s)}_a\big) \ar{u} & \,
  \end{tikzcd} $$
  Here the horizontal arrows are just push-foward along the inclusion maps. Note that the top row is natural, as is the composition of the two vertical arrows on the left; however the middle row depends on our choices of co-ordinates and splittings. Moving from the bottom left corner to the top right corner gives the functor $\Phi$, 
  because the object $\overline{\psi}^*\overline{\cK}$ is exactly the restriction of the object  $\cL$.
  \pgap

  Next we need to examine what happens to the subcategory $\Br(\widetilde{\cY}\tms \widehat{\P \OS}_a,\, W)$ under the equivalence $\Psi$, in both cases $a\in \Pf_s$ and $a\notin\Pf_s$. Given an object 
  $$\cE\in  D^b(\Stack{K\otimes Q}{\Sp(Q)}\tms  \widehat{\P \OS}_a, \, W_c)$$
  to find the representation  $h_\bullet \big(\Psi\cE |_{(0,a)}\big)$ we take $h_\bullet(\cE|_{(0,a)})$ and tensor it with  $h_\bullet (\cL|_0)$. By Proposition \ref{prop:IrrepsInKernel}, the $\Sp(Q)$-irreps that occur in $h_\bullet (\cL|_0)$ are exactly those of the form $\Schpur{\delta} Q$ where the width of $\delta$ is at most $t$ (one can also reach this conclusion by taking the Koszul resolution of $\cO_{\Hom(\Sigma/\Lambda, Q)}$).

  Now suppose that $a\notin \Pf_s$, so $t>s$.  If $\cE$ is such that  $h_\bullet(\cE|_{(0,a)})\neq 0$ then $\Psi \cE$ does not lie in the subcategory $\Br(\widetilde{\cY}\tms\widehat{\P \OS}_a,\, W)$, but $h_\bullet(\cE|_{(0,a)})= 0$ implies that $\cE\simeq 0$  by Lemma \ref{lem:minimalmodels}. This proves part (2) of the lemma.

  Suppose instead that $a\in \Pf_s$, so $t = s$. If $\Psi\cE \in \Br(\widetilde{\cY}\tms\widehat{\OS}_a,\, W)$ then we must have that $h_\bullet(\cE|_{(0,a)})$ is a trivial $\Sp(Q)$-representation, which means (by Lemma \ref{lem:minimalmodels} again) that $\cE$ is equivalent to a matrix factorization whose underlying vector bundle is equivariantly trivial.  The category of such matrix factorizations is exactly the same as the category of matrix factorizations on the underlying scheme, \emph{i.e.} the scheme-theoretic quotient of  $K\otimes Q$ by $\Sp(Q)$. This quotient is itself a Pfaffian, it's the rank $2q$ locus in  $\Wedge^2 K$. But since $a\in \Pf_s$ we have $\dim K = 2q+1$, so in fact the quotient is the whole of $\Wedge^2 K$. So in this case, $\Psi$ induces an equivalence:
  $$\Br(\widetilde{\cY}\tms\widehat{\P \OS}_a,\, W)\;\;\isoto\;\; D^b(\Wedge^2 K \tms \widehat{\OS}_a, \,  W_c) $$
  Note that once we've quotiented the superpotential $W_c$ becomes quadratic -   it's just the pairing between $\Wedge^2 K$ and $\Wedge^2 V/\langle a\rangle$. 

  Comparing this with our diagram above, we see that within our formal neighbourhood we can identify the functor $\Phi$ with the functor 
  $$D^b( \widehat{(\Pf_s)}_a) \To D^b( \Wedge^2 K \tms \widehat{\OS}_a,\, W_c ) $$
  given by `pull-up to $\Wedge^2 K \tms \widehat{(\Pf_s)}_a$, then push-forward'. By Kn\"orrer periodicity this functor is an equivalence.
\end{proof}

\begin{proof}[Proof of Theorem \ref{thm:genericequivalence}]
  Let's first prove that $\Phi_\OS$ is fully faithful. This question is local in $\P \OS$, and locally $D^b(\Pf_s^{sm})$ is generated by the structure sheaf, so it's sufficent to check that
  $$\Phi_\OS : \cO_{\Pf_s} \To \Hom\!\big(\Phi_\OS(\cO_{\Pf_s}), \Phi_\OS(\cO_{\Pf_s})\big) $$
  is an isomorphism. Part (1) of the previous lemma says that this map is an isomorphism in the formal neighbourhood of any point $a\in \Pf_s^{sm}$, so it's an isomorphism.

  Now we show that $\Phi_\OS$ is essentially surjective, which is is equivalent to the adjoint $\Phi_\OS^\dagger$ sending non-zero objects to non-zero objects.  Suppose that $\cE\in \Br(\cY\tms_{\C^*}\OS,\,W)$. By part (2) of the previous lemma, $\cE\equiv 0$ in the formal neighbourhood of any point $a\notin \Pf_s^{sm}$. If $\Phi_\OS^\dagger \cE =0$, then (by part (1)) this is also true when $a\in \Pf_s^{sm}$. Hence $\hom(\cE, \cE)$ is acyclic in the formal neighbourhood of any point $a\in \P \OS$, so it is acyclic, and $\cE$ is contractible.
\end{proof}

\subsection{Extending over the singular locus}\label{sec:Extending}

In Section \ref{sec:Br(X)} we saw that $\Br(\cX)$ is equivalent to $D^b(A)$, and it contains a finite generating set of objects 
\begin{equation}\label{eqn:Generators}
  \set{\Schpur{\gamma, k_\gamma} S,\; \gamma\in Y_{s,q}} 
\end{equation}
corresponding to the obvious projective $A$-modules. To prove our equivalence, we need to show that the category $\Br(\cY\tms_{\C^*}\Wedge^2 V,\,W)$ has exactly the same structure.

\subsubsection{Some objects on the dual side} \label{sec:DualObjects}

First we must identify the corresponding generating objects in $\Br(\cY\tms_{\C^*}\Wedge^2 V,\,W)$. The correct thing to do is to take the sky-scraper sheaf along the locus
$$0\times \Wedge^2 V \;\subset\; \Hom(V, Q)\times \Wedge^2 V $$
and twist it by vector bundles. 

For each Young diagram $\delta\in Y_{q, s}$, choose a corresponding irrep $\Schpur{\delta, l_\delta} Q$ of $\GSp(Q)$. Now define an object $\widetilde{\cP}_\delta\in D^b(\cY\tms_{\C^*}\Wedge^2 V,\,W)$ as:
$$ \widetilde{\cP}_\delta \,= \,\cO_{0\tms\wedge^2 V} \otimes \Schpur{\delta, l_\delta} Q $$
This is an object in $D^b(\cY\tms_{\C^*}\Wedge^2 V,\,W)$ because $W$ vanishes along the locus $0\times \wedge^2 V$, but it does not lie in the subcategory $\Br(\cY\tms_{\C^*}\Wedge^2 V,\,W)$. However this subcategory is right-admissible (Lemma \ref{lem:BranesToModules2})  so for each $\delta\in Y_{q,s}$ we can define
$$\cP_\delta \;\in \Br(\cY\tms_{\C^*}\Wedge^2 V,\,W) $$
be the image of $\widetilde{\cP}_\delta$  under the right-adjoint to the inclusion.

\begin{lem}\label{lem:imageofPdelta} The functor $\Phi^\dagger$ sends the object  $\cP_\delta$ to (a shift of) the vector bundle $\Schpur{\delta^{\top c}\!\!, \,j_\delta} S$  on $\cX$, for some integer $j_\delta$. 
\end{lem}
So  $\Phi^\dagger(\cP_\delta)$ is one of the generating vector bundles \eqref{eqn:Generators} for the category $\Br(\cX)$,  up to a shift and tensoring by a line bundle. The shift is constant (independent of $\delta$), the line bundle might not be.
\begin{proof}
  Recall (from Section \ref{sec:adjoints}) that $\Phi^\dagger$ is the functor
  $$\Phi^\dagger: \cE \mapsto (\pi_1)_*\hom(\psi^*\cK, \pi_2^* \cE ) $$
  up to some shift and line bundle. We claim that:
  $$\Phi^\dagger \cP_\delta = \Phi^\dagger \widetilde{\cP}_\delta $$
  This is because the kernel $\psi^*\cK$ lives in the admissible subcategory $\Br(\cX\tms_{\C^*}\cY,\, W)$, so $\Phi^\dagger \widetilde{\cP}_\delta$ only depends on the projection of $\pi_2^* \widetilde{\cP}_\delta$ into this subcategory, and this is the same as $\pi_2^* \cP_\delta$.

  Now we compute $\Phi^\dagger(\cP_\delta)$. The pull-up $\pi_2^*(\cO_{0\tms\wedge^2 V})$ is just the sky-scraper sheaf along the subspace $\Hom(S,V)\tms 0$ in $\cX\tms_{\C^*}\cY$, and now we can use the model for $\psi^*\cK$ provided by Corollary \ref{cor:resolutionofK} to see the result.
\end{proof}

\begin{rem}\label{rem:AmbiguityAgain} The value of the integer $j_\delta$ depends on the integers $(n_\gamma, m_\gamma)$ in Corollary \ref{cor:resolutionofK} where $\gamma=\delta^{\top c}$, and it also depends on the integer $l_\delta$ which was an arbitrary choice made in the definition of $\cP_\delta$. We didn't bother to calculate $(m_{\gamma}, n_{\gamma})$ so there is no advantage in specifying $l_\delta$. We shall see shortly (Remark \ref{rem:A'isAmbiguous}) why none of these integers matter.
\end{rem}

\begin{rem}\label{rem:VertexSimples}
  Another way to characterize these objects $\cP_\delta$ is to consider their images in the equivalent category $D^b(B\otimes\cO_{\wedge^2 V}, W)$. 

  To start with take just the stack $\cY$, and consider a twist of the sky-scraper sheaf at the origin $\cO_0\otimes \Schpur{\delta, l_\delta} Q$. To project this object into the subcategory $\Br(\cY)$ we  apply a functor 
  $$\Hom(T', -)_{gr}: D^b(\cY) \to D^b(B)\cong \Br(\cY) $$
  where $T'$ is the vector bundle chosen to construct the algebra $B = \Hom(T', T')_{gr}$. 

  The result is a one-dimensional $B$-module $M_\delta$;  it's the `vertex simple' at the vertex $\delta$ (possibly after a grading shift, depending on the choices made in the definitions of $\cP_\delta$ and $T'$). 

  Now simply cross everything with $\Wedge^2 V$: we see that $\cP_\delta$ corresponds to the module $M_\delta\otimes \cO_{\wedge^2 V}$. Note that this is indeed a module over the curved algebra $(B\otimes\cO_{\wedge^2 V}, W)$ because $W$ acts as zero on it.
\end{rem}

We will see in due course that  this set of objects $\set{\cP_\delta}$, together with their twists by line bundles,  generate  $\Br(\cY\tms_{\C^*}\Wedge^2 V,\,W)$. The essential point is the following: 

\begin{lem}\label{lem:Pdeltaspan}
  The set of objects 
  $$\set{\cP_\delta\otimes \langle \beta_Q\rangle^p,\, \delta\in Y_{q,s}, \,p \in \Z}\; \subset \Br(\cY\tms_{\C^*}\Wedge^2 V,\,W)$$
  has no left orthogonal.
\end{lem}
\begin{proof}
  Tensoring by the line bundle/character $\langle \beta_Q\rangle$ is just a grading shift, so the claim is equivalent to the statement that for any  non-zero object $\cE$ of the category there is some $\delta\in Y_{q,s}$ such that:
  $$\Hom(\cE , \cP_\delta)_{gr} = \Hom(\cE, \widetilde{\cP}_\delta)_{gr}   \neq 0 $$
  By Lemma \ref{lem:minimalmodels} we can assume the sheaf underlying $\cE$ is a vector bundle  and the differential vanishes at the origin, so $\cE$ is the bundle associated to the representation $\cE|_0$. In particular, $\cE^\vee|_0$ is a non-zero representation, and by the grade-restriction rule it only contains $\Sp(Q)$-irreps from the set $Y_{q,s}$.

  Decompose $\cE^\vee|_0$ under R-charge, let $t$ be the highest weight occuring, and let $(\cE^\vee|_0)_t$ denote the highest weight space. Let $(\cE^\vee)_t$ denote the corresponding subbundle of $\cE^\vee$.  Now consider:
  $$\hom(\cE, \cO_{0\tms\wedge^2 V}) = \cE^\vee|_{0\tms\wedge^2 V} $$
  The R-charge action on $\Wedge^2 V$ is trivial, so this is a bounded complex of $\GSp(Q)$-equivariant vector bundles on $\Wedge^2 V$, whose final term is the restriction of $(\cE^\vee)_t$. The differential vanishes at the origin, so the final differential cannot be surjective, and there is a non-zero homology sheaf $h_t(\cE^\vee|_{0\tms\wedge^2 V})$ in the top degree. We have surjections:
  $$(\cE^\vee)_t |_{0\tms\wedge^2 V} \To h_t(\cE^\vee|_{0\tms\wedge^2 V})
  \To  h_t(\cE^\vee|_{0\tms\wedge^2 V})|_0 = (\cE^\vee|_0)_t $$
  By the grade-restriction-rule, there is some $\delta\in Y_{q,s}$ such that $(\cE^\vee|_0)_t\otimes \Schpur{\delta} Q$ contains non-zero $\Sp(Q)$-invariants, which  implies that $h_t(\cE^\vee|_{0\tms\wedge^2 V})\otimes \Schpur{\delta} Q$ doesn't vanish after taking $\Sp(Q)$-invariants. Since taking $\Sp(Q)$ invariants is exact, it follows that
  $$\Hom(\cE, \cP_\delta)_{gr} = \big( \cE^\vee|_{0\tms\wedge^2 V}\otimes\Schpur{\delta} Q\big)^{\Sp(Q)} $$
  has non-zero homology in degree $t$.
\end{proof}

\subsubsection{The algebra on the dual side}\label{sec:dualAlgebra}

Let $\cP$ denote the direct sum:
$$\cP = \bigoplus_{\delta\in Y_{q,s}} \cP_\delta $$
Now let  $A'$ denote the dg-algebra:
$$A' = \Hom(\cP, \cP)_{gr} $$
By definition, this is a $\C^*$-equivariant dga over the ring of functions on $\Wedge^2 V$. 

In Section \ref{sec:Br(X)} we saw that $\Br(\cX)$ is equivalent to the derived category of an algebra $A$. This algebra $A$ is also defined over $[\Wedge^2 V / \C^*]$, and it is Cohen--Macaulay. In this section we will prove that $A'$ is in fact quasi-isomorphic to $A$.  The results of the previous section show that $A'$ and $A$ are quasi-isomorphic over the open set $\OS\subset \Wedge^2 V$. We will show that $A'$ has homology only in degree 0 -- so it's really an algebra -- and that it's Cohen--Macaulay. Together these facts will imply that $A'\simeq A$ globally.
\pgap

\begin{rem}\label{rem:A'isAmbiguous} Recall from Remark \ref{rem:AisAmbiguous} that the grading on $A$ was ambiguous, because the definition of the vector bundle $T$ involved some choices of integers $k_\gamma$.  The definition of $A'$ has a similar ambiguity, because in the definition of $\cP_{\delta}$ we had to choose an integer $l_\delta$ to extend $\Schpur{\delta} Q$ to an irrep of $\GSp(Q)$. A related fact is that although we know that the object $\Phi^\dagger(\cP_\delta)$  is a shift of a vector bundle, we only know the vector bundle up to tensoring by a line bundle (see Remark \ref{rem:AmbiguityAgain}). 

  We can make all these ambiguities irrelevant by the following prescription: for each $\gamma\in Y_{s,q}$, choose the corresponding summand of $T$ to be the vector bundle $\Phi^\dagger(\cP_{\gamma^{\top c}})$. As long as we use this choice then  $A$ and $A'$ will have the same grading. 
\end{rem}

\begin{lem}\label{lem:AandA'areGenericallyIsomorphic}
  Over the open set $\OS\subset \Wedge^2 V$, we have a quasi-isomorphism:
  $$\Phi_\OS^\dagger: A'|_\OS \isoto A|_\OS $$
\end{lem}

\begin{proof} Using the remark above we have that $\Phi^\dagger$ sends $\cP$ to $T$ (up to a shift), and by Theorem \ref{thm:genericequivalence} $\Phi^\dagger$ becomes an equivalence over the open set $\OS$. 
\end{proof}

\begin{lem}\label{lem:DegreesOfA'}
  The dga $A'$ has homology concentrated in degrees $[-{2q+1\choose 2},0]$.
\end{lem}

At this point (and this point only!) it's important to recall that the R-charge acts non-trivially on $\cY$ and hence the algebra $B$ is not concentrated in homological degree zero. See Remark \ref{rem:Rcharge2}.

\begin{proof}
  By Remark \ref{rem:VertexSimples}, the object $\cP_\delta$ corresponds to the module $M_\delta\otimes \cO_{\wedge^2 V}$ over the curved algebra $(B\otimes \cO_{\wedge^2 V}, W)$, where $M_\delta$ is the one-dimensional $B$-module corresponding to the vertex $\delta$. Consider the dga:
  $$E=\End_B\!\left(\bigoplus_{\delta\in Y_{q,s}} M_\delta\right)$$
  We first claim that the homology of $E$ is concentrated in degrees $[-{2q+1\choose 2},0]$.

  To see this, observe that the $B$-module $\oplus_{\delta\in Y_{q,s}} M_\delta$ is the quotient $B/\overline{B}$ where $\overline{B}$ is the ideal in $B$ generated by the arrows. Hence it has a bar resolution whose underlying module is:
  $$\bigoplus_{p\geq 0} B\otimes \overline{B}^{\otimes p}[p] $$
  Since the R-charge acts with weight 1 on $\Hom(V,Q)$, the ideal $\overline{B}$ is the subspace of $B$ where the R-charge is $\geq 1$, so every term in this bar resolution is concentrated in non-negative degrees. It follows that $E$ is concentrated in non-positive degrees. For the lower bound, we use the Serre functor on $D^b(B)=\Br(\cY)$ from Proposition \ref{thm:SerreDuality2}. 

  To compute $A'$ we can take a projective resolutions of each $M_\delta$, tensor them with $\cO_{\wedge^2 V}$, then perturb the differential  to get a module over $(B\otimes \cO_{\wedge^2 V}, W)$ which is equivalent to $M_\delta$ \cite[Lemma 3.6]{segal_equivalence_2011}. This means that $A'$ can be computed from a spectral sequence that starts with $E$. The result follows.
\end{proof}

To show that $A'$ in fact has homology only in degree zero, and is Cohen--Macaulay, we develop the following general criterion:

\begin{lem}\label{lem:CohenMacaulayCriterion}
  Let $\cF$ be a $\C^{*}$-equivariant complex on $\A^{n}$ and let $c$ be the codimension of $\supp(\cF)$. If $\cF|_{0}$ is a complex with homology concentrated in degrees $[-c,0]$ then the homology of $\cF$ is concentrated in degree zero, and is a Cohen--Macaulay sheaf. 
\end{lem}
\begin{proof}
  First suppose $\cF$ is just a sheaf, situated in degree zero.  Then $\pd_{0}(\cF) \le c$ and 
  $$\depth_{0}(\cF) \le \dim \supp(\cF) = n-c$$
  so the Auslander--Buchsbaum formula implies that $\depth_{0}(\cF) = n-c$ and $\cF$ is Cohen--Macaulay at 0.
  Since the locus where $\cF$ is Cohen--Macaulay is open and $\C^{*}$-invariant, we see that $\cF$ is Cohen--Macaulay at every point of $\A^{n}$.

  It remains to show that $\cF$ is in fact concentrated in degree zero.
  Let $\cO_0$ denote the skyscraper sheaf at $0$.
  Note that the condition on $\cF|_0$ is equivalent to the requirement that $\Ext^i(\cF,\cO_0) = 0$ if $i \not\in[0,c]$.

  We will argue by induction on $d = n-c = \dim \supp(\cF)$, so assume first that $d = 0$.
  There is a spectral sequence converging to $\Ext^{p+q}(\cF,\cO_0)$ with second page  $E_{2}^{p,q} = \Ext^p(\cH^{-q}(\cF), \cO_0)$.   All the homology sheaves $\cH^j(\cF)$ are supported at the origin (since $d=0$) so they are automatically Cohen--Macaulay, hence we have $\Ext^i(\cH^j(\cF),\cO_0) \not= 0$ if and only if $i \in [0,n]$ and $\cH^j(\cF) \not= 0$.
  Let  $j_{\min}$ and $j_{\max}$ be the minimal and maximal values of $j$ such that $\cH^{j}(\cF) \not= 0$.
  Then $\Ext^n(\cH^{j_{\min}}(\cF),\cO_0) \not= 0$, and since this group lies in the upper right-hand corner of non-vanishing terms in the $E_2$ page, we find
  $$\Ext^{n-j_{\min}}(\cF,\cO_0) = \Ext^{n}(\cH^{j_{\min}}(\cF),\cO_0) \not= 0$$
  and hence $j_{\min} \ge 0$.  Similarly we find $\Ext^{-j_{\max}}(\cF, \cO_0)\neq 0$, so $j_{\max} \le 0$.
  Hence $j_{\min} = j_{\max} = 0$, so $\cF$ is concentrated in degree 0.

  Let now $d \ge 1$, and let $\A^{n-1} \subset \A^{n}$ be a generic hyperplane through the origin.
  Then $\cF|_{\A^{n-1}}$, as a complex on $\A^{n-1}$, satisfies the assumptions of the lemma, hence by induction must be concentrated in degree 0.

  Computing Tor groups by the short exact sequence
  \[
    0 \to \cO_{\A^{n}} \to \cO_{\A^{n}} \to \cO_{\A^{n-1}} \to 0
  \]
  shows that $\text{Tor}_{-i}(\cG, \cO_{\A^{n-1}}) = \cH^{i}(\cG|_{\A^{n-1}}) = 0$ for any sheaf $\cG$ and $i \not\in\{0,-1\}$.
  Therefore the spectral sequence $\cH^{i}(\cH^{j}(\cF)|_{\A^{n-1}}) \Rightarrow \cH^{i+j}(\cF|_{\A^{n-1}})$ degenerates at the $E_{2}$-page.
  Suppose $\cH^{i}(\cF) \not= 0$; the support of $\cH^{i}(\cF)$ is $\C^{*}$-invariant so it must intersect $\A^{n-1}$, hence $\cH^{0}(\cH^{i}(\cF)|_{\A^{n-1}}) \not= 0$. But this implies $\cH^{i}(\cF|_{\A^{n-1}}) \not= 0$, and by the induction hypothesis this only happens for  $i = 0$.
  So $\cF$ is concentrated in degree 0.
\end{proof}

\begin{prop}\label{prop:A'isCM} The dga $A'$ has homology only in degree zero, and its homology is a Cohen--Macaulay sheaf on $\Wedge^2 V$.
\end{prop}

\begin{proof}
  The algebra $A$ is, by construction, supported on the locus $\widetilde{\Pf_s}\subset \Wedge^2 V$.  The codimension of this locus is ${v-2s\choose 2}={2q+1\choose 2}$. By Lemma \ref{lem:AandA'areGenericallyIsomorphic} the dga $A'$ is supported on this same locus, since the complement of $\OS$ is contained in $\widetilde{\Pf_s}$. Lemma \ref{lem:DegreesOfA'} implies that $\cA'|_0$ has homology concentrated in $[-{2q+1\choose 2}, 0]$, and so Lemma \ref{lem:CohenMacaulayCriterion} applies.
\end{proof}

So we may  replace $A'$ with its homology, and declare that $A'$ is a Cohen--Macaulay algebra. 

Finally we want to use the Cohen--Macaulay property to deduce that the isomorphism $A'|_\OS\cong A|_\OS$ extends to the whole of $\Wedge^2 V$. This is essentially standard.

\begin{lem}
  Let $\cE$ be a Cohen--Macaulay sheaf on a regular variety $X$ and let $Z \subset \supp(\cE)$ be a closed subset of codimension at least 2.   Let $j \colon (X \setminus Z) \into X$ be the inclusion.
  Then $\cE = R^{0}j_{*}(j^{*}\cE)$.
\end{lem}
\begin{proof}
  Let $x \in Z$ be a not-necessarily closed point, and let $\cE_{x}$ be the restriction of $\cE$ to the local ring $\cO_{X,x}$.
  Then as $\cE$ is Cohen--Macaulay at $x$ we have
  \[
    \depth (\cE_{x}) = \dim \supp(\cE_{x}) = \codim(\overline{x}, \supp(\cE)) \ge \codim(Z, \supp(\cE)) \ge 2.
  \]
  Applying \cite[Prop. 3.7]{hartshorne_local} then shows that $\depth_{Z}(\cE) \ge 2$, which by \cite[Thm 3.8]{hartshorne_local} shows $\cH^{1}_{Z}(\cE) = 0$, and so by \cite[Cor. 1.9]{hartshorne_local} we get that $\cE = R^{0}j_{*}(j^{*}\cE)$.
\end{proof}

\begin{cor}\label{cor:extendingIsomorphismCodimension2}
  Let $\cE$ and $\cF$ be Cohen--Macaulay modules on a regular variety $X$, and let $\phi \colon \cE \to \cF$ be a homomorphism which is an isomorphism away from a locus $Z$ such that:
  \[
    \codim(Z, \supp(\cE))\ge 2\aand  \codim(Z, \supp(\cF)) \ge 2
  \]
  Then $\phi$ is an isomorphism.
\end{cor}

Setting $\cE=A'$ and $\cF=A$ and $\phi = \Phi^\dagger$, and noting that the complement of $\OS$ is $\widetilde{\Pf}_{s-1}$ which has codimension $2(v-2s)+1$ in $\widetilde{\Pf_s}$, we immediately obtain:

\begin{prop}\label{prop:AandA'areEquivalent} We have an isomorphism:
  $$\Phi^\dagger: A'  \isoto A $$
\end{prop}

This could perhaps be viewed as a form of Koszul duality between the algebra $A$ and the curved algebra $(B\otimes \cO_{\wedge^2 V}, W)$.

\subsubsection{Completing the proof}\label{sec:completingProof}

We can now complete the proof our `Hori duality' statement, Theorem \ref{thm:HoriDualityWithL}. The preceding proposition (Proposition \ref{prop:AandA'areEquivalent}) essentially proves the special case $L=0$.

\begin{prop}\label{prop:HoriDualityWithoutL}
  We have an equivalence
  $$ \Phi: \Br(\cX) \;\isoto\;  \Br(\cY\tms_{\C^*} \Wedge^2 V, \, W) $$
  of categories over $[\Wedge^2 V/\C^*]$. 
\end{prop}
\begin{proof}
  The algebra $A$ was defined as the endomorphism algebra $\Hom(T,T)_{gr}$ of a vector bundle. The bundle $T$ is self-dual (up to line bundles), so $A$ is isomorphic to $A^{op}$ (up to grading), and the functor $\Hom(-,T)_{gr}$ gives an equivalence between $\Br(\cX)^{op}$ and the derived category $D^b(\rmod A)$ of right $A$-modules. Recall that $\cP=\oplus_{\delta\in Y^{q,s}} \cP_\delta$, so the functor $\Phi^\dagger$ maps $\cP$ to $T$ and induces an isomorphism between $A'=\Hom(\cP, \cP)_{gr}$ and $A$ (Proposition \ref{prop:AandA'areEquivalent}). It follows that the category $D^b(\rmod A')$ has finite global dimension, so we have a well-defined functor
  $$\Hom( -, \cP)_{gr} : \Br(\cY\tms_{\C^*} \Wedge^2 V, \, W)^{op} \;\To\; D^b(\rmod A') $$
  with adjoint $\otimes_{A'} \cP$, and this is an equivalence by Lemma \ref{lem:Pdeltaspan}. The result follows.
\end{proof}

From here we prove the case of general $L$ with another application of Kn\"orrer periodicity.

Choose any $L\subset\Wedge^2 V^\vee$, and equip it with a weight 1 action of the group $\C^*=\GSp(S)/\Sp(S)=\GSp(Q)/\Sp(Q)$ and a weight 2 R-charge (as in Section \ref{sec:statement}). Form the stacks $\cX\tms_{\C^*} L$ and:
$$\cY\tms_{\C^*} \Wedge^2 V \tms_{\C^*} L = \Stack{\Hom(V, Q)\tms \Wedge^2 V\tms L }{\GSp(Q)} $$
Equip the former with the zero superpotential, and the latter with the superpotential $W$ (pulled-up from $\cY\tms_{\C^*} \Wedge^2 V$), then they are both Landau--Ginzburg B-models.  We can define a category $\Br(\cX\tms_{\C^*} L)$ by our usual grade-restriction-rule at the origin, this is generated by vector bundles and agrees with the triangulated closure of the pull-up of $\Br(\cX)$. On $\cY\tms_{\C^*} \Wedge^2 V \tms_{\C^*} L$ it doesn't make sense to apply a grade-restriction rule at the origin because the hypothesis of Lemma \ref{lem:minimalmodels} doesn't hold, so instead we define
$$\Br(\cY\tms_{\C^*} \Wedge^2 V \tms_{\C^*} L,\, W) $$
as the triangulated closure of the pull-up of $\Br(\cY\tms_{\C^*} \Wedge^2 V,\, W)$ (or as the matrix factorizations on our usual set of vector bundles). Then Proposition \ref{prop:HoriDualityWithoutL} immediately implies that we have an equivalence
$$\Phi: \Br(\cX\tms_{\C^*} L)\; \isoto\; \Br(\cY\tms_{\C^*} \Wedge^2 V \tms_{\C^*} L,\, W)$$
of categories relative to the base $[\Wedge^2 V\tms L\; /\;\C^*]$. This base carries a canonical quadratic superpotential which we'll call $\dualW$, if we pull this up to $\cX\tms_{\C^*} L$ it becomes the $\dualW$ already defined. If we pull it up to the other side we can add it on to the existing $W$, getting a `perturbed' superpotential $W+W'$. It's immediate that $\Phi$ induces a functor between the categories with this extra $W'$ added in, we claim that in fact:

\begin{lem}\label{lem:PhiPerturbed}
  The functor
  $$\Phi:  \Br(\cX\tms_{\C^*} L,\, \dualW)\; \To\; \Br(\cY\tms_{\C^*} \Wedge^2 V \tms_{\C^*} L,\, W+\dualW)$$
  is an equivalence.
\end{lem}
Morally the reason this is true is that $W'$ defines an (unobstructed) class in Hochschild cohomology for both sides, and we are deforming both categories along this class. Since the original categories are equivalent, the deformed categories must also be equivalent. Since we are ignorant of the necessary foundations to state this argument precisely we present an ad-hoc proof instead.
\begin{proof}
  Suppose we have objects  
  $$(E, d_E)\in\Br(\cX\tms_{\C^*} L,\, \dualW)\aand (F, d_F)\in \Br(\cY\tms_{\C^*} \Wedge^2 V \tms_{\C^*} L,\, W+\dualW)$$
  where $E$ and $F$ are (possibly infinite) direct sums of the usual vector bundles. Let's introduce an additional $\C^*$ action coming from rescaling just the $L$ directions. Since $E$ and $F$ are non-equivariantly trivial we can canonically lift this extra action to them, using the trivial action. Then we can decompose $d_E\in \End(E)$ and $d_F\in \End(F)$ with respect to this extra grading as
  $$d_E = (d_E)_0 + (d_E)_{>0} \aand  d_F = (d_F)_0 + (d_F)_{>0}$$
  (they cannot have negative terms), and it follows that $((d_E)_0)^2=0$ and $((d_F)_0)^{2} = W\id_F$. So the pair $(E, (d_E)_0)$ defines an object in $\Br(\cX)$, and the pair $(F, (d_F)_0)$ defines an object in $\Br(\cY\tms_{\C^*} \Wedge^2 V \tms_{\C^*} L,\, W)$; we'll denote these objects by $\widehat{E}$ and $\widehat{F}$. Moreover, it's evident that
  $$\Phi \widehat{E}  = \widehat{\Phi E } \aand \Phi^{\dagger} \widehat{F}=\widehat{\Phi^\dagger F} $$
  since the kernel is trivial in the $L$ directions.

  Now take two objects $E_1, E_2\in \Br(\cX\tms_{\C^*} L,\, \dualW)$. We have a chain map:
  \beq{eq:PhiSS}\Phi: \Hom(E_1, E_2) \To \Hom(\Phi E_1, \Phi E_2)\eeq
  These are complexes of modules over the ring of functions on $\Wedge^2 V\times L$, but we will instead view them as double-complexes of vector spaces, using our extra grading. If we look at the first pages of the associated spectral sequences we see the map
  $$\Phi: \Hom(\widehat{E_1}, \widehat{E_2}) \To \Hom(\Phi \widehat{E_1}, \Phi \widehat{E_2}) $$
  and this is a quasi-isomorphism. Therefore \eqref{eq:PhiSS} is also a quasi-isomorphism, so $\Phi$ is fully-faithful. A similar argument shows that the adjoint $\Phi^\dagger$ is also fully-faithful, hence $\Phi$ is an equivalence.
\end{proof}

Since here $\dualW$ is just the pairing between $L$ and $\Wedge^2 V/L^\perp$, Kn\"orrer periodicity gives an equivalence:
$$D^b(\cY\tms_{\C^*} \Wedge^2 V \tms_{\C^*} L,\, W+\dualW) \isoto D^b(\cY\tms_{\C^*} L^\perp, W) $$
This can  be chosen to be linear over $[\Wedge^2 V/\C^*]$; indeed the relative product over this base is $\cY\tms_{\C^*} L^\perp\tms_{\C^*} L $ which is a maximal isotropic subbundle for $\dualW$. 

\begin{lem}\label{lem:KPandBranes} Kn\"orrer periodicity induces an equivalence
  $$\Br(\cY\tms_{\C^*} \Wedge^2 V \tms_{\C^*} L,\, W+\dualW) \isoto \Br(\cY\tms_{\C^*} L^\perp, W)  $$
  of categories over $[\Wedge^2 V/\C^*]$.
\end{lem}
\begin{proof}
  An object $\cE\in \Br(\cY\tms_{\C^*} \Wedge^2 V \tms_{\C^*} L,\, W+\dualW)$ can be represented as a matrix factorization built from the infinite set of vector bundles corresponding to $Y_{q,s}$. If we restrict to $\cY\tms_{\C^*} L^\perp\tms_{\C^*} L$ and push down we get an (infinite-rank) matrix factorization which satisfies the grade-restriction rule at the origin, so it lies in $\Br(\cY\tms_{\C^*} L^\perp, W)$.

  Going in the other direction, choose a splitting $\Wedge^2 V=L^\perp\oplus L^\vee$, and correspondingly write $W=W_1+W_2$. Then we have a pull-up functor:
  $$ D^b(\cY\tms_{\C^*} L^\perp, W) \To D^b(\cY\tms_{\C^*} \Wedge^2 V \tms_{\C^*} L,\, W_1)$$
  The sky-scraper sheaf along  $\cY\tms_{\C^*} L^\perp\tms_{\C^*} L $ can be viewed as a curved dg-sheaf for the superpotential $W_2 + \dualW$, and it's equivalent to a Koszul-type matrix factorization whose underlying vector bundle is the exterior algebra on $L$. The inverse to our Kn\"orrer periodicity functor can be described as `pull up, then tensor with this matrix factorization'. Since $L$ is trivial as an $\Sp(Q)$-representation this functor sends objects in $\Br(\cY\tms_{\C^*} L^\perp, W)$ to objects in $\Br(\cY\tms_{\C^*} \Wedge^2 V \tms_{\C^*} L,\, W+\dualW)$.
\end{proof}

This completes the proof of Theorem \ref{thm:HoriDualityWithL}.

\section{The projective duality}
\label{sec:Projective}

Recall (from Section \ref{sec:Br(X)}) that we have  subcategories
$$\Br(\cX^{ss})\subset D^b(\cX^{ss}) \aand \Br(\cY^{ss})\subset D^b(\cY^{ss})$$
defined as the triangulated subcategories generated by the images of $\Br(\cX)$ and $\Br(\cY)$; or equivalently as the subcategories generated by the vector bundles corresponding to the sets $Y_{s,q}$ and $Y_{q,s}$. In this section we examine these categories in more detail.

In Section \ref{sec:NCRslices} we show that various natural notions of base change for these categories from $\cX^{\ss}$ to a slice $\cX^{\ss}|_{\P L^{\perp}}$ agree, under a genericity assumption.

In Section \ref{sec:Windows} we find explicit descriptions of the categories $\Br( \cY\tms_{\C^{*}} (L^{\perp} \setminus 0) ,\, W)$ and $\Br( \cY^{\ss} \tms_{\C^*}L^{\perp}  , W)$ as `window' subcategories of $\Br(\cY \tms_{\C^{*}}  L^{\perp},W)$. The first case  is handled by the general results of \cite{halpern-leistner_derived_2015, ballard_variation_2012}, while the second case requires more work.

Finally, Section \ref{sec:HPD} shows how these window results together with Hori duality imply that $\Br(\cX^{\ss})$ is HP dual to $\Br(\cY^{\ss})$.

\subsection{Slicing the non-commutative resolution}
\label{sec:NCRslices}
We saw in Section \ref{sec:Br(X)} that $\Br(\cX)$ is equivalent to the derived category of the graded algebra $A$, which is a ($\C^*$-equivariant) non-commutative resolution of the cone $\widetilde{\Pf}_s$.  If we restrict to the complement of the origin then $A$ becomes a sheaf of algebras on the projective variety $\Pf_s$, and by Cor.\ \ref{thm:BranesToModulesBaseChanged} we have an equivalence
$$\Br(\cX^{ss}) \isoto D^b(\Pf_s, A). $$
So the category $\Br(\cX^{ss})$ is a non-commutative resolution of $\Pf_s$. 

If we pick a subspace $L\subset \Wedge^2 V$, then we have our category $\Br(\cX\tms_{\C^*} L, \, \dualW)$, which after deleting the fibre over the origin in $\Wedge^{2} V$ gives a category:
$$\Br(\cX^{ss}\tms_{\C^*} L, \, \dualW)$$
In this section we'll show how this relates to a non-commutative resolution of the slice $\Pf_s\cap \P L^\perp$.

Let's delete all the singularities in $\widetilde{\Pf}_s$ (\emph{i.e.}~restrict to the open set $\OS\cap \widetilde{\Pf}_s$), so that $\cX$ becomes equivalent to the quasi-projective variety $\Pf_s^{\sm}$. Here the subcategory $\Br(\cX)$ becomes the whole of the derived category $D^b(\Pf_s^{\sm})$, since $A$ is a trivial Azumaya algebra on this subset.
The stack $\cX\tms_{\C^*} L$ becomes the total space of the vector bundle $L(-1)$ over  $\Pf_s^{\sm}$. The embedding $\Pf_s^{\sm}\into \P(\Wedge^2 V)$ gives a canonical section of the dual bundle $L^\vee(1)$, and the superpotential $\dualW$ is just the pairing of this section with the fibre co-ordinate. If we assume that $L$ is generic, then the section is transverse, and the critical locus of $\dualW$ is the zero locus of the section, which is the slice $\Pf_s\cap \P L^\perp$. This is a standard setup for `global Kn\"orrer periodicity', and we have an equivalence
$$ D^b(\Pf_s^{\sm}\cap \P L^\perp) \isoto D^b(\mathrm{Tot}(L(-1)), \,\dualW) $$
(see for example \cite{shipman_geometric_2012, hirano_derived_2016} -- note that the R-charge is acting fibre-wise on this vector bundle as required). We want to extend this fact over the singular locus in $\Pf_s$. 

The fibre product $\cX|_{L^\perp}$ is a quotient of a cone inside $\Hom(S, V)$, and if we intersect this with $\cX^{ss}$ we get a stack $\cX^{ss}|_{\P L^\perp}$ mapping to the singular variety $\Pf_s\cap \P L^\perp$. 

We define $\Br(\cX^{ss}|_{\P L^\perp})$ to be the subcategory of $D^b(\cX^{ss}|_{\P L^\perp})$ consisting of those objects which land in $\Br(\cX^{ss})$ under the pushforward functor.

\begin{lem}\label{lem:KPwithBranes} Assume that $\cX^{ss}|_{\P L^\perp}$ has the expected dimension. Then we have an equivalence
  \[
    D^b(\cX^{ss}|_{\P L^\perp}) \isoto   D^b( \cX^{ss} \tms_{\C^*} L,\, \dualW)
  \]
  inducing an equivalence:
  \[
    \Br(\cX^{ss}|_{\P L^\perp}) \isoto   \Br( \cX^{ss} \tms_{\C^*} L,\, \dualW)
  \]
\end{lem}
\begin{proof}
  This is very similar to Lemma \ref{lem:KPandBranes} but we can be a bit slicker in this case. Under Kn\"orrer periodicity, the push-forward functor from  $D^b(\cX^{ss}|_{\P L^\perp})$ to $D^b(\cX^{ss})$ corresponds to the restriction functor from $ D^b( \cX^{ss} \tms_{\C^*} L,\, \dualW)$ to $D^b(\cX^{ss})$. It follows immediately that the brane subcategories are mapped to each other.
\end{proof}

We can also restrict the algebra $A$ to the subspace $L^\perp$, giving an algebra defined over $\widetilde{\Pf_s}\cap L^\perp$. If we further delete the origin, we get a sheaf of algebras $A|_{\P L^\perp}$ on the projective variety $\Pf_s\cap \P L^\perp$. The non-commutative analogue of the Bertini theorem \cite{RSVdB} ensures that for generic $L^\perp$ this sheaf of algebras is a non-commutative resolution.

\begin{lem}
  \label{thm:BaseChangeForModuleCategoriesIsBaseChangeFactorisationCategories}
  Assume that $\cX^{ss}|_{\P L^\perp}$ has the expected dimension. Then we have an equivalence:
  $$\Br(\cX^{ss}|_{\P L^\perp}) \isoto D^b\!\big(\Pf_s\cap \P L^\perp,\, A|_{\P L^\perp} \big)$$
\end{lem}
\begin{proof}
  Let $\pi : \cX^{\ss} \tms_{\C^{*}} L \to \P(\Wedge^{2}V)$ be the projection.
  Let $T_{L} \in \Br(\cX^{\ss}|_{\P L^{\perp}})$ be the restriction of the vector bundle $T$ from Section \ref{sec:Br(X)}.
  Using the assumption that $\cX^{\ss}|_{\P L^{\perp}}$ has the expected dimension, we find that $\pi_{*}\hom(T_{L},T_{L}) = A|_{\P L^{\perp}}$ (where homomorphisms are derived).

  We now have an adjoint pair of functors $F_{L} = \pi_{*}\hom(T_{L},-)$ and $G_{L} = - \otimes_{A} T_{L}$ going between $D(\cX^{\ss}|_{\P L^{\perp}})$ and $D(\Pf_{s} \cap \P L^{\perp}, A|_{\P L^{\perp}})$.
  A standard computation shows that $G_{L}$ is fully faithful, and we must show that it sends $D^{b}(\Pf_{s} \cap \P L^{\perp}, A|_{\P L^{\perp}})$ surjectively onto $\Br(\cX^{\ss}|_{\P L^{\perp}})$.

  Consider first the case $L = 0$.
  Then Corollary \ref{thm:BranesToModulesBaseChanged} shows that $G_{0}$ gives an equivalence $D^{b}(\Pf_{s},A) \isoto \Br(\cX^{\ss})$.

  Let now $L$ be any subspace, let $i: \Pf_{s}\cap \P L^{\perp} \to \Pf_{s}$ and $j : \cX^{\ss}|_{\P L^{\perp}} \to \cX^{\ss}$ denote the inclusions, and let $i_{*} : D(\Pf_{s} \cap \P L^{\perp},A|_{\P L^{\perp}}) \to D(\Pf_{s}, A)$ be the pushforward functor. 
  There is an equality of functors 
  \[
    j_{*} \circ G_{L} = G_{0} \circ i_{*} : D(\Pf_{s} \cap \P L^{\perp}, A|_{\P L^{\perp}}) \to D(\cX^{ss}).
  \]
  Now as $G_{0} \circ i_{*}$ sends $D^{b}(\Pf_{s} \cap \P L^{\perp}, A|_{\P L^{\perp}})$ to $\Br(\cX^{\ss})$, and $\Br(\cX^{\ss}|_{P L^{\perp}}) = (j_{*})^{-1}(\Br(\cX^{\ss}))$, it follows that $G_{L}$ sends $D^{b}(\Pf_{s} \cap \P L^{\perp}, A|_{P L^{\perp}})$ to $\Br(\cX^{\ss}|_{\P L^{\perp}})$.

  If $\cE \in \Br(\cX^{\ss}|_{\P L^{\perp}})$, then $F_{L}(\cE) = 0$ implies that 
  \[
    i_{*}\pi_{*}\hom(T_{L},\cE) = \pi_{*}\hom(T_{0},j_{*}\cE) = F_{0}(j_{*}(\cE)) = 0,
  \]
  hence $j_{*}(\cE) = 0$, and so $\cE = 0$.
  Thus we find that $\ker F_{L} \cap \Br(\cX^{\ss}|_{\P L^{\perp}}) = 0$, which by \cite[Thm.\ 3.3]{kuznetsov_homological_2007} means that $G_{L}(D^{b}(\Pf_{s} \cap \P L^{\perp}, A|_{P L^{\perp}})) = \Br(\cX^{\ss}|_{\P L^{\perp}})$.
\end{proof}

\begin{cor}\label{cor:NCRslices}
  Assume that $\cX^{ss}|_{\P L^\perp}$ has the expected dimension. Then 
  \[
    \Br( \cX^{ss} \tms_{\C^*} L,\, \dualW) \cong D^b\!\big(\Pf_s\cap \P L^\perp,\, A|_{\P L^\perp} \big).
  \]
\end{cor}
If $\cX^{ss}|_{\P L^\perp}$ does not have the expected dimension then it seems appropriate to view  $ \Br(\cX^{ss} \tms_{\C^*} L,\, \dualW)$ as the correct base-change of the category $\Br(\cX^{\ss})$ to $\P L^{\perp}$ (as we did in Section \ref{sec:HPDualityBackground}). 

We remark that these results also hold `in the affine case', \emph{i.e.} if we don't restrict to $\cX^{ss}$ then it's still true that
$$\Br(\cX\tms_{\C^*} L,\, W) \cong D^b\!\big(\Stack{\widetilde{\Pf_s}\cap L^\perp}{\C^*},\, A|_{L^\perp}\big), $$
under the stronger assumption that $\cX|_{L^\perp}$ has the expected dimension.

Now we examine the Serre functor on these categories.

\begin{prop}\label{thm:SerreOnTheProjectiveNCR} Suppose that  $\Pf_s \cap \P L^\perp$ has the expected dimension, and that $A|_{\P L^\perp}$ is a non-commutative resolution of it. Then there exists a Serre functor on $D^b(\Pf_s \cap \P L^\perp, A|_{\P L^\perp})$ given by:
  \[
    \cE \mapsto \cE \otimes \cO_{\P (\wedge^2 V)}(sv - \dim L)[\dim \Pf_s \cap \P L^\perp]
  \]
\end{prop}
Note that for generic choice of $L$ both conditions hold. In particular, for generic $L$ with $\dim L= sv$, the category $D^b(\Pf_s \cap \P L^\perp, A|_{\P L^\perp})$ is Calabi--Yau.
\begin{proof}
  Note that $\cO_{\P(\wedge^2 V)}(sv-\dim L)$ is the canonical bundle on $\Pf_s \cap \P L^\perp$.
  By \cite[Example 6.4]{yekutieli_dualizing_2006}, and using the fact that $A$ is a maximal CM sheaf, we find that $A(sv - \dim L^\perp)|_{\P L^\perp}$ is a rigid dualising complex for $(\Pf_s \cap \P L^\perp, A|_{\P L^\perp})$.
  The claim then follows from \cite[Prop.\ 6.14]{yekutieli_dualizing_2006}.
\end{proof}

There are of course exactly analogous results for $\cY$: we have an equivalence
$$\Br(\cY^{ss}\tms_{\C^*} L^\perp, W) \isoto D^b(\Pf_q\cap \P L,\, B|_{\P L}) $$
under the assumption that $\cY^{ss}|_L$ has the expected dimension. 

\subsection{Windows}\label{sec:Windows}

Since the equivalence of Theorem \ref{thm:HoriDualityWithL} is defined relative to the base $[\Wedge^2 V/\C^*]$, we can restrict it to the complement of the origin and get:
\begin{cor}\label{thm:HoriDualityWithLOverSemistableLocus}
  We have an equivalence:
  $$ \Br(\cX^{\ss}\tms_{\C^*} L, \, W') \;\isoto \;  \Br(\cY\tms_{\C^*} (L^\perp \setminus 0), \, W) $$
\end{cor}
\begin{proof} This follows formally from Theorem \ref{thm:HoriDualityWithL}, because these categories are by definition the subcategories generated by the images of the corresponding brane subcategories on $\cX\tms_{\C^*} L$ and $\cY\tms_{\C^*} L^\perp$.
\end{proof}

If $L$ is sufficiently generic then by Corollary \ref{cor:NCRslices} the first category is equivalent to $D^b(\Pf_s\cap \P L^\perp,\, A|_{\P L^\perp})$, our non-commutative resolution of the variety $\Pf_s\cap \P L^\perp$. However, our non-commutative resolution for the dual slice $\Pf_q\cap \P L$ is equivalent to the category $\Br(\cY^{ss}\tms_{\C^*} L^\perp,\, W) $. So we need to understand how the categories
$\Br(\cY\tms_{\C^*} (L^\perp\setminus 0),\, W)$ and $\Br(\cY^{ss}\tms_{\C^*} L^\perp,\, W) $ are related. 

The stacks $\cY\tms_{\C^*} (L^\perp\setminus 0)$ and $\cY^{ss}\tms_{\C^*} L^\perp$ are related by variation of GIT stability; they are the two possible semi-stable loci in the ambient stack $\cY\tms_{\C^*} L^\perp$. We can use the technique of `windows' \cite{segal_equivalence_2011, halpern-leistner_derived_2015, ballard_variation_2012} to compare them, by lifting their associated categories to subcategories defined on the ambient stack.

The existence of such lifts has been worked out in large generality in the mentioned literature, but our situation is complicated by two factors: we really care about the B-brane subcategories of each stack rather than the full derived category, and we want the window categories to be of a specific form, in order that they be comparable as subcategories of $D^{b}(\cY \tms_{\C^{*}} L^{\perp}, W)$.
It will therefore take us some work to construct these lifts.

Recall (from Section \ref{sec:GSp}) that irreps $\Schpur{\delta, k} Q$ of $\GSp(Q)$ are determined by a weight $\delta$ of $\Sp(Q)$ and a weight $k$ of the diagonal subgroup $\Delta\subset \GSp(Q)$. Our previous `brane' subcategory $\Br(\cY\tms_{\C^*} L^\perp,\, W)$ was defined by a rule which restricted the allowed $\Sp(Q)$-representations occuring at the origin $0$; if we express it in terms of $\GSp(Q)$-irreps, then we allowed $\Schpur{\delta, k} Q$ if $\delta$ was in $Y_{q,s}$, but we placed no restriction on $k$.

We now define some smaller subcategories by placing a restriction on $k$ as well. For an interval $I\subset \Z$, we define
$$\Br(\cY\tms_{\C^*} L^\perp,\, W)_I \; \subset \; \Br(\cY\tms_{\C^*} L^\perp,\, W) $$
to be the full subcategory of objects $\cE$ such that the $\Delta$-weights of $h_\bullet(\cE|_0)$ all lie in $I$. By Lemma \ref{lem:minimalmodels}, this is equivalent to requiring that $\cE$ can be represented by a matrix factorization built only from the vector bundles $\Schpur{\delta, k} Q$, where $\delta\in Y_{q,s}$ and $k\in I$.

Note that these subcategories are not defined relative to any base and are not preserved by tensoring with line bundles.

General techniques give the first window result:

\begin{prop} \label{prop:EasyWindow}Set $l'=\dim L^\perp$. For any $n\in \Z$, the restriction functor
  $$ \Br(\cY\tms_{\C^*} L^\perp,\, W)_{[n, n+2l')} \;\To\; \Br(\cY\tms_{\C^*} (L^\perp\setminus 0),\, W) $$
  is an equivalence.
\end{prop}
\begin{proof}
  Define a subcategory 
  $$D^b(\cY\tms_{\C^*} L^\perp,\, W)_{[n, n+2l')} \;\subset \; D^b(\cY\tms_{\C^*} L^\perp,\, W)$$
  by taking only objects whose $\Delta$-weights at the origin lie in the interval $[n, n+2l')$. By general theory \cite{segal_equivalence_2011, halpern-leistner_derived_2015, ballard_variation_2012}, the restriction functor 
  $$D^b(\cY\tms_{\C^*} L^\perp,\, W)_{[n, n+2l')}\;\To\;D^b(\cY\tms_{\C^*} (L^\perp\setminus 0),\, W)$$
  is an equivalence --  note that the width of this interval here is $2l'$ (instead of $l'$) because $\Delta$ acts with weight 2 on $L^\perp$. So we just need to argue that this equivalence matches up the brane subcategories.

  If an object lies in $\Br(\cY\tms_{\C^*} L^\perp,\, W)_{[n, n+2l')}$ then (by definition) it restricts to give an object in $\Br(\cY\tms_{\C^*} (L^\perp\setminus 0),\, W) $.
  This shows that restriction gives an embedding:
  $$\Br(\cY\tms_{\C^*} L^\perp,\, W)_{[n, n+2l')} \;\into\; \Br(\cY\tms_{\C^*} (L^\perp\setminus 0),\, W)$$
  Given an object $\cE\in \Br(\cY\tms_{\C^*} (L^\perp\setminus 0),\, W)$, it is the restriction of some object $\tilde{\cE}\in \Br(\cY\tms_{\C^*} L^\perp,\, W)$. 
  The general recipe of \cite{segal_equivalence_2011, halpern-leistner_derived_2015, ballard_variation_2012} gives a way to modify $\widetilde{\cE}$ to a new object $\widetilde{\cE}^{\pr}$, which still restricts to $\cE$, but which lies in the subcategory $D^b(\cY\tms_{\C^*} L^\perp,\, W)_{[n, n+2l')}$.
  The object $\widetilde{\cE}^{\pr}$ is constructed by taking cones over objects supported at $\cY \times 0$.
  It is easy to see that this process of taking cones will not introduce new $\Sp(Q)$-weights, and so since all the $\Sp(Q)$-weights of $\widetilde{\cE}$ lie in $Y_{q,s}$, the same will be true of $\widetilde{\cE}^{\pr}$.
  Hence $\widetilde{\cE}^{\pr} \in \Br(\cY\tms_{\C^*} L^\perp,\, W)_{[n, n+2l')}$, which proves that the functor $\Br(\cY\tms_{\C^*} L^\perp,\, W)_{[n, n+2l')} \to \Br(\cY\tms_{\C^*} (L^\perp \setminus 0),\, W)$ is essentially surjective.
\end{proof}

For the other GIT quotient, we have:
\begin{thm}
  \label{thm:ProjectiveWindows}
  The restriction functor induces an equivalence:
  \[
    \Br(\cY \tms_{\C^{*}} L^\perp,\, W)_{[-qv,qv)} \isoto \Br(\cY^{ss} \tms_{\C^{*}} L^\perp,\, W)
  \]
\end{thm}
This is the main technical result of Section \ref{sec:Projective}, we will prove it as Propositions \ref{prop:Fullyfaithful} and \ref{prop:EssentiallySurjective} in the next two sections. 

\begin{rem}  If we set $L^\perp=0$ then as a special case we get the equivalence:
  \begin{align}\label{eq:BrForY} \Br(\cY )_{[-qv,qv)} \isoto \Br(\cY^{ss}) \end{align}
  This is the real content of the theorem; notice that $l'=\dim(L^\perp)$ doesn't enter into the definition of the category $\Br(\cY \tms_{\C^{*}} L^\perp,\, W)_{[-qv,qv)}$, so morally speaking we get the general case from the special case by crossing with $L^\perp$ and then `peturbing by $W$'. We have chosen to prove the general case directly, but still most of the work takes place on $\cY$. 
\end{rem}

\begin{rem}
  The category $\Br(\cY )_{[-qv,qv)}$ has a full strong exceptional collection of vector bundles, essentially by definition, and \eqref{eq:BrForY} is equivalent to saying that these same vector bundles form a full exceptional collection for $\Br(\cY^{ss})$  (see Lemmas \ref{thm:BundlesExceptionalAndSemiorthogonal} and \ref{thm:LefschetzDecompositionExists} below).  In the case $q=1$ the stack $\cY^{ss}$ is actually the variety $\Gr(V,2)$, and the inclusion $\Br(\cY^{ss})\into D^b(\cY^{ss})$ is an equivalence, so we get a full exceptional collection on $\Gr(V,2)$. This is exactly the collection discovered by Kuznetsov \cite{KuznetsovECs}; our equivalence \eqref{eq:BrForY} is a generalization of Kuznetsov's work to $q>1$.
\end{rem}

\begin{rem}
  Note that \eqref{eq:BrForY} does not follow from the general theory of \cite{halpern-leistner_derived_2015, ballard_variation_2012}, which would tell us to consider a Kempf--Ness stratification of $\cY\setminus\cY^{ss}$. 
\end{rem}

Set $n=-qv$ in Proposition \ref{prop:EasyWindow}.  Our two subcategories of $\Br(\cY \tms_{\C^{*}} L^{\perp}, W)$ are then contained one inside the other; we have
$$ \Br(\cY\tms_{\C^*} L^\perp,\, W)_{[{-qv},\, 2l'-qv)} \;\subset \; \Br(\cY\tms_{\C^*} L^\perp,\, W)_{[{-qv}, qv)} $$
if $l'\leq qv$, and vice-versa if $l'\geq qv$. Combining this with our other results, we get the following `HPD lite' statement:

\begin{cor}\label{cor:HPDlite}
  If $l'\leq qv$ we have a fully faithful functor
  $$\Br(\cX^{ss}\tms_{\C^*} L,\, \dualW) \into \Br(\cY^{ss}\tms_{\C^*} L^\perp,\, W) $$
  and if $l'\geq qv$ we have a fully faithful functor:
  $$\Br(\cY^{ss}\tms_{\C^*} L^\perp,\, W)\into \Br(\cX^{ss}\tms_{\C^*} L,\, \dualW) $$
  If $l'=qv$ the two categories are equivalent.
\end{cor}
If $L$ and $L^\perp$ are sufficiently generic then (by Corollary \ref{cor:NCRslices}) we may replace the two categories by $D^b(\Pf_s\cap \P L^\perp,\, A|_{\P L^\perp})$ and $D^b(\Pf_q\cap \P L,\, B|_{\P L})$.

\subsubsection{Fully faithfulness}
In this section we prove half of Theorem \ref{thm:ProjectiveWindows}: that the restriction functor from $\Br(\cY \tms_{\C^{*}} L^\perp,\, W)_{[-qv,qv)}$ to $\Br(\cY^{ss} \tms_{\C^{*}} L^\perp,\, W)$ is fully faithful.

\begin{lem}
  \label{thm:CohenMacaulayDuality}
  Let $R$ be a graded local Gorenstein ring of dimension $n$ with $\omega_{R} \cong R(-d)$, and write $k=R/\mathfrak{m}$ for the unique 1-dimensional $R$-module. If $M$ is a maximal Cohen--Macaulay module on $R$ then
  \[
    \Ext_{R}^{i}(k,M) = 0
  \]
  for $i \not= n$ and:
  \[
    \Ext_{R}^{n}(k,M)^\vee = M^{\vee}(-d)\otimes k
  \]
  as graded $k$-vector spaces.
\end{lem}
\begin{proof}
  Let $i : \Spec k \to \Spec R$ be the inclusion.
  We have 
  $$i^{!}(M) =(i^{*}(\RHom(M, \omega_{R}[n])))^{\vee},$$
  e.g.\ by \cite[\href{http://stacks.math.columbia.edu/tag/0AU2}{Tag 0AU2}]{stacks-project}.
  Since $M$ is maximal Cohen--Macaulay, we have 
  \[
    \RHom(M,\omega_{R}[n]) \cong M^{\vee} \otimes \omega_{R}[n],
  \]
  e.g.\ by \cite[Prop.\ 21.12]{eisenbud_commutative_1995}.
  The claim follows.
\end{proof}

Let $\pi$ denote the map:
$$\pi : \cY \to  \Stack{\widetilde{\Pf_q}}{\C^*}$$

\begin{lem}
  \label{thm:SchpurDuals}
  If $\cE$ is a locally free sheaf on $\cY$, then $\pi_{*}(\cE^{\vee}) \cong \pi_{*}(\cE)^{\vee}$.
\end{lem}
\begin{proof}
  Since $\pi$ does not send any divisor to a codimension $\ge 2$ subset, the pushforward functor preserves reflexive sheaves by \cite[Prop.\ 1.3]{brion_sur_1993}, and so $\pi_{*}(\cE^{\vee})$ is reflexive.

  Over the smooth locus $\widetilde{\Pf}_q^{\sm}$ the map $\pi$ is an equivalence of stacks, so within this locus $\pi_{*}(\cE^{\vee}) \cong \pi_{*}(\cE)^{\vee}$. 
  Let $j : \widetilde{\Pf}_{q}^{\sm} \into \widetilde{\Pf}_{q}$ be the inclusion.
  Since the singular locus has codimension $\ge 2$, and both $\pi_{*}(\cE^{\vee})$ and $\pi_{*}(\cE)^{\vee}$ are reflexive, we have
  \[
    \pi_{*}(\cE^{\vee}) \cong j_{*}(\pi_{*}(\cE^{\vee})|_{\widetilde{\Pf}_{q}^{\sm}}) \cong j_{*}(\pi_{*}(\cE)^{\vee})|_{\widetilde{\Pf}_{q}^{\sm}}) \cong \pi_{*}(\cE)^{\vee},
  \]
  using \cite[Prop.\ 1.6]{hartshorne_stable_1980}.
\end{proof}

Let $Z=\set{0}\subset \widetilde{\Pf}_q$ be the vertex of the cone, and let $H_Z^\bullet(-) : D^{b}(\Stack{\widetilde{\Pf}_q}{\C^*}) \to D^{b}(\Vect)$ denote the functor of taking local cohomology at $Z$ (and $\C^{*}$-invariants).

\begin{lem}
  \label{prop:FullyfaithfulSimpleCase}
  Take two irreps of $\GSp(Q)$ with highest weights  $(\alpha_{1}, k_{1})$ and $(\alpha_{2},k_{2})$ lying in the set $Y_{q,s} \tms [-qv, qv)$. We have associated vector bundles $\Schpur{\alpha_{1}, k_{1}} Q$ and $\Schpur{\alpha_{2}, k_{2}} Q$ on $\cY$. Let:
  $$N = \pi_{*}\!\big(\Schpur{\alpha_{1},k_{1}}Q^{\vee} \otimes \Schpur{\alpha_{2},k_{2}}Q\big) \quad \in \Coh\!\big(\Stack{\widetilde{\Pf_q}}{\C^*}\big)   $$
  Then  $H^{*}_{Z}(N(i)) = 0$ for all $i \ge 0$.
\end{lem}
\begin{proof}
  Let us restrict to the case $i = 0$, the general case is shown the same way.
  Decomposing the representation $\Schpur{\alpha_{1},k_{1}}Q^{\vee} \otimes \Schpur{\alpha_{2},k_{2}}Q$ into irreps, we obtain a decomposition of $N$. 
  We may write $H^\bullet_{Z}(N) = \lim_{\rightarrow}\Hom(\cO/I_{Z}^{p},N)$, where $\lim_{\rightarrow}$ denotes a homotopy colimit \cite[Prop.\ 1.5.3]{lipman_lectures_2002}.
  Using the short exact sequence $I_{Z}^{p}/I_{Z}^{p+1} \into \cO/I_{Z}^{p+1} \onto \cO/I_{Z}^{p}$, the required vanishing then follows if $\Hom(I_{Z }^{p}/I_{Z }^{p+1}, N) = 0$ for all $p \ge 0$.
  As $I_{Z}^{p}/I_{Z}^{p+1}$ is isomorphic to a direct sum of copies of $\cO_{Z}(-p)$, it finally suffices to show that $\RHom(\cO_{Z},N(p)) = 0$ vanishes when $p \ge 0$.

  Write $n$ for the dimension of $\widetilde{\Pf}_q$, and recall that its canonical bundle is $\cO(-2qv)$.  By \cite[Thm.\ 1.6.4]{spenko_non-commutative_2015} the module $N(p)$ is Cohen--Macaulay.
  By Lemma \ref{thm:CohenMacaulayDuality} we have that $\Ext^{j}(\cO_{Z},N(p)) = 0$ if $j\neq n$, and $\Ext^{n}(\cO_{Z},N(p))^{\C^{*}}$ is the dual of the invariants in $N^{\vee}(-p-2qv)|_{Z}$. 
  Now by Lemma \ref{thm:SchpurDuals}, we have $N^{\vee} = \pi_{*}(\Schpur{\alpha_{1},k_{1}}Q \otimes \Schpur{\alpha_{2},k_{2}}Q^{\vee})$, and this module is generated in degrees $\ge k_{1} - k_{2} > -2qv$.
  So for $p\geq 0$ the module $N^{\vee}(-p-2qv)$ is generated in positive degree, hence its restriction to $Z$ has no $\C^{*}$-invariants.
\end{proof}

\begin{lem}
  \label{lem:SectionsVanishing}
  Let $\Stack{\A^n}{\C^*}$ be a vector space with a diagonal $\C^*$-action of weight $-1$. Suppose that $\cE\in D^b(\Stack{\A^n}{\C^*})$ has the property that $(\cE(p)|_0)^{\C^*}=0$ for all $p\geq 0$. 
  Then $\cE$ has no global sections.
\end{lem}
\begin{proof}
  This follows easily from Lemma \ref{lem:minimalmodels}.
\end{proof}

Let's abuse notation and continue to use $\pi$ for the map:
$$\pi: \cY\tms_{\C^*} L^\perp \To \Stack{\widetilde{\Pf_q}\tms L^\perp}{\C^*} $$

\begin{prop}
  \label{prop:Fullyfaithful}
  The restriction functor
  \[
    \Br(\cY \tms_{\C^{*}} L^\perp,\,W)_{[-qv,qv)} \to \Br(\cY^{\ss} \tms_{\C^{*}} L^\perp,\,W)
  \]
  is fully faithful.
\end{prop}
\begin{proof}
  Let $\cE, \cF \in \Br(\cY \tms_{\C^{*}} L^\perp,W)_{[-qv,qv)}$. We can assume $\cE$ and $\cF$ are matrix factorizations whose underlying vector bundles are direct sums of the bundles $\set{\Schpur{\alpha,k} Q,\, (\alpha, k)\in Y_{q,s}\tms[-qv,qv)}$. The morphisms between $\cE$ and $\cF$ are just the homology of the chain-complex $\Gamma(\cE^\vee\otimes \cF) = (\cE^\vee\otimes \cF)^{\GSp(Q)}$. If we restrict to $\cY^{\ss} \tms_{\C^{*}} L^\perp$ then the morphism are \emph{a priori} more complicated, because we must take the derived global sections of $\cE^\vee\otimes \cF$. However, we claim that in fact the vector bundle $\cE^\vee\otimes \cF$ has the property that
  $$R\Gamma(\cY^{\ss} \tms_{\C^{*}} L^\perp,\, \cE^\vee\otimes \cF) = \Gamma(\cY \tms_{\C^{*}} L^\perp, \cE^\vee\otimes \cF) $$
  so the morphisms between $\cE$ and $\cF$ do not change upon restriction. This is the statement of the proposition.

  Now we prove the claim. Obviously it's sufficient to prove it for the summands of $\cE$ and $\cF$, so pick two vector bundles $\Schpur{\alpha_1,k_1} Q$ and $\Schpur{\alpha_2,k_2} Q$ with $(\alpha_1, k_1)$ and $(\alpha_2, k_2)$ both in $Y_{q,s}\tms[-qv, qv)$. Let $M = \pi_{*}( \Schpur{\alpha_{1},k_{1}}Q^{\vee} \otimes \Schpur{\alpha_{2},k_{2}}Q)$. We are claiming that the map
  $$\Gamma(M) \to R\Gamma( (\widetilde{\Pf_q}\setminus Z)\tms L^\perp,  M) $$
  is an isomorphism, \emph{i.e.} that the local cohomology $H^\bullet_{Z \tms L^\perp}(M)$ vanishes. Arguing as in the proof of Lemma \ref{prop:FullyfaithfulSimpleCase}, this reduces to checking that $\RHom(\cO_{Z \tms L^\perp}, M(p)) = 0$ for all $p \ge 0$. In other words, if we let $i : Z \tms L \into \widetilde{\Pf_q}\tms L$ be the inclusion, we want to see that $i^{!}M(p)$ has no global sections when $p\geq 0$. 

  Everything here is flat over $L^\perp$, so we can use Lemma \ref{prop:FullyfaithfulSimpleCase} to understand the restriction of $i^{!}M(p)$ to the origin in $L^\perp$. In particular, $i^{!}M(p)|_{0}$ has no invariants for $p\geq 0$. Now apply Lemma \ref{lem:SectionsVanishing}.
\end{proof}

\subsubsection{Essential surjectivity}

In this section we prove that the restriction functor from $\Br(\cY \tms_{\C^{*}} L^\perp,\, W)_{[-qv,qv)}$ to $\Br(\cY^{ss} \tms_{\C^{*}} L^\perp,\, W)$ is essentially surjective.

Pick an irrep $\Schpur{\alpha,k}Q$ with highest weight $(\alpha,k)$ lying in $Y_{q,s}\tms \Z$. On $\cY$ we have a corresponding twist  $\cO_0\otimes \Schpur{\alpha, k} Q$ of the skyscraper sheaf at the origin, and by Lemma \ref{lem:BranesToModules2} we can project this into $\Br(\cY)$ to get an object:
$$P_{\alpha, k} \in \Br(\cY)$$
Up to a grading shift this object corresponds to the `vertex simple' $B$-module at the vertex $\alpha$. It also agrees with the restriction of the object $\cP_\alpha$ from Section \ref{sec:DualObjects} to the slice $\cY\tms 0$, again up to grading; c.f.~Remark \ref{rem:VertexSimples}. 

For any object $\cE\in \Br(\cY)$, if we restrict to the origin we get a $\GSp(Q)$-representation $h_\bullet(\cE|_0)$, and by definition of $\Br(\cY)$ this is a direct sum of irreps with highest weight in $Y_{q,s}\tms \Z$. 
We also know (by Lemma \ref{lem:minimalmodels}) that $\cE$ is equivalent to a matrix factorization on the associated vector bundle. For brevity, let's refer to the set of irreps that occur in $h_\bullet(\cE|_0)$ as the \emph{weights} of $\cE$. 

Our first task is to determine the weights of the object $P_{\alpha, k}$. 

\begin{lem}\label{lem:WeightsByHom}
  For any  $\cE \in \Br(\cY)$, we have that $(\alpha, k)$ is a weight of $\cE$ if and only if $\Hom_\cY(\cE ,P_{\alpha,k})\neq 0$.
\end{lem}
\begin{proof}
  By adjunction:
  $$\Hom\!\big(\cE|_0, \Schpur{\alpha, k} Q\big) = \Hom_\cY\!\big(\cE, \cO_0\otimes \Schpur{\alpha, k}Q\big) = \Hom_\cY\!\big(\cE, \cP_{\alpha, k}\big) $$
\end{proof}

\begin{lem}
  \label{lem:usefulresolutions}
  The set of weights of $P_{\alpha,k}$ contains one copy of $(\alpha,k)$,  and all the remaining weights are contained in the set $Y_{q,s} \times [k-2qv, k )$.
\end{lem}

\begin{proof}
  To compute the weights of $P_{\alpha, k}$ we need to resolve it by vector bundles associated to the set $Y_{q,s}\tms \Z$. As a first step, replace $\Schpur{\alpha,k}Q \otimes \cO_{0}$ by a twist of the usual Koszul resolution on $\Hom(V, Q)$. This complex has a single copy of $\Schpur{\alpha, k} Q$, and the remaining summands are of the form $\Schpur{\alpha' , k'} Q$ with $k'<k$. 

  Now we need to project into the subcategory $\Br(\cY)$. If $\alpha'\in Y_{q,s}$ then $\Schpur{\alpha', k'} Q$ is unaffected by this projection, however most summands will not be of this form and will become something more complicated. To understand what they become, we use the results of \v{S}penko--Van den Bergh. They construct a large set of objects in $D^b(\cY)$ which are orthogonal to $\Br(\cY)$, and also find locally-free resolutions of them. Projecting the resolutions into $\Br(\cY)$ gives exact sequences, which allows us to express the projections of vector bundles in terms of projections of other vector bundles. Using these exact sequences repeatedly, we can eventually express the projection of any vector bundle in terms of vector bundles associated to the set $Y_{q,s}\tms \Z$ (this is how they prove that $\Br(\cY)$ has finite global dimension). 

  To be precise, we consider the complexes which are denoted $C_{\lambda, \chi}$ in \cite[\S 11.2]{spenko_non-commutative_2015} -- taking $\chi$ as a $\GSp(Q)$-weight and $\lambda$ a 1-parameter subgroup of $\Sp(Q)$ we obtain a complex on $\cY$.
  We only need to know one fact: the resulting exact sequences replace a bundle $\Schpur{\alpha', k'} Q$ with bundles $\Schpur{\alpha'', k''}Q$ where $k''\leq k'$. This follows from \cite[Lemma 11.2.1]{spenko_non-commutative_2015} and the fact that $\Delta$ acts with positive weights on $\Hom(V,Q)$. This proves that the set of weights of $P_{\alpha,k}$ contains one copy of $(\alpha,k)$ and the other weights satisfy $k^{\pr} < k$.

  Now we apply Serre duality for the category $\Br(\cY)$ (Proposition \ref{thm:SerreDuality2}), recalling that the Serre functor is given (up to a shift) by tensoring by $(\det Q)^{-v}=\Schpur{\textbf{0}, -2qv} Q$. By Lemma \ref{lem:WeightsByHom}, 
  if $(\alpha^{\pr},k^{\pr})$ is a weight of $P_{\alpha,k}$ then
  $$\Hom_\cY(P_{\alpha,k},P_{\alpha^{\pr},k^{\pr}}) \neq 0
  \; \implies \;   \Hom_\cY(P_{\alpha^{\pr},k^{\pr}},P_{\alpha,k-2qv})\neq 0
  $$
  so $(\alpha, k-2qv)$ is a weight of $P_{\alpha', k'}$. Hence (by the first part of the proof) 
  we must have $k-2qv\leq k'$.
\end{proof}

In fact the Serre duality argument proves slightly more: if $(\alpha', k')$ is a weight of $P_{\alpha, k}$ with $k'$ minimal, then $k'$ is exactly $k-2qv$, and $\alpha'=\alpha$. 

\begin{lem}\label{lem:RestrictToZero}
  The objects $P_{\alpha, k}\in \Br(\cY)$ all restrict to $0$ in $\Br(\cY^{ss})$. 
\end{lem}
\begin{proof}
  Obviously the object $\cO_0\otimes \Schpur{\alpha, k}Q$ is supported at the origin in $\widetilde{\Pf_q}$. Now observe that the projection $D^b(\cY)\to \Br(\cY)$ is linear over $\widetilde{\Pf_q}$. 
\end{proof}

\begin{lem}
  \label{thm:BundlesGenerateProjectiveWindow}
  For any $(\alpha,k)\in Y_{q,s} \tms \Z$, the vector bundle $\Schpur{\alpha, k}Q$ on $\cY^{ss}$ has a finite resolution by vector bundles associated to the set $Y_{q,s}\tms [-qv, qv)$.
\end{lem}
\begin{proof}
  Using Lemma \ref{lem:minimalmodels}, the claim is equivalent to $\Schpur{\alpha,k}Q$ being in the image of the restriction functor $F : \Br(\cY)_{[-qv,qv)} \to D^{b}(\cY^{\ss})$.

  We argue by induction on $|k|$.
  The case when $|k| < qv$ and $k = -qv$ are obvious.
  If $k > qv$, we consider the following exact triangle in $D^b(\cY)$:
  \[
    \cE \to \Schpur{\alpha,k}Q \to P_{\alpha,k}.
  \]
  The weights of $\cE$ are precisely the weights of $P_{\alpha,k}$, minus $(k,\alpha)$.
  Hence any $(\alpha^{\pr}, k^{\pr})$ appearing as a weight of $\cE$ are such that $(\alpha^{\pr}, k^{\pr}) \in Y_{q,s} \times [k-2qv, k)$, and hence $\cE$ admits a resolution in terms of vector bundles $\Schpur{\alpha^{\pr},k^{\pr}}Q$.
  By induction these vector bundles are in the in the image of $F$, and since $F(P_{k,\alpha}) = 0$, it follows that $\Schpur{\alpha,k}Q$ is in the image of $F$.

  If $k < -qv$, we argue similarly starting from the exact triangle
  \[
    P_{\alpha,-k}^{\vee} \to \Schpur{\alpha,k}Q \to \cF
  \]
  where the weights of $\cF$ are contained in $Y_{q,s} \times (k, k + 2qv]$.
\end{proof}

\begin{prop}
  \label{prop:EssentiallySurjective}
  The restriction functor
  \[
    \Br(\cY \tms_{\C^{*}} L^\perp,\, W)_{[-qv,qv)} \to \Br(\cY^{\ss} \tms_{\C^{*}} L^\perp,\, W)
  \]
  is essentially surjective.
\end{prop}
\begin{proof}
  Take $\cE\in \Br(\cY^{\ss} \tms_{\C^{*}} L^\perp,\, W)$. By definition $\cE$ is equivalent to the image of an object in $\Br(\cY\tms_{\C^{*}} L^\perp,\, W)$, so we can assume that $\cE$ is a vector bundle built from the set $Y_{q,s}\tms \Z$. This vector bundle (although not the twisted differential) is pulled up from $\cY^{ss}$, so it has a resolution by bundles from the set $Y_{q,s}\tms [-qv, qv)$. By Proposition \ref{prop:Fullyfaithful} this set of bundles has no higher Ext groups between them, so we can use the perturbation process of \cite[Lemma 4.10]{ADS_pfaffian_2015}, \cite[Lemma 3.6]{segal_equivalence_2011}  to conclude that $\cE$ is equivalent to a matrix factorization built from this set. Then such a matrix factorization is automatically the restriction of a matrix factorization in $\Br(\cY \tms_{\C^{*}} L^\perp,\, W)_{[-qv,qv)}$ (because the twisted differential must extend).
\end{proof}

\subsection{The homological projective duality statement}\label{sec:HPD}

In this section we produce a Lefschetz decomposition of $\Br(\cY^{\ss})$, and show that $\Br(\cX^{\ss})$ is equivalent to the tautological HP dual defined in Section \ref{sec:HPDualityBackground}.

\begin{lem}\noindent
  \label{thm:BundlesExceptionalAndSemiorthogonal}\begin{enumerate}\item
  Let $(\alpha,k) \in Y_{q,s} \times \Z$.
  The vector bundle $\Schpur{\alpha, k}Q$ on $\cY^{\ss}$ is exceptional, \emph{i.e.}\ $\REnd(\Schpur{\alpha, k}Q) = \C$.
\item   Let $(\alpha_{1},k_{1}), (\alpha_{2},k_{2}) \in Y_{q,s} \tms [-qv, qv)$, with $k_{1} > k_{2}$.
  Then:
  \[
    \RHom_{\cY^{\ss}}(\Schpur{\alpha_{1}, k_{1}}Q, \Schpur{\alpha_{2}, k_{2}}Q) = 0
  \]
\end{enumerate}
\end{lem}
\begin{proof}
  By the fully faithfulness result of Proposition \ref{prop:Fullyfaithful}, we may replace the bundles on $\cY^{\ss}$ with the corresponding bundles on $\cY$ and compute Hom spaces for these.
  The claims are then easy to verify.
\end{proof}

Let $\cA_{0} = \Br(\cY)_{[-qv,-qv+1]}$, and let $\cA_{0}(i) = \cA_{0} \otimes \cO_{\P(\wedge^{2}V^{\vee})}(i)$.

\begin{lem}\label{thm:LefschetzDecompositionExists}
  There is a Lefschetz decomposition
  \[
    \Br(\cY^{\ss}) = \langle\cA_{0}, \cA_{0}(1), \ldots, \cA_{0}(qv-1)\rangle.
  \]
\end{lem}
\begin{proof}
  Applying Theorem \ref{thm:ProjectiveWindows} in the special case $L^\perp=0$ gives us  $\Br(\cY^{\ss}) \cong \Br(\cY)_{[-qv, qv)}$. Using this equivalence, the claim follows from Lemmas \ref{thm:BundlesGenerateProjectiveWindow} and \ref{thm:BundlesExceptionalAndSemiorthogonal}.
\end{proof}

\begin{lem} The subcategory $\Br(\cY^{ss})\subset D^b(\cY^{ss})$ is admissible and saturated.
\end{lem}
\begin{proof}
  Since $\Br(\cY^{\ss})$ admits a full exceptional collection, it has a strong generator given by the direct sum of the objects in the exceptional collection. Hence by the main result of \cite{bondal_generators_2003} every functor $\Br(\cY^{\ss}) \to D^{b}(\C)^{op}$ is representable. Since $\Br(\cY^{\ss})$ has a Serre functor (Prop.~\ref{thm:SerreOnTheProjectiveNCR}) it folows that every functor $\Br(\cY^{\ss})^{\mathrm{op}} \to D^{b}(\C)^{\mathrm{op}}$ is representable. Hence $\Br(\cY^{\ss})$ is saturated.

  As $D^{b}(\cY^{\ss})$ is ext finite, $\Br(\cY^{\ss}) \subseteq D^{b}(\cY^{\ss})$ is admissible.
\end{proof}

We can now apply the theory described in Section \ref{sec:HPDualityBackground} to the stack $\cS := \cY^{ss}$ with the subcategory $\cW:=\Br(\cY^{ss})$. The construction involves the total space of a line bundle $\cL^\vee$, in our example this is the stack
$$\cL^\vee = \Stack{ \Hom(V,Q)^{\ss}\times \C }{\GSp(Q) } $$
where $\GSp(Q)$ acts on $\C$ as the character $\langle\beta_Q\rangle^{-1}$. Then we consider the stack
$$ \cT^{ss} := \Stack{\Wedge^2 V \times \Hom(V,Q)^{\ss} \times \C  }{(\GSp(Q) \tms \C^{*})}$$
and its open substack:
\[
\cT^{ss, *} := \Stack{(\Wedge^2 V \setminus 0) \times \Hom(V,Q)^{\ss} \times \C  }{(\GSp(Q) \tms \C^{*})}
\]
We're adding the ``ss'' superscript to emphasise that these stacks are constructed from $\cY^{ss}$; we'll shortly want to consider the analogous stacks constructed from $\cY$ instead, and we'll denote those by $\cT$ and $\cT^*$. 

The stack $\cT^{ss, *}$  is a family of copies of $\cL^\vee$ over the base $\P(\Wedge^2 V)$, and it carries a superpotential $W$ and R-charge as in Section \ref{sec:HPDualityBackground}. The `tautological' HP dual of $\cW$ is the subcategory
\[
\cW^{\vee} \subset D^{b}(\cT^{ss,*}, W)
\]
consisting of those objects $\cE$ such that for each point $p \in \P(\Wedge^2 V)$, the object $\cE|_{p \times \cY^{\ss} \times 0 }$ lies in the subcategory $\cA_0\subset D^b(\cY^{ss})$. 

Also recall that for each subspace $L\subset \Wedge^2 V^\vee$ we write $\cT^{ss,*}_{L^\perp}$ for the restriction of $\cT^{ss,*}$ to $\P L^\perp$, and we define a base-changed category $\cW^\vee_{L^\perp}\subset D^b(\cT^{ss,*}_{L^\perp}, W)$ by the same point-wise condition.

\begin{prop}
\label{thm:GITHPDualEqualsTautologicalHPDual}
For any linear subspace $L \subseteq \wedge^{2}V^{\vee}$ there is an equivalence:
$$\Br( \cY\tms_{\C^*} (L^{\perp} \setminus 0) ,\, W) \;\cong\; \cW^{\vee}_{L^{\perp}}$$
\end{prop}

The proof will follow after some discussion and a couple of lemmas.
\pgap

Proposition \ref{thm:GITHPDualEqualsTautologicalHPDual} is closely related to Proposition \ref{prop:EasyWindow} and Theorem \ref{thm:ProjectiveWindows} but the role of $L^\perp$ is a bit different. Take those two results and specialize them to the case when $L^\perp$ is one-dimensional, so $L=\C$ with a weight $-1$ action of $\GSp(Q)$. Then we can compare the categories:
$$\Br(\cY\tms_{\C^*} \C^*, W) \aand \Br(\cY^{ss}\tms_{\C^*} \C, W) $$
Note that $\cY^{ss}\tms_{\C^*} \C$ is exactly the line-bundle $\cL^\vee$; also $\cY\tms_{\C^*} \C^*$ is just $\widetilde{\cY}$. We compare the two categories by lifting them to `window' subcategories defined on the ambient stack $\cY\tms_{\C^*} \C$. They are equivalent to the subcategories
$$\Br(\cY\tms_{\C^*} \C, W)_{[-qv, -qv+2)} \aand \Br(\cY\tms_{\C^*} \C, W)_{[-qv, qv)} $$
respectively, and the first one is a subcategory of the second one, specified by a stricter grade-restriction rule on the $\Delta$ weights. This gives us an embedding
$$\Br(\cY\tms_{\C^*} \C^*, W)\; \into \; \Br(\cY^{ss}\tms_{\C^*} \C, W)$$
and the image is easy to describe: it is the subcategory of objects $\cE$ such that the restriction $\cE|_{\cY^{ss}}$ lands in the subcategory $\cA_0\subset D^b(\cY^{ss})$. 

This statement holds for any choice of line $\C\into \Wedge^2 V$,  the line just determines the superpotential $W$. Now we want to put this construction into a family where we allow the line to vary through a linear subspace $\P L^\perp\subset \P (\Wedge^2 V)$. This means we consider the stack:
$$\cT^*_{L^\perp} = \Stack{(L^\perp \setminus 0) \times \Hom(V,Q) \times \C  }{(\GSp(Q) \tms \C^{*})}$$
This is the analogue of $\cT^{ss,*}_{L^\perp}$ for the stack $\cY$ instead of $\cY^{ss}$, and the superpotential and R-charge obviously extend. It contains $\cT^{ss,*}_{L^\perp}$ as an open substack, and it also contains the following open substack:
\al{
\Stack{(L^\perp \setminus 0) \times \Hom(V,Q) \times \C^*  }{(\GSp(Q) \tms \C^{*})} & \cong 
\Stack{(L^\perp \setminus 0) \times \Hom(V,Q) }{\GSp(Q) }\\ 
&\cong  \cY \tms_{\C^*} (L^\perp\setminus 0) 
}
To get Proposition \ref{thm:GITHPDualEqualsTautologicalHPDual} we're going to lift $\Br(\cY \tms_{\C^*} (L^\perp\setminus 0) , W)$ to a `window' subcategory in $D^b(\cT^*_{L^\perp}, W)$, restrict it to $D^b(\cT^{ss,*}_{\perp}, W)$, and prove that the image is exactly $\cW^\vee_{L^\perp}$. 

\begin{rem}
We're going to present the proof of Proposition \ref{thm:GITHPDualEqualsTautologicalHPDual} (including Lemmas \ref{lem:BrOnOpen} and \ref{lem:pointwise}) for the special case $L^\perp=\Wedge^2 V$. This is purely to unclutter our notation, \emph{e.g.}~so we can just write $\cT^*$ instead of $\cT^*_{L^\perp}$. The proof for a general $L^\perp$ is identical.
\end{rem}

Let
$$\cT = \Stack{\Wedge^2 V \times \Hom(V,Q) \times \C  }{(\GSp(Q) \tms \C^{*})}$$
with the evident superpotential and R-charge extended from $\cT^*$, this is the analogue of $\cT^{ss}$ for the stack $\cY$ instead of $\cY^{ss}$. This stack $\cT$ contains all the stacks we're currently interested in as open substacks, as summarized in the following diagram:
\begin{center}
\begin{tikzcd}
\cT^{ss} \arrow[r, hook]& \cT \arrow[r, hookleftarrow] & \cY\tms_{\C^*} \Wedge^2 V  \\
\cT^{ss, *} \arrow[r, hook] \arrow[u, hook] & \cT^* \arrow[u, hook] \arrow[r, hookleftarrow]& \cY\tms_{\C^*} (\Wedge^2 V\setminus 0) \arrow[u, hook] 
\end{tikzcd}
\end{center}
The top row all live over $[\Wedge^2 V/ \C^*]$, and the bottom row are their restrictions to $\P(\Wedge^2 V)$. We define a subcategory
$$\Br(\cT, W) \subset D^b(\cT, W)$$
by our usual grade-restriction-rule at the origin, \emph{i.e.}~using the set $Y_{q,s}$ of irreps of $\Sp(Q)$. For any $\C^*_R$-invariant open substack  $\cU \subset \cT$ we as usual define the `brane' subcategory 
$$\Br(\cU, W) \subset D^b(\cU, W)$$
to be the subcategory generated by the image of $\Br(\cT, W)$ under restriction.

The following lemma shows that $\Br(\cY \times_{\C^*}\Wedge^2 V,\, W)$ (under our old definition) agrees with the new definition as the image of $\Br(\cT, W)$. 
\begin{lem}\label{lem:BrOnOpen}
An object $\cE\in D^b(\cY \tms_{\C^*}\Wedge^2 V,\, W)$ is the restriction of an object in $\Br(\cT, W)$ if and only if $\cE|_0$ contains only irreps from the set $Y_{q,s}$.
\end{lem}
\begin{proof}
Take $\cF\in \Br(\cT,W)$ and write in the form provided by Lemma \ref{lem:minimalmodels}, then its restriction to $\cY \tms_{\C^*}\Wedge^2 V$ obviously satisfies the grade-restriction-rule at the origin. For the other direction observe that that there is a map
$$\Stack{ \Wedge^2 V\times \C}{\C^*} \To \Wedge^2 V $$
given by multiplying the two factors, this induces a map:
$$\mu: \cT \To \cY \tms_{\C^*}\Wedge^2 V$$
Take $\cE\in D^b(\cY \tms_{\C^*}\Wedge^2 V,\, W)$ satisfying the grade-restriction-rule, then $\mu^*\cE$ lies in $\Br(\cT, W)$ and it restricts to give $\cE$. 
\end{proof}

The category $\cW^\vee$ is defined by a grade-restriction-rule `point-wise' on the base $\P (\Wedge^2 V)$.
The following lemma shows that the condition of lying in the categories $\Br(\cY \tms_{\C^{*}} (\Wedge^{2} V \setminus 0), W)$ and $\Br(\cT^{ss,*},W)$ is also a point-wise condition over $\P(\Wedge^{2}V)$.
\begin{lem}\noindent \label{lem:pointwise}
 \begin{enumerate}
\item An object $\cF\in D^b(\cY\tms_{\C^*} (\Wedge^2 V \setminus 0), W)$ lies in the subcategory 
$$\Br(\cY\tms_{\C^*} (\Wedge^2 V \setminus 0), W)$$
  if and only if it obeys the following condition: for every point $p\in \P(\Wedge^2 V)$, the restriction $\cF|_{\widetilde{\cY}\times p} \in D^b(\widetilde{\cY}, W_p)$ lies in the subcategory $\Br(\widetilde{\cY}, W_p)$.

\item An object $\cE \in D^b(\cT^{ss,*}, W)$ lies in the subcategory $\Br(\cT^{ss,*}, W)$ if and only if it obeys the following condition: for every point $p\in \P(\Wedge^2 V)$, the restriction $\cE|_{p\times \cY^{ss} \times 0} \in D^b(\cY^{ss})$ lies in the subcategory $\Br(\cY^{ss})$. 
\end{enumerate}

\end{lem}

Another way to express the condition in part (1) is to say that for every $p$ the restriction $\cF|_{0\times p}$ only contains $\Sp(Q)$ irreps from the set $Y_{q,s}$.

\begin{proof}
\begin{enumerate}\item If $\cF$ is the restriction of an object in $\Br(\cY\tms_{\C^*} \Wedge^2 V, W)$ then we can write it in the form provided by Lemma \ref{lem:minimalmodels}, so it obviously obeys the stated condition for all $p\in \P(\Wedge^2 V)$. The condition is preserved under taking mapping cones so it follows that it holds for an arbitrary object of the category $\Br(\cY\tms_{\C^*} (\Wedge^2 V \setminus 0), W)$.

Conversely, suppose $\cF$ is an object that obeys the stated condition for all $p$. The subcategory 
$$\Br(\cY\tms_{\C^*} \Wedge^2 V, W) \; \subset D^b(\cY\tms_{\C^*} \Wedge^2 V, W)$$
is right admissible, by the analogue of Lemma \ref{lem:BranesToModules2}, and the projection functor is linear over $\Wedge^2 V$. It follows that the subcategory 
$$\Br(\cY\tms_{\C^*}(\Wedge^2 V\setminus 0), W) \; \subset D^b(\cY\tms_{\C^*} (\Wedge^2 V\setminus 0), W)$$
 is also right-admissible. Write $\cF^\diamond$ for the projection of $\cF$ into the orthogonal subcategory; we need to prove that $\cF^\diamond\simeq0$. 

 If we restrict $\cF$ to a fibre $\widetilde{\cY}\times p$ and then project into the orthogonal to $\Br(\widetilde{\cY}, W_p)$ then we get zero, by assumption. But these two operations commute, so $\cF^\diamond$ restricts to zero on each fibre.
 Then $\hom(\cF^\diamond, \cF^\diamond)$ is acyclic on each fibre, so it is acyclic, and $\cF^\diamond$ is contractible.

\item This is the same argument as in part (1), apart from the following two points. Firstly the subcategory $\Br(\cT^{ss,*}, W)$ is admissible because the subcategory $\Br(\cT, W)\subset D^b(\cT, W)$ is admissible (again by an analogue of Lemma \ref{lem:BranesToModules2}) and the projection functor is linear over $\Wedge^2 V \times \Wedge^2 V^\vee$. Secondly the condition on fibres only shows that the projected object $\cE^\diamond$ is zero at points lying in $p\times \cY^{ss}\times 0$ for some $p$, but since the support of $\hom(\cE^\diamond, \cE^\diamond)$ must be closed and invariant under $\GSp(Q)\times \C^*\times \C^*_R$ it follows that $\cE^\diamond$ is contractible everywhere.
\end{enumerate}
\end{proof}

\begin{proof}[Proof of Proposition \ref{thm:GITHPDualEqualsTautologicalHPDual}]
Just as in Section \ref{sec:Windows}, for an interval $I\subset \Z$ let us define 
$$\Br(\cT, W)_I \; \subset \Br(\cT, W) $$
to be the full subcategory of objects $\cE$ such that the $\Delta$-weights of $h_\bullet(\cE|_0)$ all lie in the interval $I$. We also define
$$\Br(\cT^*, W)_I \; \subset \Br(\cT^*, W) $$
as (the subcategory generated by) the image of $\Br(\cT, W)_I$ under restriction.\footnote{The subcategory $\Br(\cT, W)_I$ is preserved under tensoring by characters of $\C^*$, so we can regard it as a category over $[\Wedge^2 V / \C^*]$ and its restriction to $\P(\Wedge^2 V)$ is well-behaved, in particular it depends on the choice of $I$.}

The proofs of  Proposition \ref{prop:EasyWindow} and Theorem \ref{thm:ProjectiveWindows} adapt trivially to prove that restriction induces equivalences
$$\Br(\cT, W)_{[-qv, qv)}\; \isoto\; \Br(\cT^{ss}, W) $$
and:
$$\Br(\cT, W)_{[-qv, -qv+2)}\; \isoto\; \Br(\cY\tms_{\C^*}\Wedge^2 V, W)$$
These equivalences are linear over $[\Wedge^2 V / \C^*]$, so if we restrict to the complement of the origin in $\Wedge^2 V$ we get  equivalences
\begin{equation}\label{eq:Twindow1} \Br(\cT^*, W)_{[-qv, qv)}\; \isoto\; \Br(\cT^{ss, *}, W) \end{equation}
and:
\begin{equation}\label{eq:Twindow2}\Br(\cT^*, W)_{[-qv, -qv+2)}\; \isoto\; \Br(\cY\tms_{\C^*}(\Wedge^2 V\setminus 0), W)\end{equation}

We can characterize the subcategory $\Br(\cT^*, W)_{[-qv, -qv+2)}\subset D^b(\cT^*, W)$ in the manner of Lemma \ref{lem:pointwise}. The equivalence \eqref{eq:Twindow2} is part of an embedding
\begin{equation*}D^b(\cY\tms_{\C^*}(\Wedge^2 V\setminus 0), W) \; \into \; D^b(\cT^*, W) \end{equation*}
whose image consists of all objects $\cF$ such for any $p$ the restriction $\cF_{p\times 0\times 0}$ has $\Delta$-weights lying in the interval $[-qv, -qv+2)$. Over any point $p$ this embedding restricts to the embedding
$$D^b(\widetilde{\cY}, W_p) \into D^b(\Stack{\Hom(V, Q)\times \C}{\GSp(Q)}, W_p) $$
which maps:
$$\Br(\widetilde{\cY}, W_p)\; \isoto\; \Br\!\left(\Stack{\Hom(V, Q)\times \C}{\GSp(Q)}, W_p\right)_{[-qv, -qv+2)}$$
Using part (1) of Lemma \ref{lem:pointwise}, this proves that an object $\cE\in D^b(\cT^*, W)$  lies in the subcategory $\Br(\cT^*, W)_{[-qv, -qv+2)}$ if and only if for any $p$, the restriction $\cE|_{p\times 0\times 0}$ contains only $\GSp(Q)$-irreps from the set $Y_{q,s}\times [-qv, -qv+2)$, which is if and only if $\cE|_{p\times \cY\times 0} \in \Br(\cY)_{[-qv, -qv+2)}$. 
 
Using \eqref{eq:Twindow1}, restricting to $\cT^{ss, *}$ gives an equivalence between $\Br(\cT^*, W)_{[-qv, -qv+2)}$ and some subcategory of $\Br(\cT^{ss, *}, W)$. Our claim is that this subcategory is exactly the tautological HP-dual category $\cW^\vee$; combining this claim with \eqref{eq:Twindow2} proves the proposition.

Over a point $p$ we can restrict from $\cT^*$ to the locus $p\times \cY\times 0$, and we can restrict from $\cT^{ss, *}$ to the locus $p\times \cY^{ss}\times 0$. On these loci \eqref{eq:Twindow1} restricts to the equivalence
$$\Br(\cY)_{[-qv, qv)} \isoto \Br(\cY^{ss})$$
and the image of $\Br(\cY)_{[-qv,-qv+2)}$ here is the category $\cA_0$, by definition. It follows immediately that the image of $\Br(\cT^*, W)_{[-qv, -qv+2)}$ inside $D^b(\cT^{ss, *}, W)$ lies inside $\cW^\vee$.

For essential-surjectivity, let $\cE\in \cW^\vee$. By part (2) of Lemma \ref{lem:pointwise} $\cE$ lies in the category $\Br(\cT^{ss, *}, W)$, so under \eqref{eq:Twindow1} it lifts to a unique object $\hat{\cE}\in \Br(\cT^*, W)_{[-qv, qv)}$. For any $p$ the restriction $\hat{\cE}|_{p\times \cY \times 0}$ lies in $\Br(\cY)_{[-qv,qv)}$.
Then since $\hat\cE|_{p\times \cY^{ss} \times 0} \in \cA_{0}$, by definition of $\cA_{0}$ we must have $\hat \cE|_{p\times \cY \times 0} \in \Br(\cY)_{[-qv, -qv+2)}$, and from our characterization above this shows that $\hat{\cE}$ lies in $\Br(\cT^*, W)_{[-qv, -qv+2)}$.
\end{proof}

So for any $L$, we have equivalences
$$ \cW^\vee_{L^\perp} \; \cong \;  \Br( \cY\tms_{\C^*} (L^{\perp} \setminus 0) ,\, W)\;\cong\; \Br(\cX^{\ss}\tms_{\C^*} L, \, W')$$
by Proposition \ref{thm:GITHPDualEqualsTautologicalHPDual} and Corollary \ref{thm:HoriDualityWithLOverSemistableLocus} respectively. If we apply the main theorem of HP duality (Theorem \ref{thm:generalHPDTheorem}), and observe that $\dim \Wedge^2 V^\vee = sv + qv$, we immediately get the following result.

\begin{thm}
  \label{thm:HPDualityBranes}
  Let $L \subset \Wedge^{2} V^{\vee}$ be a linear subspace, let $l=\dim L$ and let $l'=\dim L^\perp$. There are semiorthogonal decompositions
  \[
    \Br( \cY^{\ss}\tms_{\C^*} L^\perp,\, W) = \Big\langle \cC_{L},\; \cA_{l'}(l'), \;\ldots,\; \cA_{qv-1}(qv-1) \Big\rangle 
  \]
  and
  \[
    \Br(\cX^{\ss} \tms_{\C^*} L,\, W) = \Big\langle \cB_{1-sv}(1-sv),\;\ldots, \;\cB_{-l}(-l),\; \cC_{L} \Big\rangle
  \]
  for some category $\cC_L$. 
\end{thm}

This statement is an upgrade of Corollary \ref{cor:HPDlite}. If $l \geq sv$ (so $ l'\leq qv$) then the second semi-orthogonal decomposition contains only a single piece, so the category  $\Br(\cX^{\ss} \tms_{\C^*} L,\, W)=\cC_L$ embeds as an admissible subcategory of  $\Br( \cY^{\ss}\tms_{\C^*} L^\perp,\, W)$. If the inequalities are reversed then the embedding goes the other way, and if $l=sv$ (so $l'=qv$) then the categories are equivalent.

Combining this theorem with Lemma \ref{thm:BaseChangeForModuleCategoriesIsBaseChangeFactorisationCategories}, we arrive at the claim that $(\Pf_q,B)$ is HP dual to $(\Pf_s,A)$:
\begin{thm}
  \label{thm:HPDualityAlgebrasConcrete}
  If $L \subset \wedge^{2} V^{\vee}$ is a linear subspace such that $\cX^{\ss}|_{\P L^{\perp}}$ and $\cY^{\ss}|_{\P L}$ have the expected dimensions, then we get semiorthogonal decompositions
  \[
    D^{b}(\Pf_{q} \cap \P L,\; B|_{\P L}) = \Big\langle \cC_{L},\; \cA_{l'}(l'), \;\ldots,\; \cA_{qv-1}(qv-1) \Big\rangle 
  \]
  and
  \[
    D^{b}(\Pf_{s} \cap \P L^\perp,\; A|_{\P L^\perp}) = \Big\langle \cB_{1-sv}(1-sv), \;\ldots, \;\cB_{l}(l),\; \cC_{L} \Big\rangle.
  \]
\end{thm}

\begin{rem}
  The category $\Br(\cX^{\ss})$ is equipped with two natural Lefschetz decompositions: The first via the natural variant of Lemma \ref{thm:LefschetzDecompositionExists}, and the second obtained from the fact that it is the HP dual of $\Br(\cY^{\ss})$.
  We don't know if these two decompositions are the same, but this matching up is not needed for the relation of HP duality.
\end{rem}

\section{The case when $\dim V$ is even}\label{sec:EvenCase}

In this section - as promised in the introduction - we briefly discuss the case when $v=\dim V$ is even instead of odd. Many of our previous results continue to hold, but a few crucial ones fail. We explain which parts work or do not work, and in particular explain how the `window for the easy phase' (Theorem \ref{thm:ProjectiveWindows}) must be modified.
\pgap

We keep all of our notation from before: the vector spaces $V, S, Q$, the stacks $\cX$ and $\cY$, etc. We now set the dimension $v$ of $V$ to be:
$$v = 2s + 2q $$
Then $\Pf_s$ is still the classical projective dual to $\Pf_q$. We continue to use $Y_{s,q}$ for the set of Young diagrams of height at most $s$ and width at most $q$, and define the subcategory
$$\Br(\cX\tms L,\, \dualW)\;\subset\;D^b(\cX\tms L,\, \dualW)$$
in exactly the same way as before (and the same is true on the $\cY$ side). The crucial thing that changes is that, although $\Br(\cX)$ is still a non-commutative resolution of $\widetilde{\Pf_s}$ by \cite{spenko_non-commutative_2015}, \emph{it is no longer crepant} in general. 
This means that we still have an equivalence between $\Br(\cX)$ and the derived category of an algebra $A$ (defined exactly as before), but $A$ is no longer Cohen--Macaulay.\footnote{In Van den Bergh's definition of a non-commutative crepant resolution `crepancy' is the requirement that $A$ is Cohen--Macaulay \cite{van_den_bergh_non-commutative_2004}. Assuming the singularity is Gorenstein (which is part of Van den Bergh's definition) this is equivalent to the dualizing complex of $A$ being trivial.}  All other statements of Section \ref{sec:Br(X)} continue to hold.

Section \ref{sec:Kernel} in which we define and study the kernel is completely unchanged, since $V$ plays essentially no role here. Section \ref{sec:generic}, in which we prove that $\Phi$ is generically an equivalence, also continues to work without modification (indeed the case when $v$ is even was considered in our earlier papers on which this section is based). 

However, our proof that $\Phi$ is an equivalence everywhere fails. We can still define our objects $\cP_\delta$ (Section \ref{sec:DualObjects}) and their endomorphism dga $A'$, but since $A$ is not Cohen--Macaulay we cannot prove Proposition \ref{prop:AandA'areEquivalent}. In fact the first failure is at Lemma \ref{lem:DegreesOfA'}, which means we cannot even prove that $A'$ is an algebra rather than a dga.

The final steps of Section \ref{sec:completingProof} continue to work, and the end result is that we have a pair of adjoint functors 
$$\begin{tikzcd}[column sep=40pt] 
  \Br(\cX\tms_{\C^*} L, \, \dualW)  \arrow[yshift=.4ex]{r}  &
  \Br(\cY\tms_{\C^*} L^\perp, \, W) \ar[yshift=-.4ex]{l}
\end{tikzcd} $$
but we only know that they are mututally inverse once we restrict to the open set $\OS\subset \Wedge^2 V$ of bivectors having rank $\geq 2s$.
\pgap

Now we move on to Section \ref{sec:Projective}. Base-changing the above adjunction to the complement of the origin in $\Wedge^2 V$, we get an adjunction between $\Br(\cX^{ss}\tms_{\C^*} L, \, \dualW)$ and $\Br(\cY\tms_{\C^*} (L^\perp\setminus 0), \, W)$, which we know to be an equivalence over the open set $\P \OS\subset \P(\Wedge^2 V)$.  The results of Section \ref{sec:NCRslices} are unaffected, so if $L$ is generic then the former category is equivalent to the derived category of the sheaf of algebras $A|_{\P L^\perp}$ on the slice $\Pf_s\cap \P L^\perp$. Since the intersection of $\Pf_s$ with $\P \OS$ is exactly the smooth locus $\Pf_s^{sm}$, this shows that over the smooth locus we have
$$\Br(\cY\tms_{\C^*} (\P\OS\cap L^\perp),\, W) \; \cong \; D^b(\Pf_s^{sm}, A|_{\P L^\perp})$$
which is simply $D^b(\Pf_s^{sm})$.  But once we include the singular locus we don't know whether
these two categories are the same.

The `windows' results of Section \ref{sec:Windows}, which allow us to compare the categories
$$\Br(\cY\tms_{\C^*} (L^\perp\setminus 0), \, W) \aand \Br(\cY^{ss}\tms_{\C^*} L^\perp, \, W)$$
continue to hold, but in a modified form. The window for the `difficult phase' (Proposition \ref{prop:EasyWindow}) needs no adjusting, but we must modify the window for the `easy phase' (Theorem \ref{thm:ProjectiveWindows}) as we now describe.

Recall that the problem is essentially to lift the category $\Br(\cY^{ss})$ to an equivalent subcategory in $\Br(\cY)$. When $v$ was odd, we did this using `prism' in the set of weights of $\GSp(Q)$, restricting the weights of $\Sp(Q)$ and the diagonal 1-parameter subgroup $\Delta$ separately. However, one sees already in the case $q=1$ that this does not work for $v$ even, because in this case Kuznetsov's Lefschetz decomposition of $D^b(\Gr(V, 2))$ is not rectangular \cite{KuznetsovECs}. So our set of allowed weights must have a slightly more complicated shape.

Recall that irreps of $\GSp(Q)$ have highest weights $(\delta, k)$ for $\delta$ a dominant weight of $\Sp(Q)$ and $k$ a weight of $\Delta$ such that $\sum_i \delta_i + k \cong 0 $ mod 2.  We define a subset
$$\Omega \subset Y_{q,s}\tms \Z $$
as the set of weights $(\delta, k)$ such that either
\begin{itemize}
\item $k \in \big[ {-qv}, (q-1)v\big)$, or
\item $k\in  \big[(q-1)v, qv\big)$ and $\delta \in Y_{q,s-1}$.
\end{itemize}
Then we define a corresponding full subcategory
$$\Br(\cY\tms_{\C^*} L^\perp,\, W)_{\proj} \; \subset \; D^b(\cY\tms_{\C^*} L^\perp,\, W) $$
of objects $\cE$ such that $h_\bullet(\cE|_0)$  contains only irreps from the set $\Omega$. The analogue of Theorem \ref{thm:ProjectiveWindows} for the $v$ even case is:

\begin{thm}\label{thm:ProjectiveWindowsEvenCase}
  The restriction functor
  $$\Br(\cY\tms_{\C^*} L^\perp,\, W)_{\proj}  \; \To \; \Br(\cY^{ss}\tms_{\C^*} L^\perp,\, W) $$
  is an equivalence.
\end{thm}

It takes several pages of detailed calculations to prove this theorem and we won't present them here. The interested reader should consult the addendum to this paper \cite{rennemo_segal_addendum}. 

Applying Theorem \ref{thm:ProjectiveWindowsEvenCase} with $L^{\perp} = 0$ and arguing as in Lemma \ref{thm:LefschetzDecompositionExists}, we see that $\Br(\cY^{\ss})$ has a Lefschetz decomposition
\begin{equation}
  \label{eqn:evenLefschetzDecomposition}
  \Br(\cY^{\ss}) = \Big\langle \cA,\, \cA(1),\, \ldots,\, \cA(qv-\onehalf v),\, \cA^{\prime}( qv +\onehalf v), \,\ldots,\, \cA^{\prime}(qv) \Big\rangle
\end{equation}
where 
$$\cA = \big\langle \Schpur{\delta,k}Q\,\mid\, \delta\in Y_{q,s}, \, k\in \{{-qv}, {-qv}+1\}\big \rangle $$
and:
$$ \cA^{\prime} = \big\langle \Schpur{\delta,k}Q \,\mid\,\delta\in Y_{q,s-1}, \, k\in \{{-qv}, {-qv}+1\}\big\rangle$$
Proposition \ref{thm:GITHPDualEqualsTautologicalHPDual} holds with the same proof, \emph{i.e.}~$\Br(\cY^{ss}\tms_{\C^*} (\Wedge^{2} V \setminus 0),\, W)$ is the HP dual of $\Br(\cY^{\ss})$.

Thus we get relations between $\Br(\cY^{ss}\tms_{\C^*} (L^\perp \setminus 0),\, W)$ and $\Br(\cY^{\ss}|_{\P L})$, which can also be seen concretely in terms of the window categories.
Recall that the window for the `difficult phase' is obtained by restricting the $\Delta$ weights to the interval $[-qv, 2l'-qv)$ where $l'=\dim L^\perp$. For the right ranges of $l'$ one window is contained in the other, we have
$$ \Br(\cY\tms_{\C^*} L^\perp,\, W)_{\proj}\; \subset \; \Br(\cY\tms_{\C^*} L^\perp,\, W)_{[{-qv},\, 2l'-qv)}\quad\quad \mbox{if } l'\geq qv  $$
and:
$$ \Br(\cY\tms_{\C^*} L^\perp,\, W)_{[{-qv},\, 2l'-qv)}\; \subset \; \Br(\cY\tms_{\C^*} L^\perp,\, W)_{\proj}\quad\quad \mbox{if } l'\leq (q-\onehalf )v  $$
Therefore, if $l'\geq qv$ and $L$ is generic then we get an embedding
$$D^b(\Pf_q\cap \P L, B|_{\P L}) \; \into \Br(\cY\tms_{\C^*}(L^\perp\setminus 0),\, W ) $$
and we know that the latter is a categorical resolution of $\Pf_s$. In the other direction, if $l'\geq (q-\onehalf)v$ and $L$ is generic then we get an embedding:
$$ \Br(\cY\tms_{\C^*}(L^\perp\setminus 0),\, W ) \; \into \; D^b(\Pf_q\cap \P L, B|_{\P L})  $$

A similar decomposition to (\ref{eqn:evenLefschetzDecomposition}) exists for $\Br(\cX^{\ss})$, of course.
The Lefschetz pieces of $\Br(\cX^{\ss})$ and $\Br(\cY^{\ss})$ all have full exceptional collections, and it is easy to check that the sizes of these Lefschetz pieces agree with what they would be if $\Br(\cX^{\ss})$ and $\Br(\cY^{\ss})$ were HP dual to each other.
This lends some support to the possibility that Theorem \ref{thm:HoriDualityWithL} and our other results in fact hold when $\dim V$ is even as well.

\bibliographystyle{alpha}

\bibliography{bibliography}

\end{document}